\documentclass{article}
\usepackage{graphicx} 
\usepackage{amsmath} 
\usepackage{amsfonts}
\usepackage{amsthm}
\theoremstyle{definition}

\usepackage{amssymb}
\usepackage{wrapfig}
\usepackage{pifont}
\usepackage{lmodern}  
\usepackage[none]{hyphenat} 
\usepackage{graphicx} 
\graphicspath{ {./images/} } 
\usepackage{caption}
\usepackage{subcaption}
\usepackage{booktabs}
\usepackage{parskip} 
\usepackage[export]{adjustbox} 
\usepackage{float} 
\usepackage{comment}
\usepackage{mathptmx}
\usepackage{bbm}
\usepackage[pdfencoding=auto, psdextra]{hyperref}
\usepackage[british]{babel}

\usepackage{fullpage}
\usepackage{authblk}

\usepackage{titlesec}

\usepackage{appendix}
\usepackage[svgnames]{xcolor}
\colorlet{LightGray}{gray!20}
\usepackage{minted}
\numberwithin{equation}{section}
\usepackage{graphicx} 
\usepackage{lmodern}
\usepackage{tikz}
\usetikzlibrary{arrows}
\usepackage{amsmath}
\usepackage{extarrows}
\usepackage{subcaption}

\newtheorem{prop}{Proposition} 
\newtheorem{lem}[prop]{Lemma}
\newtheorem{thm}[prop]{Theorem}
\newtheorem{cor}[prop]{Corollary}

\newtheorem{remark}[prop]{Remark}

\newcommand{\dd}{\mathrm{d}}

\begin{document}
\title{Rate-induced tipping in a coral reef ecosystem: A slow increase in fishing effort can induce reef collapse}
\author[$1$]{Irakli Antidze}
\author[$1,2$]{Brian Hennessy} 
\author[$1$]{Nikola Popovi\'c}
\author[$1$]{Zak Sattar}
\affil[$1$]{School of Mathematics and Maxwell Institute for Mathematical Sciences, University of Edinburgh, James Clerk Maxwell Building, King’s Buildings, Peter Guthrie Tait Road, Edinburgh EH9 3FD, United Kingdom}
\affil[$2$]{School of Mathematical \& Computer Sciences and Maxwell Institute for Mathematical Sciences, Heriot-Watt University, Edinburgh EH14 4AP, United Kingdom}

\maketitle

\begin{abstract}
    \noindent
Critical transitions describe sudden changes in the state of an ecosystem. In classical bifurcation theory, such transitions occur when the value of a parameter exceeds a threshold (``bifurcation") value. More recently, critical transitions which are triggered by the rate of change of a parameter were described by Wieczorek et al. [Wieczorek, S., Ashwin, P., Luke, C.M., Cox, P.M., \textit{Proceedings of the Royal Society A: Mathematical, Physical and Engineering Sciences} 467(2129), 1243--1269, 2011]. In mathematical ecology, these rate-induced transitions correspond to environmental conditions that deteriorate too rapidly for the ecosystem to adapt, resulting in population collapse (``R-tipping").

In this article, we consider the potential for rate-induced tipping due to increased anthropogenic stress in a recently proposed behavioural-demographic model for herbivorous fish, algae, and coral in a coral reef ecosystem [Gil, M.A., Baskett, M.L., Munch, S.B., Hein, A.M., \textit{Proceedings of the National Academy of Sciences} 117(41), 25580--25589, 2020]. We first show that the underlying demographic model can be reframed naturally as a singularly perturbed system with two fast variables and one slow variable in which bistability can occur in ecologically relevant parameter regimes. Then, we explore the potential for canard-type dynamics in the model, complementing numerical results with an analytical description through the lens of geometric singular perturbation theory, and we describe R-tipping as a result of an increase in the fishing effort. In particular, we show that trajectories will undergo canard-induced tipping by passage through a folded node singularity, whereas a folded focus may give rise to tipping of jump type; in both scenarios, a catastrophic collapse occurs in the populations of herbivorous fish and coral, with the population of algae experiencing a ``bloom". Alternatively, we may observe ``tracking" of a sustainable coexistence state between the three populations in the presence of a folded focus.
\end{abstract}


\section{Introduction}
\label{sec: intro}
Coral reef ecosystems are tightly coupled networks in which fish regulate competition between coral and algae. Herbivorous fish can suppress algae, helping maintain space for coral growth, while reductions in herbivory through overfishing allow algae to proliferate and inhibit coral recovery after disturbances, such as bleaching events. These coral–algae–fish feedbacks can produce tipping-point dynamics, where changes in grazing pressure increase the likelihood of transitions from coral-dominated to algae-dominated reefs, with substantial consequences for biodiversity and the function of the ecosystem \cite{hughes_1994_catastrophes,mumby_2007_thresholds}.
Mathematical modelling of coral reef ecosystems can help identify alternative stable states and thresholds that prevent coral populations from going extinct, which remains a major goal in coral reef restoration and conservation \cite{suding_2009_threshold}, especially in the face of increasing pressures due to the climate crisis.

The demographic model introduced by Gil et al. is a three-dimensional system of ordinary differential equations (ODEs) which models the population of herbivorous fish, algae, and coral in a coral reef environment. The model is based on similar models studied by Fung, Blackwood, and others \cite{blackwood_2010_the, briggs_2018_macroalgae, fung_2011_alternative}, which are in turn derived from previous work by Mumby \cite{mumby_2007_thresholds}. Gil et al. build on these models by also accounting separately for fish herbivory, taking advantage of field measurements of reef fish foraging behaviour, and incorporating these into their demographic model \cite{gil_2017_social}.

A particularly interesting aspect of the study by Gil et al. \cite{gil_2017_social} is their identification of rate-induced tipping (``R-tipping") which occurs in their demographic model. Rate-induced tipping refers to a qualitative change in the environment as a result of the rate of change of a parameter exceeding some threshold, called the critical rate. By contrast, static bifurcations considered in classical bifurcation theory result in qualitative environmental change as a consequence of a parameter exceeding some threshold value. The concept of R-tipping first emerged in relation to tipping points in climate science, but has seen an increased interest in the study of dynamical systems following its identification by Wieczorek et al. \cite{wieczorek_2010_excitability}. Subsequently, Vanselow et al. \cite{vanselow_2019_when} showed that the classical Rosenzweig-MacArthur predator-prey model can undergo a rate-induced critical transition due to a slow change in prey carrying capacity. Specifically, they considered the following singularly perturbed model, where the carrying capacity of prey is a dynamic variable that represents a changing climate:
\begin{equation}
\begin{aligned}
    \varepsilon\frac{\dd u}{\dd t}&=u(1-\phi(t) u)-\frac{uv}{1+\eta u},\\
    \frac{\dd v}{\dd t}&=\frac{uv}{1+\eta u}-v,\\
    \frac{\dd\phi}{\dd t}&=r.
    \label{Vanselow}
\end{aligned}
\end{equation}
Here, $u$ and $v$ are the prey and predator densities, respectively, with $0<\varepsilon\ll1$ a small parameter that is defined as the ratio of the death rate of predator and the growth rate of prey. The analysis of Equation~\eqref{Vanselow} in \cite{vanselow_2019_when} is based on geometric singular perturbation theory (GSPT) \cite{MTSD}, where it is shown that 
for $r=0$, that is, for a fixed prey carrying capacity, the unique coexistence equilibrium lies on a normally attracting, parabolic-type critical manifold, and that it is attracting under the reduced flow on that manifold. For $r>0$, \eqref{Vanselow} is truly three-dimensional; it is then possible to show that the critical manifold is a concave parabolic cylinder, with attracting and repelling regions separated by a non-hyperbolic (``fold") curve corresponding to the maximum thereon. Crucially, 
it is proven in \cite{vanselow_2019_when} that there exists a unique ``canard point" \cite{SZMOLYAN2001419} on that fold curve which gives rise to trajectories that are initiated on the attracting portion of the critical manifold, pass through the canard point, and follow the repelling portion over an $\mathcal{O}(1)$ time interval in the slow time $\tau=\varepsilon t$ before being repelled. These so-called canard trajectories are observed for a large class of initial conditions in a neighbourhood of the coexistence equilibrium found for $r=0$ in \eqref{Vanselow} and capture the collapse of the prey and predator populations due to a slowly changing climate. 
The concept of R-tipping has since found widespread application in ecological modelling. Thus, in \cite{Plankton-tipping}, a subset of the same set of authors show that a rise in water temperature can cause plankton blooms. The model considered is the so-called ``Truscott-Brindley" model \cite{TRUSCOTT1994981}, in which the phytoplankton population undergoes logistic growth, while the zooplankton population grows with a Holling-type $\mathrm{III}$ functional response and dies with linear mortality. Furthermore, the water temperature is the dynamic variable; again, a canard mechanism is shown to cause rate-induced tipping and, ultimately, the creation of a plankton bloom.

Our primary aim in this article is to apply tools from dynamical systems theory and, in particular, GSPT, to analyse a rescaled version of the demographic model introduced by Gil et al. \cite{gil2020}. Using a variation of the methodology in Vanselow et al. \cite{Plankton-tipping,vanselow_2019_when}, we show that rate-induced tipping can arise when the rate of change of the mortality of fish due to fishing is above a certain threshold.
Ultimately, a collapse of herbivorous fish and coral populations occurs, which results in an algal bloom. In that sense, a slow increase in the fishing effort hence makes it possible for a coral ecosystem to be initiated close to a coexistence state between fish, algae, and coral, but ultimately to evolve towards an equilibrium where only the population of algae is non-zero. We will show that this evolution can occur in biologically realistic parameter regimes and that it depends on the rate at which the fishing effort increases.

To begin, we state the demographic model introduced by Gil et al. \cite{gil2020},
\begin{equation}
    \begin{aligned}
    \frac{\dd H}{\dd T}&={\color{black}\lambda_0} s(H)AH - \mu H - fH, \\
    \frac{\dd A}{\dd T} &=r_{A}A(1-A-C)-\lambda_0 s(H)AH, \\
    \frac{\dd C}{\dd T} &=r_CC(1-A-C)-mC,
    \label{PNAS Model}
    \end{aligned}
\end{equation}
where $s(H)=d+\frac{H}{1+H}$. Here, $H$, $A$, and $C$ are dimensionless variables representing the populations of herbivorous fish, algae, and coral, respectively, where $H,\ A,\ C\geq 0$ and $H\geq 0$. 

Details on the derivation of \eqref{PNAS Model}, as well as an interpretation of the various, mostly standard, terms therein can be found in \cite{gil2020}. Here, we merely highlight the function $s(H)$, which describes the steady‑state per‑capita feeding rate of herbivorous fish. Specifically, $s(H)$ models the behavioural feedback in which fish, when at higher densities, tend to follow each other and be more willing to feed in unsheltered and hence risky open areas, with the result that each individual feeds more on average. The function $s(H)$ was obtained by calibrating a behaviour-based model to observational data collected with camera arrays \cite{gil_2017_social}.
Definitions of the model parameters in \eqref{PNAS Model}, as well as their values, are summarised in Table~\ref{param_table}. 

\begin{table}[H]
    \centering
    \begin{tabular}{|c|c|c|} 
    \hline
    \textbf{Parameter} & \textbf{Interpretation} & \textbf{Range} \\
    \hline
    
    \(\mu\) & Natural mortality rate of fish & $0.02$ yr$^{-1}$ \\
    \hline
    $m$ & Natural mortality rate of coral & $[0.008, 0.08]$ yr$^{-1}$ \\
    \hline
    \(f\) & Mortality rate due to fishing & $[0, 0.5]$ yr$^{-1}$ \\
    \hline
    \(r_A\) & Growth rate of algae & $[1, 8]$ yr$^{-1}$ \\
    \hline
    $r_C$ & Growth rate of coral & $[0.02, 0.2]$ yr$^{-1}$ \\
    \hline
    $d$ & Minimum per-capita herbivory feeding rate & $0.22$ \\
    \hline
    $\lambda_0$ & Average per-capita herbivory rate of algae by fish & $[0, 3.2]$ yr$^{-1}$ \\
    \hline
    \end{tabular}
    \caption{Parameters and their ranges for Equation~\eqref{PNAS Model} from Gil et al. \cite[Appendix]{gil2020}.}
    \label{param_table}
\end{table}

This article is organised as follows: in Section~\ref{section GPST}, we transform Equation~\eqref{PNAS Model} into a $(2,1)$ fast-slow system. We then apply standard GSPT \cite{MTSD} to show that, in certain parameter regimes, the rescaled system exhibits bistability, in the sense that there exist two local attractors: a coexistence state or coral-free equilibrium, as well as an algae-only equilibrium. We characterise these parameter regimes in terms of the fishing effort.  
Then, in Section~\ref{section R-tipping}, we consider the fishing effort to be a time-dependent dynamic variable in the style of \cite{Plankton-tipping, vanselow_2019_when}, and we show that rate-induced tipping is mediated by a canard phenomenon; in particular, we show that trajectories will typically tip due to the presence of a folded node-type singularity \cite{Canards}. 

\section{Geometric singular perturbation analysis}
\label{section GPST}
To apply GSPT to \eqref{PNAS Model}, we first perform a rescaling which introduces an explicit separation of time scales. 
From Table~\ref{param_table}, we see that the value of $r_A$ within is the largest model parameter, which implies that algae are the fastest evolving species in \eqref{PNAS Model}. 
Therefore, we introduce a time rescaling of the form $T=\frac{t}{r_{A}}$, with $t$ being a dimensionless time variable. Equation~\eqref{PNAS Model} then becomes
\begin{equation}
\begin{split}
    \frac{\dd H}{\dd t} &=  \lambda H(s(H) A - \alpha), \\
    \frac{\dd A}{\dd t} &= A(1-A-C-\lambda H s(H)), \\
    \frac{\dd C}{\dd t} &= \varepsilon C(1-\beta-A-C),\\
\end{split}
\label{fast-system}
\end{equation}
where the corresponding dimensionless parameters are given in Table~\ref{param_table_2}. (We note that the parameter $d$ is dimensionless {\it a priori}.) Of particular relevance to us is the parameter $\alpha$, which is proportional, with proportionality factor $\lambda$, to the (non-dimensionalised) total mortality rate of herbivorous fish. As we assume the natural mortality rate of fish to be fixed, while mortality due to fishing varies, we will take $\alpha$ to reflect the fishing effort in the following.

We have the following simple result regarding \eqref{fast-system}.
\begin{lem}
    The set $\mathcal{D}=\{(H,A,C)\,|\,(H,A,C)\in [0, \infty)\times [0,1]\times[0,1]\}$ is invariant under the flow of Equation~\eqref{fast-system}. 
    \label{lem: octant invariance}
\end{lem}
We restrict the dynamics of \eqref{fast-system} to $\mathcal{D}$, so that the algae and coral populations are below their carrying capacity. (We note that the population of fish $H$ does not grow logistically, in contrast to those of algae and coral, and that it is therefore not limited by a carrying capacity.)  
\begin{table}[ht]
    \centering
    \begin{tabular}{|c|c|c|c|} 
    \hline
    \textbf{Dimensionless parameter} & \textbf{Definition} & \textbf{Possible range} &\textbf{Considered range}\\
    \hline
    $\lambda$ & $\frac{\lambda_0}{r_A}$ & $[0, 3.2]$ & $[0,1]$\\
    \hline
    $\alpha$ & $\frac{\mu+f}{ \lambda_0}$ & $[0.00625, \infty)$  & $[0.1, 0.7]$\\
    \hline
    $\beta$ & $\frac{m}{r_C}$ & $[0.04, 4]$ & $[0.04, 1)$\\
    \hline
    $\varepsilon$ & $\frac{r_C}{r_A}$ & $[0.0025, 0.2]$ & $[0.0025, 0.01]$\\
    
    \hline
    \end{tabular}
    \caption{Dimensionless parameters and their ranges for Equation~\eqref{fast-system}.}
    \label{param_table_2}
\end{table}
As the growth rate of coral is significantly lower than that of algae, with $r_C\ll r_A$ implying $\varepsilon=\frac{r_C}{r_A}\ll 1$, \eqref{fast-system} is a $(2,1)$ fast-slow system in the standard form of GSPT \cite{MTSD}. The remaining dimensionless parameters $\lambda$, $\alpha$, $\beta$, and $d$ are considered to be $\mathcal{O}(1)$. 
We refer to \eqref{fast-system} as the fast system, written in terms of the fast time $t$.

\subsection{Singular geometry}

The slow dynamics are approximated by the ``critical manifold" \cite{MTSD} $S_0$, which is defined as the set of equilibria of the fast system in \eqref{fast-system}, in the limit as $\varepsilon\to0$.
We refer to the resulting system as the layer problem,
\begin{equation}
\begin{split}
    \frac{\dd H}{\dd t} &= \lambda H(s(H) A - \alpha), \\
    \frac{\dd A}{\dd t} &= A(1-A-C-\lambda H s(H)), \\
    \frac{\dd C}{\dd t} &=0,\\
\end{split}
\label{layer-problem}
\end{equation}
the equilibria of which are located on 
\begin{equation}
    S_0 = \{(H,A,C)\in \mathcal{D}\,|\,\lambda H (s(H) A-\alpha)=0,\ A(1-A-C-\lambda H s(H))=0\}.
\end{equation}
The manifold $S_0$ can be expressed as a union $S_0 = S_0^0\cup S_0^1\cup S_0^2$ of sub-manifolds, where
\begin{align}
    &S_0^0 = \big\{(H,A,C)\in S_0\,\big|\, H=0,\ A=0\big\}, \\ 
    &S_0^1 =\big\{(H,A,C)\in S_0\,\big|\, H=0,\ A=1-C\big\},\quad\text{and} \\
    &S_0^2 = \bigg\{(H,A,C)\in S_0\,\bigg|\, A = \frac{\alpha}{s(H)},\ C=1-\frac{\alpha}{s(H)}-\lambda Hs(H)\bigg\}.
    \label{eq: S_0}
\end{align}

We have the following result regarding the stability of $S_0$.
\begin{lem}
    Assume that $\alpha>\lambda d^3$, and define the function
    \begin{equation}
        Q(H;\alpha)=\lambda \big[H+s(H)(1+H)^2\big]-\frac{\alpha}{s(H)^2}.
    \end{equation}
    Then, the following statements hold.
    \begin{enumerate}
        \item The function $Q(H;\alpha)$ is monotonically increasing in $H$ and admits a unique root $H_0$.
        \item The sub-manifold $S_0^0$ is normally repelling for $0\leq C<1$.
        The point $C=1$ corresponds to the point of intersection $n_1=S_0^0\cap S_0^1$ of $S_0^0$ and $S_0^1$, which is non-hyperbolic, and is known as a transcritical singularity.
        \item For $0<A<\frac{\alpha}{d}$ \big($A>\frac{\alpha}{d}$\big), the sub-manifold $S_0^1$ is normally attracting (normally repelling), losing hyperbolicity at $A=\frac{\alpha}{d}$. Furthermore, $A=\frac{\alpha}{d}$ corresponds to the point of intersection $n_2=S_0^1\cap S_0^2$ of $S_0^1$ and $S_0^2$, which is again a transcritical singularity. 
        \item For $\lambda< 4$ and $0<H<H_0$ ($H>H_0$), the sub-manifold $S_0^2$ is normally repelling (normally attracting), with a regular fold point at $H=H_0$, denoted by $n_0=(H_0, A_0, C_0)\in S_0^2$. The curve $C(H)=1-\frac{\alpha}{s(H)}-\lambda H s(H)$ assumes its maximum at $H=H_0$ for $H\in[0, \infty)$.
    \end{enumerate}
\label{lemma: Stability of Critical Manifold}
\end{lem}
\begin{proof}
    See Appendix~\ref{sec: stability of S_0}.
\end{proof}

For future reference, we denote the attracting (repelling) portion of $S_0^2$ by  $S_{0}^{2,a}$ \big($S_{0}^{2,r}$\big).

\begin{figure}[H]
    \centering
    \begin{subfigure}[t]{0.45\textwidth}
        \centering
        \includegraphics[scale = 0.52]{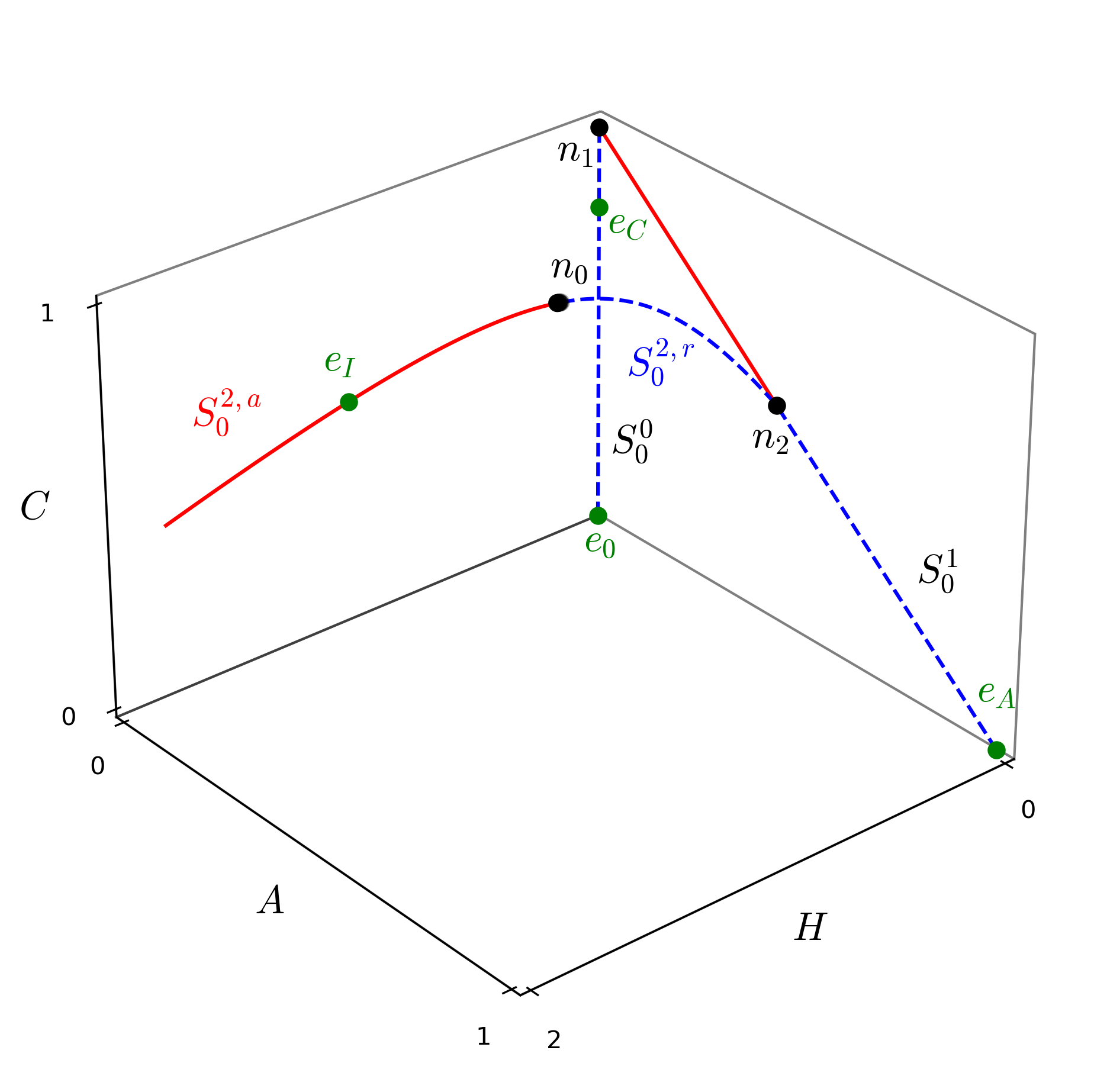}
        \subcaption{$\alpha<d$} 
    \end{subfigure}
    \hfill
    \begin{subfigure}[t]{0.45\textwidth}
        \centering
        \includegraphics[scale = 0.52]{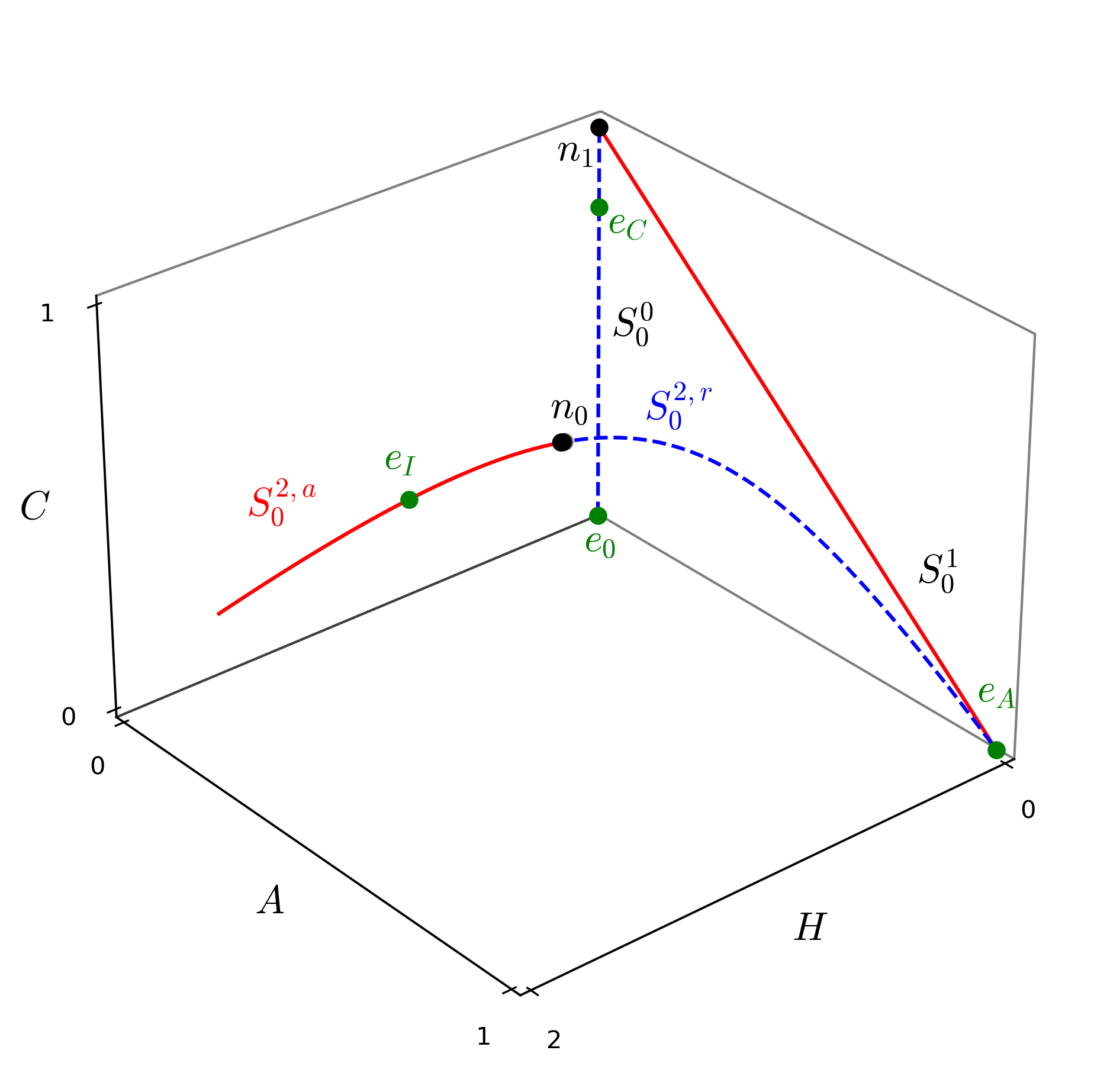}
        \subcaption{$\alpha=d$} 
    \end{subfigure}
    \hfill
    \begin{subfigure}[t]{0.45\textwidth}
        \centering
        \includegraphics[scale = 0.52]{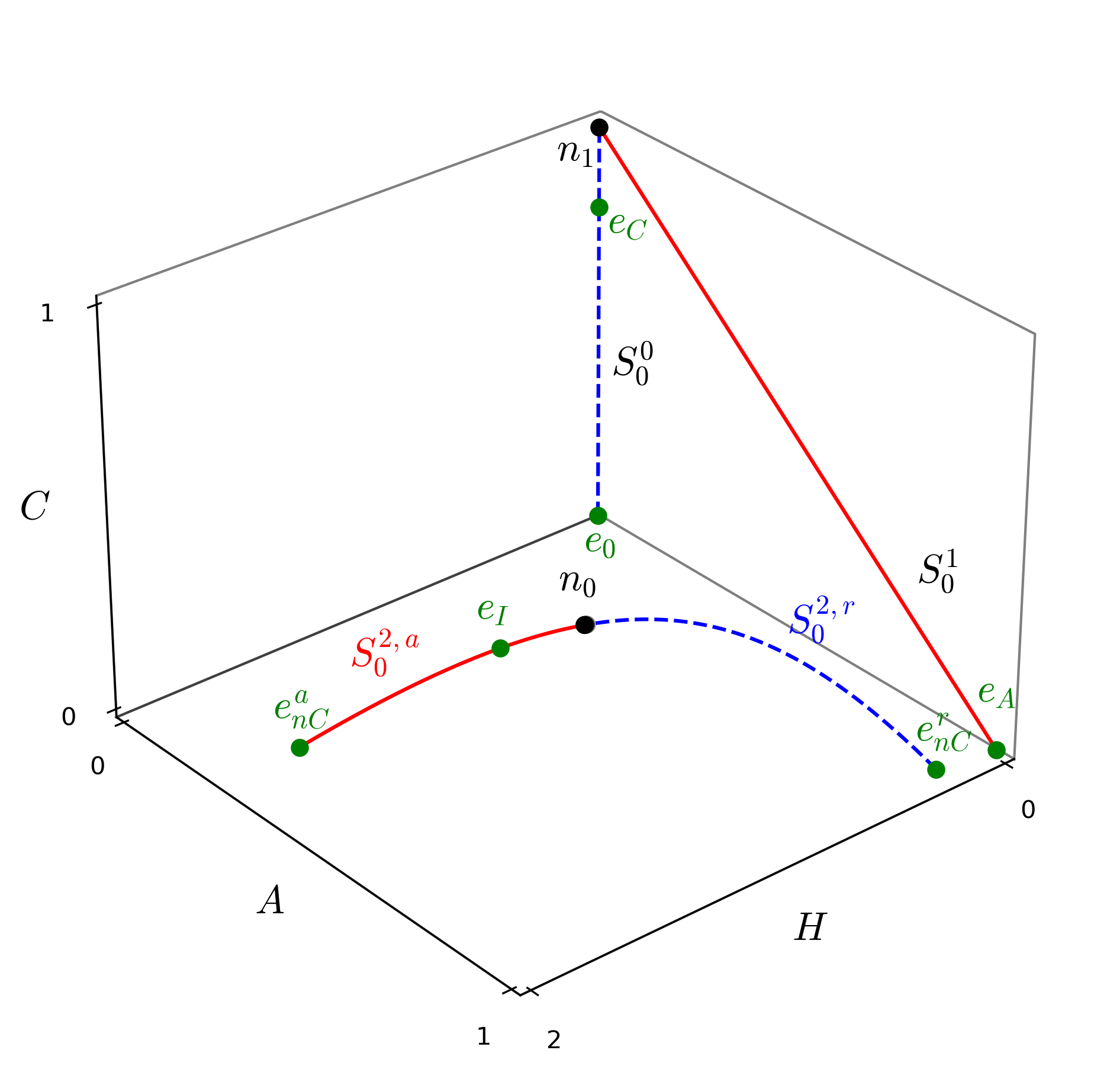}
        \subcaption{$\alpha>d$} 
    \end{subfigure}
    \caption{Geometry of the critical manifold $S_0$, with attracting and repelling portions indicated in solid red and dashed blue, respectively. Equilibria are indicated by green dots, while black dots correspond to non-hyperbolic singularities. The corresponding parameter values are $\beta = 0.2 = \lambda$, with $d=0.22$ and (a) $\alpha=0.1$; (b) $\alpha=0.22(=d)$; and (c) $\alpha=0.4$.}
    \label{fig: S_0-stability}
\end{figure}

\begin{remark}
The requirement that $\alpha>\lambda d^3$, which guarantees the uniqueness of the root $H_0$ of $Q(H;\alpha)$, is only partially restrictive from an ecological perspective. The value $d=0.22$, as given in \cite{gil2020}, was obtained from field measurements of reef fish foraging behaviour. Furthermore, from Table~\ref{param_table_2}, we find that $\lambda\leq 3.2$ and, therefore, $\lambda d^3\leq 3.2(0.22)^3\approx0.016$. Similarly, from Table~\ref{param_table_2}, we see that we are only considering $\alpha\in [0.1,0.7]$, which is a subset of the possible values $\alpha\in [0.00625, \infty)$ can take.
\end{remark}
The following result characterises the geometry of $S_0$ in terms of the fishing effort $\alpha$.
\begin{lem}
    Let $\lambda d^2<1$. Then, there exists a unique $\alpha$-value $\hat{\alpha}>d$ such that the following holds in the non-negative octant:
    \begin{enumerate}
        \item For $\alpha< d<\hat{\alpha}$ or $d<\alpha=\hat{\alpha}$, $S_0^2$ intersects the plane $\{C=0\}$ in a unique point.
        \item For $d\leq \alpha<\hat{\alpha}$, $S_0^2$ intersects the plane $\{C=0\}$ in two distinct points.
        \item For $\hat{\alpha}<\alpha$, $S_0^2$ never intersects the plane $\{C=0\}$.
    \end{enumerate}
    Furthermore, $\hat{\alpha}=\hat\alpha(\lambda, d)$ is a function of $\lambda$ and $d$ only.
    \label{lem: S_0^2 intersections}
\end{lem}
\begin{proof}
    It follows from the definition of $S_0^2$ in \eqref{eq: S_0} that the points of intersection of that sub-manifold with $\{C=0\}$ are given by the solutions of $\Pi(H)=\alpha$, where 
    \begin{equation}
        \Pi(H)=s(H)-\lambda H s(H)^2
        \label{eq: Pi(H)}
    \end{equation}
    for $H\geq 0$.
    We begin by showing that $\Pi(H)$ is a concave function on $[0, \infty)$. Note that $\Pi(0)=d>0$ and that $\Pi(H)\to -\infty$ as $H\to \infty$, since $s(H)\to d+1$ in that limit. Furthermore, $\frac{\dd}{\dd H}\Pi(H)\big\lvert_{H=0}=1-\lambda d^2>0$ by assumption. Therefore, $\Pi$ is increasing for $H>0$ sufficiently small. A straightforward but tedious calculation shows that $\frac{\dd}{\dd H}\Pi(H)$ is strictly decreasing. Hence, it follows that there exists a unique $H$-value $\hat{H}$ for which $\hat{\alpha}\equiv\Pi(\hat{H})$ is the global maximum of \eqref{eq: Pi(H)} on $(0, \infty)$. Furthermore, by the definition of $s(H)$, we have $\hat\alpha=\hat\alpha(\lambda,d)$. The result then follows by considering solutions of $\Pi(H)=\alpha$ for $H\geq 0$.
\end{proof}
Figure~\ref{fig: S_0-stability} visualises the singular geometry of \eqref{eq: S_0} for $\alpha=0.1$, $\alpha=0.22$, and $\alpha=0.4$, corresponding to the cases where $\alpha<d$, $\alpha=d$, and $\alpha>d$, respectively.

\subsection{Reduced flow}

The fast system in \eqref{fast-system} evolves in the fast time $t$ and converges to an $\mathcal{O}(\varepsilon)$-perturbation of the normally attracting portion of the critical manifold $S_0$. The flow on that so-called slow manifold is described in terms of the slow time $\tau=\varepsilon t$; that time rescaling transforms \eqref{fast-system} into the slow system:
\begin{equation}
    \begin{aligned}
        \varepsilon\frac{\dd H}{\dd \tau} &=  \lambda H(s(H) A - \alpha), \\
    \varepsilon\frac{\dd A}{\dd \tau} &= A(1-A-C-\lambda H s(H)), \\
    \frac{\dd C}{\dd \tau} &= C(1-\beta-A-C).
    \label{slow-system}
    \end{aligned}
\end{equation}

Taking $\varepsilon\to 0$ in \eqref{slow-system}, we obtain the so-called reduced problem, 
\begin{equation}
    \begin{aligned}
        0&=  \lambda H(s(H)A - \alpha), \\
    0&= A(1-A-C-\lambda H s(H)), \\
    \frac{\dd C}{\dd \tau} &= C(1-\beta-A-C),
    \label{reduced-system}
    \end{aligned}
\end{equation}
which approximates the dynamics on $S_0$ for $\varepsilon=0$. The next result describes the reduced flow on the critical manifold $S_0$.
\begin{lem}
The reduced flow on $S_0$ can be written as  
    \begin{equation}
        \begin{aligned}
            \frac{d C}{d\tau}& = C(1-\beta-C) \quad\text{on } S_0^0, \\
            \frac{d C}{d\tau}& = -\beta C \quad \text{on } S_0^1, \text{ and} \\
            \frac{d H}{d\tau}& = -\frac{(1+H)^2}{Q(H;\alpha)} \bigg(1-\frac{\alpha}{s(H)}-\lambda H s(H)\bigg)(\lambda H s(H)-\beta) \quad \text{on } S_0^2. \label{reduced flow}
        \end{aligned}
    \end{equation}
\end{lem}
\begin{proof}
    Standard calculation; see \cite[Chapter $3$]{MTSD}. In particular, the flow on $S_0^2$ is obtained from $\frac{\dd C}{\dd \tau}=-\frac{Q(H;\alpha)}{(1+H)^2}\frac{\dd H}{\dd \tau}$.
\end{proof}

Next, we determine the equilibria of the reduced flow in Equation~\eqref{reduced flow}.
On $S_0^0$, we have either $C=0$ or $C=1-\beta$ which, in terms of the reduced problem, Equation~\eqref{reduced-system}, corresponds to the extinction state $e_0=(0,0,0)$ and the coral-only equilibrium $e_C=(0,0,1-\beta)$, respectively. On $S_0^1$, we must have $C=0$, which yields the algae-only equilibrium $e_A=(0,1,0)$. It follows from \eqref{reduced flow} that $e_A$ is always an attracting equilibrium of the reduced problem.
Furthermore, by Lemma~\ref{lem: S_0^2 intersections}, $S_0^2$ intersects the plane $\{C=0\}$ in two distinct points for $\lambda d^2<1$ and $d<\alpha<\hat{\alpha}$. Since item 4 in Lemma~\ref{lemma: Stability of Critical Manifold} implies that $H_0$ maximises the curve $C(H)$ which defines $S_0^2$, we can conclude that the $H$-coordinates of those points satisfy $H_{nC}^r<H_0<H_{nC}^a$. We denote the corresponding equilibria by $e_{nC}^r=(H_{nC}^r, A_{nC}^r, 0)$ and $e_{nC}^a=(H_{nC}^a, A_{nC}^a, 0)$, respectively, where the superscripts indicate that $e_{nC}^r$ $(e_{nC}^a)$ lies on the repelling (attracting) portion of $S_0^2$; see Figure~\ref{fig: S_0-stability} for an illustration.

Finally, there exists a unique interior (coexistence) equilibrium $e_I=(H_I, A_I, C_I)$, with $\lambda H_Is(H_I)=\beta$, where
\begin{equation}
    H_I = \frac{\frac{\beta}{\lambda}-d+\sqrt{\big(\frac{\beta}{\lambda}+d\big)^2+\frac{4\beta}{\lambda}}}{2(d+1)},\quad
    A_I = \frac{\alpha}{s(H_I)},\quad\text{and}\quad
    C_I = 1-\beta - \frac{\alpha}{s(H_I)};
    \label{eq: coexistence point}
\end{equation}
moreover, $e_I$ lies in the non-negative octant when $\alpha\leq \alpha^+$, where
\begin{equation}
    \alpha^+=\alpha^+(\beta,\lambda,d)\equiv(1-\beta)s(H_I).
    \label{eq: alpha^+}
\end{equation}

\begin{remark}
    We note that when $\alpha=\hat{\alpha}$, $e_{nC}^r=n_0=e_{nC}^a$. Correspondingly, the equilibria $H_{nC}^r$ and $H_{nC}^a$ of the reduced flow on $S_0^2$, as defined in \eqref{reduced flow}, undergo a saddle-node bifurcation at $\alpha=\hat{\alpha}$, with $H_{nC}^r=\hat{H}=H_{nC}^a$ there. (Here, $\hat{H}$ is defined as in Lemma~\ref{lem: S_0^2 intersections}, with $\hat{\alpha}=\Pi(\hat{H})$.)
\end{remark}

We henceforth assume that $\beta<1$, which is a necessary condition for $e_I$ to be ecologically relevant, with $\alpha^+>0$.

\begin{remark}
    We note that the population $H_I$ of herbivorous fish at the coexistence equilibrium $e_I$, as given in \eqref{eq: coexistence point}, is independent of the parameter $\alpha$. From an ecological perspective, that is surprising, as the parameter $\alpha$ is proportional to mortality due to fishing. 
\end{remark}

\begin{remark}
    The requirement that $\lambda d^2<1$ as stated in Lemma~\ref{lem: S_0^2 intersections}, is again non-restrictive ecologically, since $d=0.22$ by Table~\ref{param_table_2} which, in combination with $\lambda d^2<1$, gives a  bound of $\lambda\lesssim 21$ that is consistent with $\lambda \in[0, 3.2]$ in Table~\ref{param_table_2}. Furthermore, the condition that $\alpha>\lambda d^3$ imposed in Lemma~\ref{lemma: Stability of Critical Manifold} holds simultaneously with $\lambda d^2<1$ when $\alpha>d$, which is a condition for $e_A$ to lie on the attracting portion of $S_0^1$; see Lemma~\ref{lemma: Stability of Critical Manifold}. We note that $e_A$ is an equilibrium of the full system, Equation~\eqref{fast-system}, corresponding to an algal bloom. As discussed in Section~\ref{sec: intro}, we are interested in regimes where \eqref{fast-system} is bistable; therefore, we require $A=1$ on $S_0^1$ to be a local attractor.  As a result, we will assume $\alpha>d$ unless stated otherwise. 
\end{remark}

\begin{cor}
    Suppose that $\lambda d^2<1$. Then, it follows that $\alpha^{+}\leq\hat\alpha$.
    \label{cor: hat bound}
\end{cor}
\begin{proof}
    The statement follows from the observation that there is a positive solution (in $H$) to the expression $\Pi(H)=\alpha^{+}$, which is equivalent to solving
    \begin{align*}
        (1-\beta)s(H_I) = s(H)-\lambda H s(H)^2.
    \end{align*}
    In particular, substituting in $H=H_I$, we find
    \begin{align*}
        1-\beta = 1-\lambda H_Is(H_I)
    \end{align*}
    or, equivalently, $\beta-\lambda H_{I}s(H_{I})=0$, which holds by the definition of the coexistence state $e_I$. Hence, we have shown that for $\alpha=\alpha^{+}$, $\Pi(H)=\alpha$ has the positive solution $H_I$; however, since $\hat\alpha$ is defined as the minimal $\alpha$-value such that $\Pi(H)=\alpha$ has no positive solution for $\alpha>\hat\alpha$, we must have $\alpha^{+}\leq \hat\alpha$, which completes the proof.
\end{proof}

The transcritical singularity \cite{MKrupa_2001} at $n_2=S_0^1\cap S_0^2$ is only located in the non-negative octant when $\alpha<d$ and merges with the algae-only equilibrium $e_A$ at $\alpha=d$. For $\alpha>d$, $S_0^2$ intersects the plane $\{C=0\}$ in the points $e_{nC}^{r}$ and $e_{nC}^a$, which are equilibria of the reduced problem, Equation~\eqref{reduced-system}, on $S_{0}^{2,r}$ and $S_{0}^{2,a}$, respectively. We note that, necessarily, $\alpha<\alpha^+\leq \hat{\alpha}$ in Figure~\ref{fig: S_0-stability}, since $e_I$ is shown to to be lying inside the non-negative octant. We now consider the stability of the equilibria on $S_0$. From the reduced flow on $S_0^0$ and $S_0^1$ in \eqref{reduced flow}, it is straightforward to see that $C=0$ corresponds to a repelling equilibrium on $S_0^0$, while $C=1-\beta(>0)$ and $C=0$ are attracting under the reduced flow on $S_0^0$ and $S_0^1$, respectively. The next result considers the stability of the equilibria of the reduced flow on $S_0^2$.

\begin{prop}
Assume that $\lambda d^2<1$, $\beta<1$, and $\alpha>d$. Then, there exists a unique $\alpha$-value $\alpha^\ast$, with $Q(H_I; \alpha^\ast)=0$, such that the  following holds.
\begin{enumerate}
    \item For $\alpha<\text{min}\{\alpha^+, \alpha^\ast\}$, $H_{nC}^r$ and $H_{I}$ are attracting equilibria for the reduced flow on $S_0^2$, whereas $H_{nC}^a$ is repelling under the reduced flow. Moreover, $H_{nC}^r<H_0<H_I<H_{nC}^a$, i.e., $H_I$ lies on the attracting branch $S_{0}^{2,a}$ of $S_0^2$.
    \item When $\alpha=\alpha^+<\alpha^\ast$, $H_{nC}^a$ and $H_{I}$ coalesce and undergo a transcritical bifurcation.
    \item For $\alpha^+<\alpha<\alpha^\ast$, $H_{nC}^a$ and $H_{nC}^r$ are attracting equilibria for the reduced flow on $S_0^2$. Moreover, $H_I$ is ecologically irrelevant due to $C_I<0$; additionally, $H_{nC}^r<H_0<H_{nC}^a<H_I$.
    \item When $\alpha=\alpha^\ast<\alpha^+$, $H_I=H_0$, with $H_I$ a non-hyperbolic fold point. Moreover, $H_{nC}^a$ is a repelling equilibrium for the reduced flow, whereas $H_{nC}^r$ is attracting.
    \item For $\alpha^\ast<\alpha<\alpha^+$, $H_I$ is a repelling equilibrium for the reduced flow. Moreover, $H_{nC}^r<H_I<H_0<H_{nC}^a$, i.e., $H_I$ lies on the repelling branch $S_{0}^{2,r}$ of $S_0^2$, with $H_{nC}^a$ repelling and $H_{nC}^r$ attracting.
\end{enumerate}
Moreover, 
\begin{equation}
\alpha^\ast=\alpha^\ast(\beta,\lambda,d)=\lambda s(H_I)^2\big[H_I+s(H_I)(1+H_I)^2 \big]
\label{alpha star of H_I}
\end{equation}
is a function of $\beta$, $\lambda$, and $d$.
\label{prop: equilibrium stability}
\end{prop}
\begin{proof}
See Appendix~\ref{sec: Stability of internal equilibrium}.
\end{proof}
\begin{remark}
We do not consider the case where $\alpha>\max\{\alpha^+,\alpha^\ast\}$ in Proposition~\ref{prop: equilibrium stability}, as it will turn out to be irrelevant to our analysis.
\end{remark}
Representative geometries resulting from the ``generic" regimes in items 1, 3, and 5 of Proposition~\ref{prop: equilibrium stability} are illustrated in Figure~\ref{fig:S_0^2 cases}.

\begin{figure}[H]
    \centering
    \begin{subfigure}[t]{0.45\textwidth}
        \centering
        \includegraphics[scale = 1]{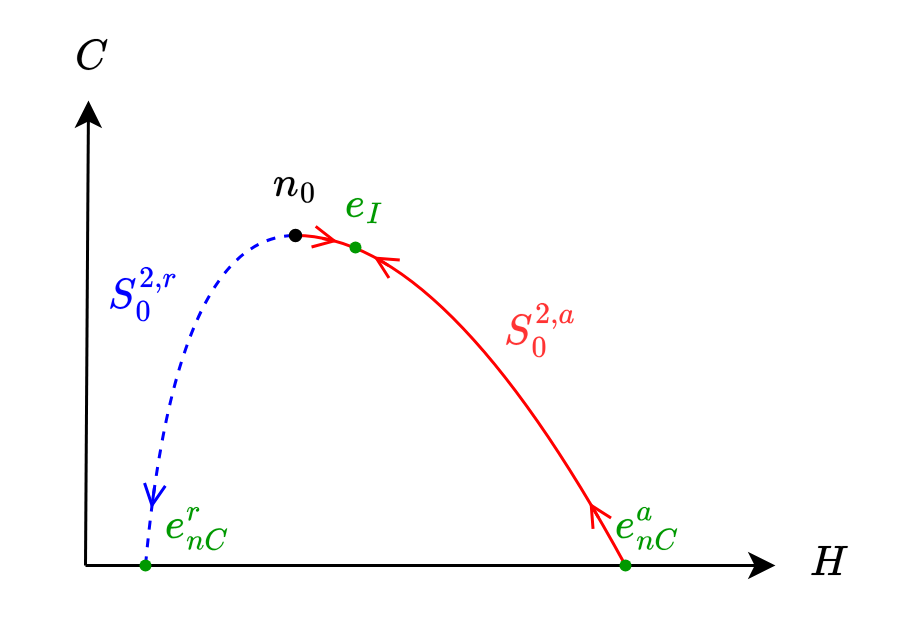}
        \subcaption{$\alpha<\text{min}\{\alpha^+, \alpha^\ast\}$} 
    \end{subfigure}
    \hfill
    \begin{subfigure}[t]{0.45\textwidth}
        \centering
        \includegraphics[scale = 1]{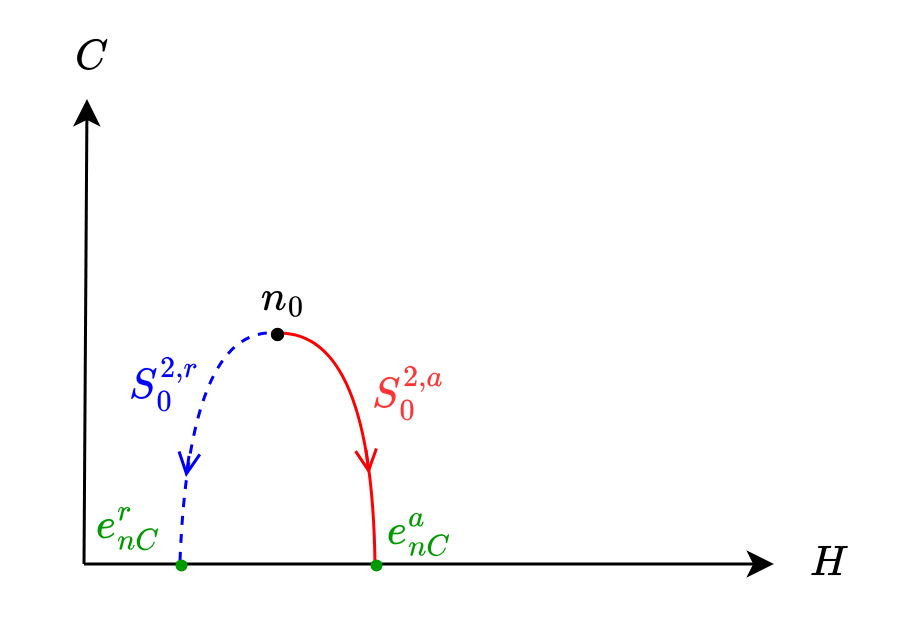}
        \subcaption{$\alpha^+<\alpha<\alpha^\ast$}
    \end{subfigure}
    \hfill
    \begin{subfigure}[t]{0.45\textwidth}
        \centering
        \includegraphics[scale = 1]{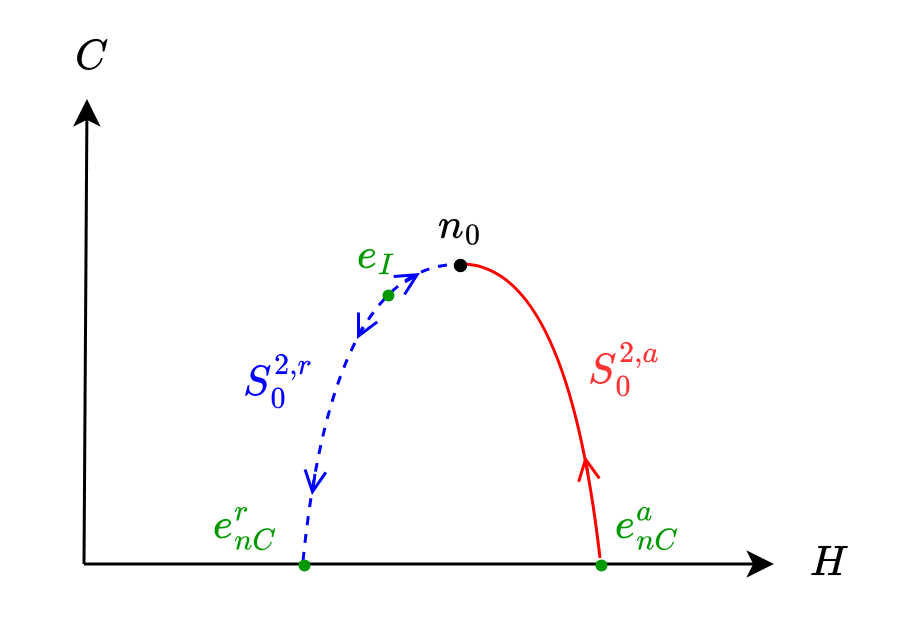}
        \subcaption{$\alpha^\ast<\alpha<\alpha^+$}
    \end{subfigure}
    \caption{Reduced flow on $S_0^2$ for $\alpha>d$; cf.~Equation~\eqref{reduced flow} and Proposition~\ref{prop: equilibrium stability}.}
    \label{fig:S_0^2 cases}
\end{figure}

\begin{remark}
    In the case where $\alpha^\ast<\alpha^{+}$, the transcritical bifurcation in Proposition~\ref{prop: equilibrium stability}, item 2 occurs between $e_{I}$ and $e_{nC}^{r}$, instead of between $e_{I}$ and $e_{nC}^a$. That case is not of interest, however, as $e_{nC}^r$ is located on the unstable branch of the sub-manifold $S_0^2$ then.
\end{remark}
\begin{prop}Suppose that  $\lambda d^2<1$ and that any of $\alpha^{+}$, $\alpha^\ast$, and $\hat\alpha$ are pairwise equal. Then, it follows that $\alpha^{+}=\alpha^\ast=\hat\alpha$.
\label{Prop: alpha equality}
\end{prop}
\begin{proof}
We first assume that $\alpha^{+}=\hat\alpha$. From the definition of $S_0^2$ in \eqref{eq: S_0}, we recall that the curve $C(H)$ is defined by
\begin{align}
    C(H) &= 1-\frac{\alpha}{s(H)}-\lambda Hs(H). 
    \label{eq: C(H)}
\end{align}
By Lemma~\ref{lem: S_0^2 intersections}, the function in \eqref{eq: C(H)} has a unique positive root for $\alpha=\hat{\alpha}$ and is therefore maximised at $\hat{\alpha}$. Thus, $C'(H)=0$ when $\alpha=\hat{\alpha}$, where 
\begin{equation}
    C'(H) = \frac{\alpha}{s(H)^2(1+H)^2}-\lambda\bigg[s(H)+\frac{H}{(1+H)^2}\bigg] = -\frac{Q(H;\alpha)}{(1+H)^2}.
    \label{eq: C'(H)}
\end{equation}
Furthermore, it was shown in the proof of Corollary~\ref{cor: hat bound}, that $\Pi(H)=\alpha^+$ has a positive solution given by $H=H_I$. Evaluating $C'(H)=0$ in \eqref{eq: C'(H)} at $H=H_I$, we find $Q(H_I;\hat\alpha)=0$, which, by Proposition~\ref{prop: equilibrium stability}, implies $\hat\alpha=\alpha^\ast$.

Since that argument is reversible, the statement equally follows if we assume $\alpha^\ast=\hat\alpha$. Finally, for the case where $\alpha^+=\alpha^\ast$, one can verify directly that $H_{I}$ is a solution to both $C(H)=0$ and $C'(H)=0$, which completes the proof.
\end{proof}

\subsection{Bifurcation analysis}

Corollary~\ref{cor: hat bound} and Proposition~\ref{Prop: alpha equality} essentially define three distinct regimes for $\alpha$: $\alpha^\ast<\alpha^{+}<\hat\alpha$, $\alpha^{+}<\alpha^\ast<\hat\alpha$, and $\alpha^{+}<\hat\alpha<\alpha^\ast$. Our next result concerns the potential for transitions between these regimes. 
\begin{prop}
Assume that $\lambda d^2<1$, and let $\beta_0<1$ be given. Furthermore,  suppose that $\alpha^{+}(\beta_0,\lambda,d)<\alpha^\ast(\beta_0,\lambda,d)<\hat\alpha(\lambda,d)$ when $\beta=\beta_0$. Then, for all $\beta_0<\beta<1$, we must have $\alpha^{+}(\beta,\lambda,d)<\alpha^\ast(\beta,\lambda,d)<\hat\alpha(\lambda,d)$, i.e., no transition between these $\alpha$-regimes can occur under variation of $\beta$.
\label{prop: possible bifurcations}
\end{prop}
\begin{proof}
We proceed by contradiction. For a transition to be possible, we require a value of $\beta$ such that any two of $\alpha^+$, $\alpha^\ast$, and $\hat\alpha$ are pairwise equal. We assume that such a $\beta$-value exists, which we denote by $\beta_1$. Then, by Proposition~\ref{Prop: alpha equality}, $\alpha^+(\beta_1,\lambda,d)=\alpha^\ast(\beta_1,\lambda,d)=\hat\alpha(\lambda,d)$. As $\hat\alpha$ is independent of $\beta$ and as $\alpha^{+}\leq \hat{\alpha}$, we may conclude that $\frac{\partial\hat\alpha}{\partial \beta}\rvert_{\beta=\beta_1}=0=\frac{\partial\alpha^{+}}{\partial \beta}\rvert_{\beta=\beta_1}$. In particular, for $\alpha^\ast$ to remain wedged between $\alpha^+$ and $\hat\alpha$, its derivative must also be zero at $\beta=\beta_1$. However, one may verify that $\frac{\partial\alpha^\ast}{\partial \beta}=\frac{\partial H_I}{\partial \beta} \frac{\partial \alpha^\ast}{\partial H_I}>0$ for all $\beta$, due to $H_I\equiv H_I(\beta)$ and $\alpha^\ast\equiv\alpha^{\ast}(H_I)$, as defined in \eqref{eq: coexistence point} and \eqref{alpha star of H_I}, respectively, being strictly monotonically increasing with respect to $\beta$ and $H_I$, respectively; see Figure~\ref{fig: proof of alpha regiems} for an illustrative sketch.
\end{proof}
\begin{figure}[H]
    \centering
    \includegraphics[scale = 1.2]{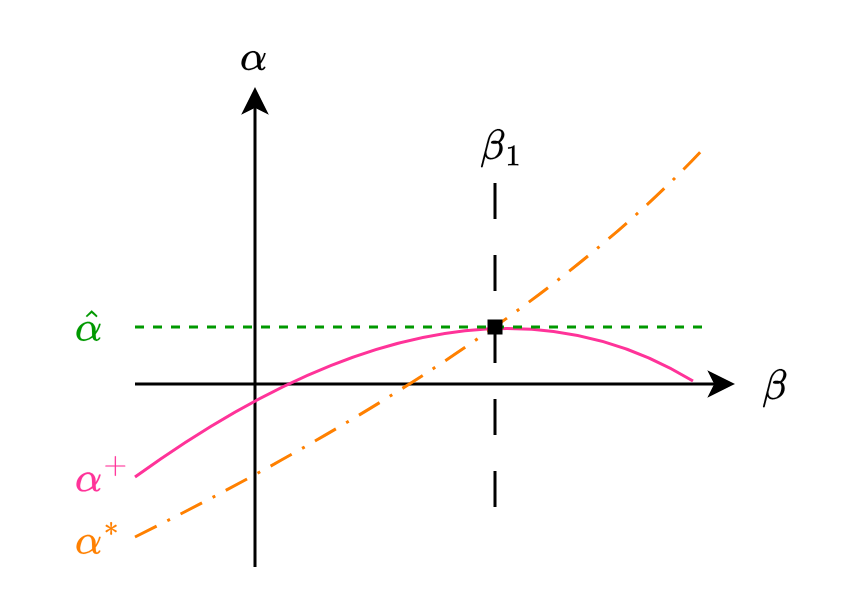}
    \caption{Illustrative sketch of the proof of Proposition~\ref{prop: possible bifurcations}.}
    \label{fig: proof of alpha regiems}
\end{figure}

Proposition \ref{prop: possible bifurcations} ensures that only transitions between the regimes where $\alpha^\ast<\alpha^{+}<\hat\alpha$ or $\alpha^{+}<\hat\alpha<\alpha^\ast$ are possible. In particular, we can evaluate $\alpha^+$ and $\alpha^\ast$ at $\beta=0$ to determine which of these regimes are feasible for any positive $\beta(<1)$. We note that we have $H_I=0$ when $\beta=0$, which implies $\alpha^{+}=d$ and $\alpha^\ast=\lambda d^3$; recall Equations~\eqref{eq: alpha^+} and \eqref{alpha star of H_I}, respectively. In particular, $\alpha^{+}-\alpha^\ast=d(1-\lambda d^2)>0$ when $\beta=0$, placing us within the regime where $\alpha^\ast<\alpha^{+}<\hat\alpha$. It follows that the regime where $\alpha^{+}<\alpha^\ast<\hat\alpha$ will never occur for any value of $\beta$, and that it hence need not be considered. For $d=0.22$ fixed, the two regimes of interest are illustrated in Figure~\ref{fig: minimum alpha plot}.

\begin{figure}[H]
    \centering
    \includegraphics[width=0.5\linewidth]{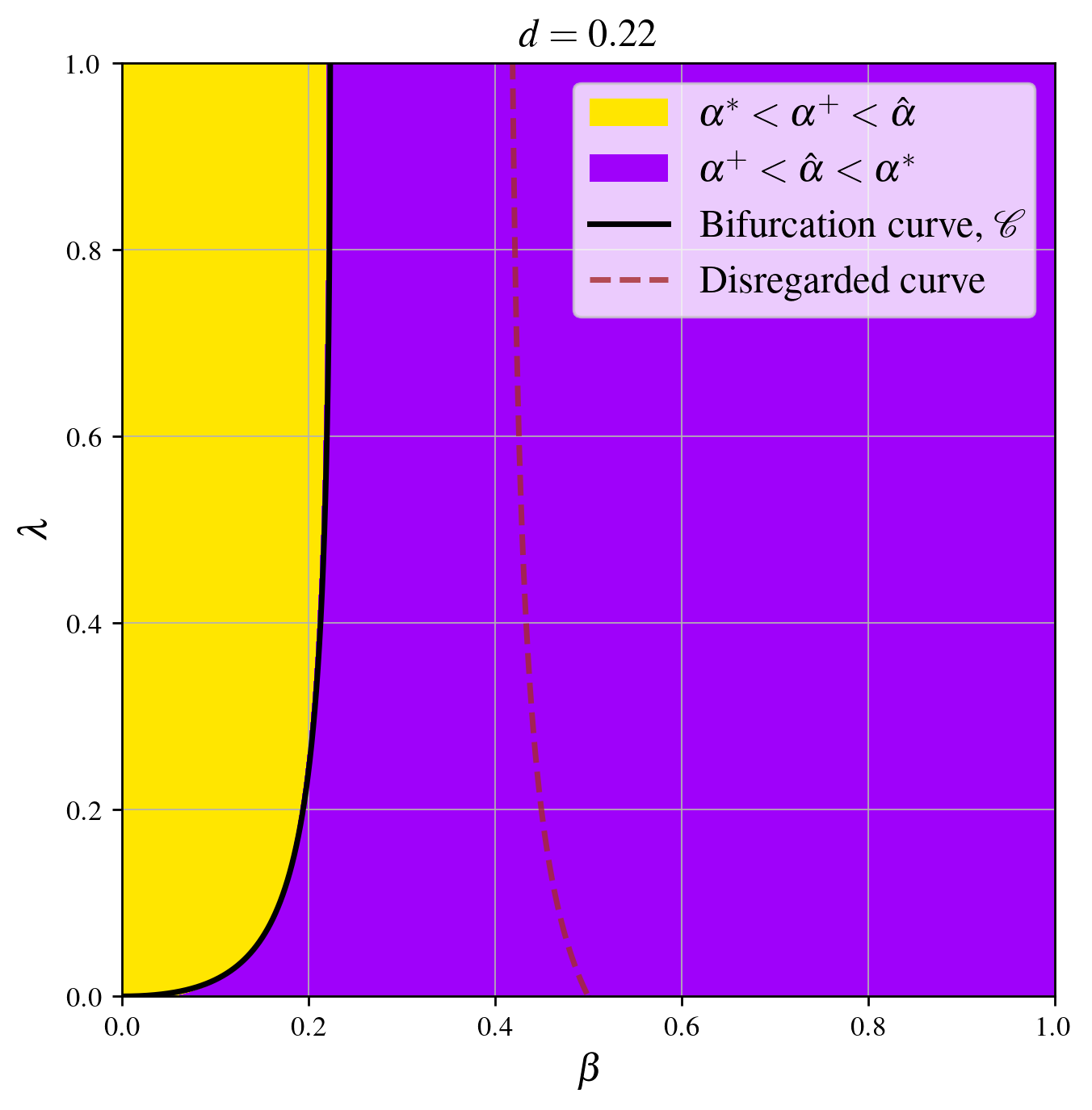}
    \caption{Illustration of the two relevant regimes for $\alpha^+$, $\alpha^\ast$, and $\hat\alpha$, with $\beta, \lambda\in (0,1)^2$ and $d=0.22$ fixed. The regime where $\alpha^\ast<\alpha^{+}<\hat\alpha$ ($\alpha^{+}<\hat\alpha<\alpha^{\ast}$) is shaded in yellow (purple); the bifurcation curve $\mathcal{C}$ separating the two regimes (solid black) is defined in Equation~\eqref{eq: bifurcation curve}, along with a disregarded curve (dashed brown)}
    \label{fig: minimum alpha plot}
\end{figure}

The bifurcation curve $\mathcal{C}$ shown in Figure~\ref{fig: minimum alpha plot} that marks the transition between the two regimes is given by the solution of $\alpha^\ast=\alpha^{+}$ for $\lambda$, with $H_I$ substituted in from \eqref{eq: coexistence point}, which yields
\begin{equation}
    \lambda=-\frac{2\beta d(2\beta-1)+(3\beta-1)^2\pm(3\beta-1)\sqrt{4\beta d(2\beta-1)+(3\beta-1)^2}}{2d^2(2\beta-1)}.
    \label{eq: bifurcation curve}
\end{equation}

The union of the two solution branches in \eqref{eq: bifurcation curve} defines two disjoint curves in $(0,1)^2$. One can show that only one of those curves -- the curve $\mathcal{C}$ indicated in solid black in Figure~\ref{fig: minimum alpha plot} -- is relevant as a bifurcation curve, though, as it satisfies $\lambda\geq\frac{\beta^2}{1-2\beta-2\beta d}$, while the other curve (indicated in dashed brown) does not: for $\lambda<\frac{\beta^2}{1-2\beta-2\beta d}$, 
\begin{align*}
    \alpha^\ast-\alpha^{+} = s(H_I)\bigg(2\beta-1+\frac{\beta^2}{\lambda}+2\beta s(H_I)+\lambda s(H_I)^2\bigg)\geq s(H_I)\bigg(2\beta-1+\frac{\beta^2}{\lambda}+2\beta d\bigg)>0,
\end{align*}
which contradicts the definition of a bifurcation curve ($\alpha^\ast=\alpha^{+}$). Hence, that other curve can be discounted. (Here, we have again used $\lambda s(H_I)H_I=\beta$.) 

We conclude this section by summarising our key findings. We emphasise that our main aim here was to show that there exist parameter regimes where Equation~\eqref{fast-system} represents a bistable fast-slow system. 

Recall that $e_A=(0,1,0)$ lies on $S_0^1$, which is normally attracting in a neighbourhood of $e_A$ provided $\alpha>d$; cf. Lemma~\ref{lemma: Stability of Critical Manifold}. It follows from \eqref{reduced flow} that the point $A=1$ is always an attracting equilibrium for the reduced flow on $S_0^1$. Therefore, to ensure bistability, we need to show the existence of ecologically relevant local attractors, under the reduced flow, on the attracting portion $S_0^{2,a}$ of $S_0^2$. By Proposition~\ref{prop: equilibrium stability}, there are three non-trivial $\alpha$-regimes to consider: $\alpha<\min\{\alpha^+, \alpha^{\ast}\}$, $\alpha^+<\alpha<\alpha^\ast$, and $\alpha^\ast<\alpha<\alpha^+$, which correspond to items 1, 3, and 5 therein, respectively; see also Proposition~\ref{prop: possible bifurcations} and Figure~\ref{fig: minimum alpha plot}.

We first note that the equilibrium $e_{nC}^a$ lies on $S_0^{2,a}$ in all three regimes, and that it is hence ecologically relevant by definition; by contrast, $e_I$ lies in the positive octant provided $C_I>0$, which is the case when $\alpha<\text{min}\{\alpha^+, \alpha^\ast\}$ (Proposition~\ref{prop: equilibrium stability}, item 1); moreover, $e_I\in S_0^{2,a}$ then, with $H_I$ being attracting under the reduced flow. When $\alpha^+<\alpha<\alpha^\ast$ (Proposition~\ref{prop: equilibrium stability}, item 3), $e_I$ is irrelevant due to $C_I<0$; however, $H_{nC}^a$ is attracting. Finally, for $\alpha^{\ast}<\alpha<\alpha^+$ (Proposition~\ref{prop: equilibrium stability}, item 5), neither of the equilibria $e_I$ and $e_{nC}^a$ is an attractor under the reduced flow on $S_0^2$; hence, we can discount that regime.

Due to $S_0^1$ and $S_0^2$ being normally hyperbolic in a neighbourhood of $e_A$, $e_{nC}^a$, and $e_I$ for $\alpha>d$ and $\alpha\neq \alpha^+,\alpha^\ast$, it follows from standard persistence results \cite[Theorem $3.1.4$]{MTSD} that $e_A$ and $e_I$ are attractors for Equation~\eqref{fast-system} when $\varepsilon\in(0, \varepsilon_0)$, with $\varepsilon_0$ sufficiently small and $\alpha<\text{min}\{\alpha^+, \alpha^\ast\}$. Likewise, $e_A$ and $e_{nC}^a$ are attractors for \eqref{fast-system} when $\alpha^+<\alpha<\alpha^\ast$. In sum, we hence observe bistability in those two $\alpha$-regimes.

Finally, we emphasise that the regimes corresponding to items 2 and 4 in Proposition~\ref{prop: equilibrium stability} equally cause coral to go extinct. Specifically, at $\alpha=\alpha^+<\alpha^\ast$, the coexistence state $e_I$ which was initially a local attractor for \eqref{fast-system} exchanges stability with $e_{nC}^a$ in a transcritical bifurcation. As a results, we observe extinction of the coral population, while the populations of fish and algae increase. Ecologically speaking, an increase in the fishing effort $\alpha$ causes a reduction in the consumption of fast-evolving algae, which then out-compete coral and drive it to extinction. Similarly, when $\alpha=\alpha^\ast<\alpha^+$, the local attractor $e_I$ loses stability, with $e_A$  the only remaining attractor for $\alpha^\ast<\alpha<\alpha^+$. In ecological terms, the increase in the fishing effort $\alpha$ causes both fish and coral to become extinct, while the population of algae will converge to its carrying capacity, resulting in an algal bloom. While both of these regimes hence result in the collapse of the coral reef ecosystem, that collapse is due to a ``classical" bifurcation (``B-tipping"), rather than to rate-induced tipping (``R-tipping") \cite{ashwin_2012_tipping}. Hence, we do not consider them further here.

\begin{remark}
  Alternatively, \eqref{PNAS Model} can be reformulated as a $(1,2)$ fast-slow system, with algae the only fast species. Due to $\lambda_0\in[0, 3.2]$, by Table~\ref{param_table}, we can assume that $\lambda=\rho \varepsilon$ for small $\lambda$, where $\rho$ is considered to be $\mathcal{O}(1)$. The resulting system is given by 
    \begin{equation}
        \begin{aligned}
            \frac{\dd H}{\dd t} &=  \varepsilon \rho H(s(H) A - \alpha), \\
    \frac{\dd A}{\dd t} &= A(1-A-C-\lambda H s(H)), \\
    \frac{\dd C}{\dd t} &= \varepsilon C(1-\beta-A-C), 
    \label{eq (1,2) fast slow system}
        \end{aligned}
    \end{equation}
   which admits the critical manifold $\mathcal{M}_0=\mathcal{M}_{0}^0\cup \mathcal{M}_{0}^1$, again restricted to $\mathcal{D}$, where $\mathcal{M}_{0}^0=\{(H,A,C)\in \mathcal{M}_0\,|\, A=0\}$ and $\mathcal{M}_{0}^1=\{(H,A,C)\in \mathcal{M}_0\,|\, A=1-C\}$. Since no coexistence equilibrium exists for the reduced flow on either $\mathcal{M}_0^0$ or $\mathcal{M}_0^1$, with the algae-only equilibrium $e_A=(0,1,0)$ the only attractor for Equation~\eqref{eq (1,2) fast slow system}, rate-induced tipping is not possible in the monostable Equation~\eqref{eq (1,2) fast slow system}. 
\end{remark}

\section{Tipping versus tracking}
\label{section R-tipping}

\subsection{Extended (ramped) system}
To investigate R-tipping as a result of an increase in the fishing effort $\alpha$, we interpret $\alpha$ as a dynamic parameter. Correspondingly, we append an equation for the rate of change of $\alpha$ to Equation~\eqref{slow-system}; see also the set-up in \cite{Plankton-tipping, vanselow_2019_when}.
As in Section \ref{section GPST}, we assume that we are in a regime where \eqref{slow-system} exhibits bistability, with the coexistence equilibrium $e_I$ being ecologically relevant and locally attracting and the algae-only state $e_A$ being an attractor. In other words, we only consider the parameter regime in item 1 of Proposition~\ref{prop: equilibrium stability}.
Correspondingly, we take $\alpha\in(\alpha_{\rm min},\alpha_{\rm max})$, where $\alpha_{\rm min}\equiv d$ and $\alpha_{\rm max}\equiv \min \{\alpha^+, \alpha^\ast\}$. Since we assume an increase in the fishing effort from coexistence here, we require that the state $e_I$ lies on the attracting portion of the critical manifold when $\alpha=d$, which places a further restriction on the permissible ranges for the parameters $\beta$ and $\lambda$ in Figure~\ref{fig: minimum alpha plot}; see Figure~\ref{fig: bifurcation regions} for an illustration.

As we are interested in slowly increasing the fishing effort $\alpha$ while remaining within ecologically relevant regimes, we ``ramp up" $\alpha$ from $\alpha_{\min,\delta} = d+\delta$ to $\alpha_{\max,\delta} = \alpha_{\rm max}-\delta$; here, $\delta=0.01$ is a small, but fixed constant which is introduced to ensure that the bifurcations which occur when $\alpha$ crosses either of $\alpha_{\rm min}$ or $\alpha_{\rm max}$ do not play a role in the dynamics; recall Lemma~\ref{lem: S_0^2 intersections} and Proposition~\ref{prop: equilibrium stability}.

The combination of these restrictions ensures that any regime shift caused by a change in $\alpha$ will not be due to a ``classical" bifurcation. In fact, we will show that rate-induced tipping can occur either due to a folded-node-mediated canard \cite{SZMOLYAN2001419, Canards} or via a regular jump point \cite{fold}. 

We first write the extended (or ``ramped") system resulting from Equation~\eqref{slow-system} in the slow formulation,
\begin{equation}
\begin{split}
    \varepsilon \frac{\dd H}{\dd \tau} &= \lambda H(s(H)A-\alpha), \\
    \varepsilon \frac{\dd A}{\dd \tau} &= A(1-A-C-\lambda Hs(H)), \\
    \frac{\dd C}{\dd \tau} &= C(1-\beta-A-C),\\
    \frac{\dd \alpha}{\dd \tau} &=\begin{cases}
        r & \quad\text{for }\alpha_{\min,\delta}< \alpha < \alpha_{\max,\delta},\\
        0 & \quad \text{otherwise}.
    \end{cases}
\end{split}
\label{eq: full 4D system}
\end{equation}
Equation~\eqref{eq: full 4D system} is a $(2,2)$ fast-slow system, with $(H,A)$ the fast variables and $(C,\alpha)$ the slow ones; in particular, we assume that the fishing effort $\alpha$ is increasing with rate $r$ due to an increased demand for fish.
The dynamics on the corresponding critical manifold, which is now the union of two-dimensional surfaces in the four-dimensional $(H,A,C,\alpha)$-space, are governed by the reduced system that is obtained from \eqref{eq: full 4D system} with $\varepsilon=0$:
\begin{align}
    0 &= \lambda H(s(H)A-\alpha)\label{eq: crit manifold 1}, \\
    0 &= A(1-A-C-\lambda Hs(H))\label{eq: crit manifold 2}, \\
    \frac{\dd C}{\dd \tau} &= C(1-\beta-A-C),\\
    \frac{\dd \alpha}{\dd \tau} &=\begin{cases}
        r & \quad\text{for } \alpha_{\min,\delta}< \alpha < \alpha_{\max,\delta},\\
        0 & \quad \mathrm{otherwise}.
    \end{cases}
\end{align}
Specifically, the ``extended" critical manifold, which we now denote by $S_{0,\alpha}$, can be written as a union
$S_{0,\alpha}=S_{0,\alpha}^0\cup S_{0,\alpha}^1\cup S^2_{0,\alpha}$, with
\begin{align}
    &S_{0,\alpha}^0 = \big\{(H,A,C,\alpha)\in S_{0,\alpha}\,\big|\, H=0,\ A=0,\ \alpha_{{\rm min},\delta}< \alpha < \alpha_{{\rm max},\delta}\big\}, \\ 
    &S_{0,\alpha}^1 =\big\{(H,A,C,\alpha)\in S_{0,\alpha}\,\big|\, H=0,\ A=1-C,\ \alpha_{{\rm min},\delta}< \alpha < \alpha_{{\rm max},\delta}\big\}, \quad\text{and} \\
    &S_{0,\alpha}^2 = \bigg\{(H,A,C,\alpha)\in S_{0,\alpha}\,\bigg|\,  A = \frac{\alpha}{s(H)},\ C=1-\frac{\alpha}{s(H)}-\lambda Hs(H),\ \alpha_{{\rm min},\delta}< \alpha < \alpha_{{\rm max},\delta}\bigg\}\label{eq: S^2_0 def},
\end{align}
where $\alpha$ is now included as a state variable; here, we again restrict to $(H,A,C)\in\mathcal{D}$. 
The critical manifold $S_{0,\alpha}$ for \eqref{eq: full 4D system} is illustrated in Figure~\ref{fig: full critical manifold}. (While the region where $0\leq \alpha< \alpha_{\rm \min,\delta}$ is not considered, it is included in Figure~\ref{fig: full critical manifold}, for completeness.)

The stability properties of $S_{0, \alpha}$ are consistent with those of the critical manifold $S_0$ for \eqref{slow-system}, in that the stability of any point $(H,A,C)\in S_0$ will be identical to that of the point $(H,A,C,\alpha)\in S_{0, \alpha}$. The main difference between the singular geometries in Figures~\ref{fig: S_0-stability} and \ref{fig: full critical manifold} is that the equilibria $e_X$ ($X\in\{I,A\}$) and $e_{nC}^i$ ($i\in\{a,r\}$) and the non-hyperbolic singularities $n_j$ ($j\in\{0, 1,2\}$) on the critical manifold $S_0$ for \eqref{slow-system} extend to curves on $S_{0,\alpha}$ for \eqref{eq: full 4D system}; correspondingly, we henceforth include the subscript $\alpha$ to emphasise their $\alpha$-dependence.
  
\begin{figure}
    \centering
     \includegraphics[scale = 0.2]{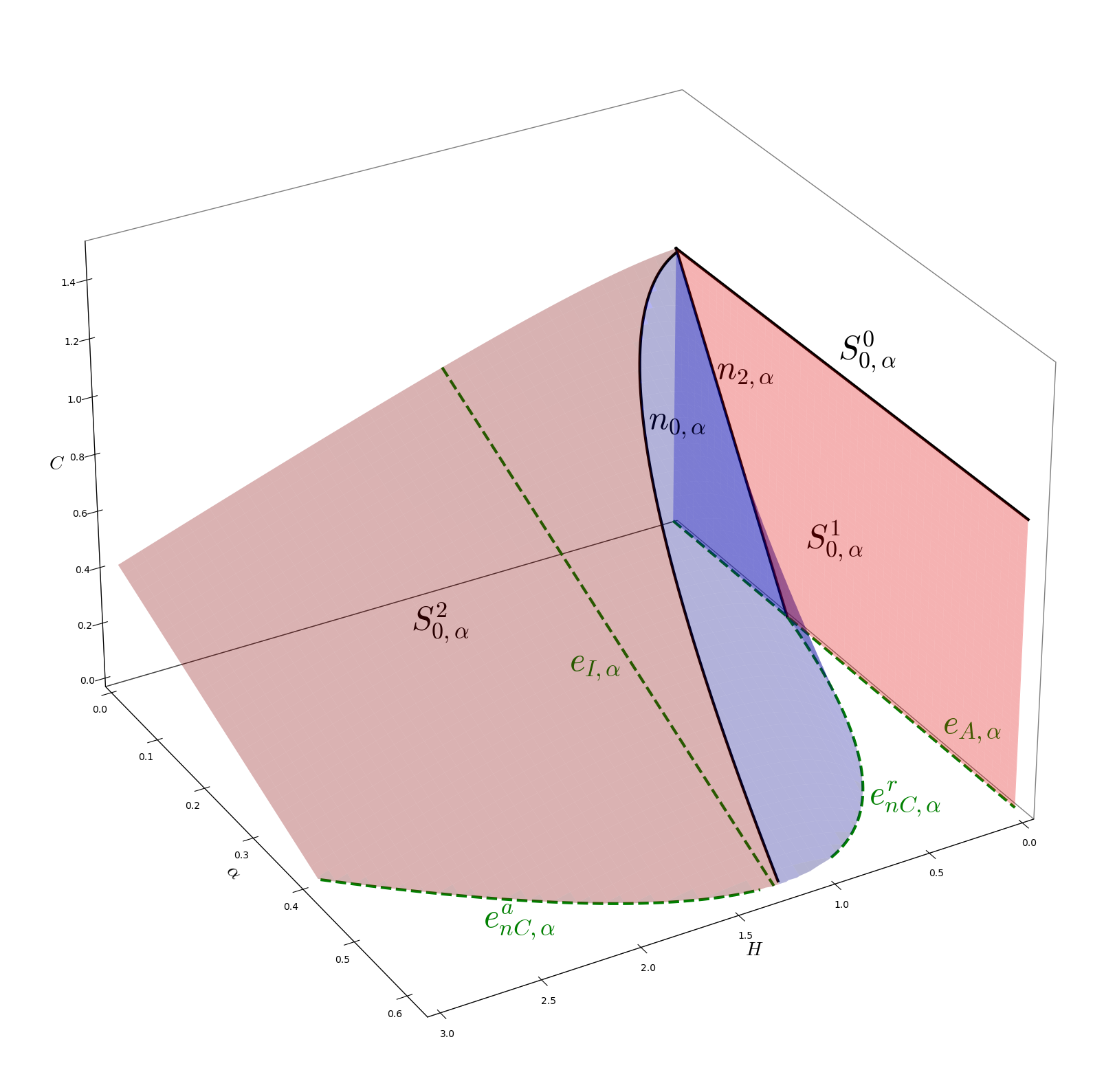}
    \includegraphics[scale = 0.2]{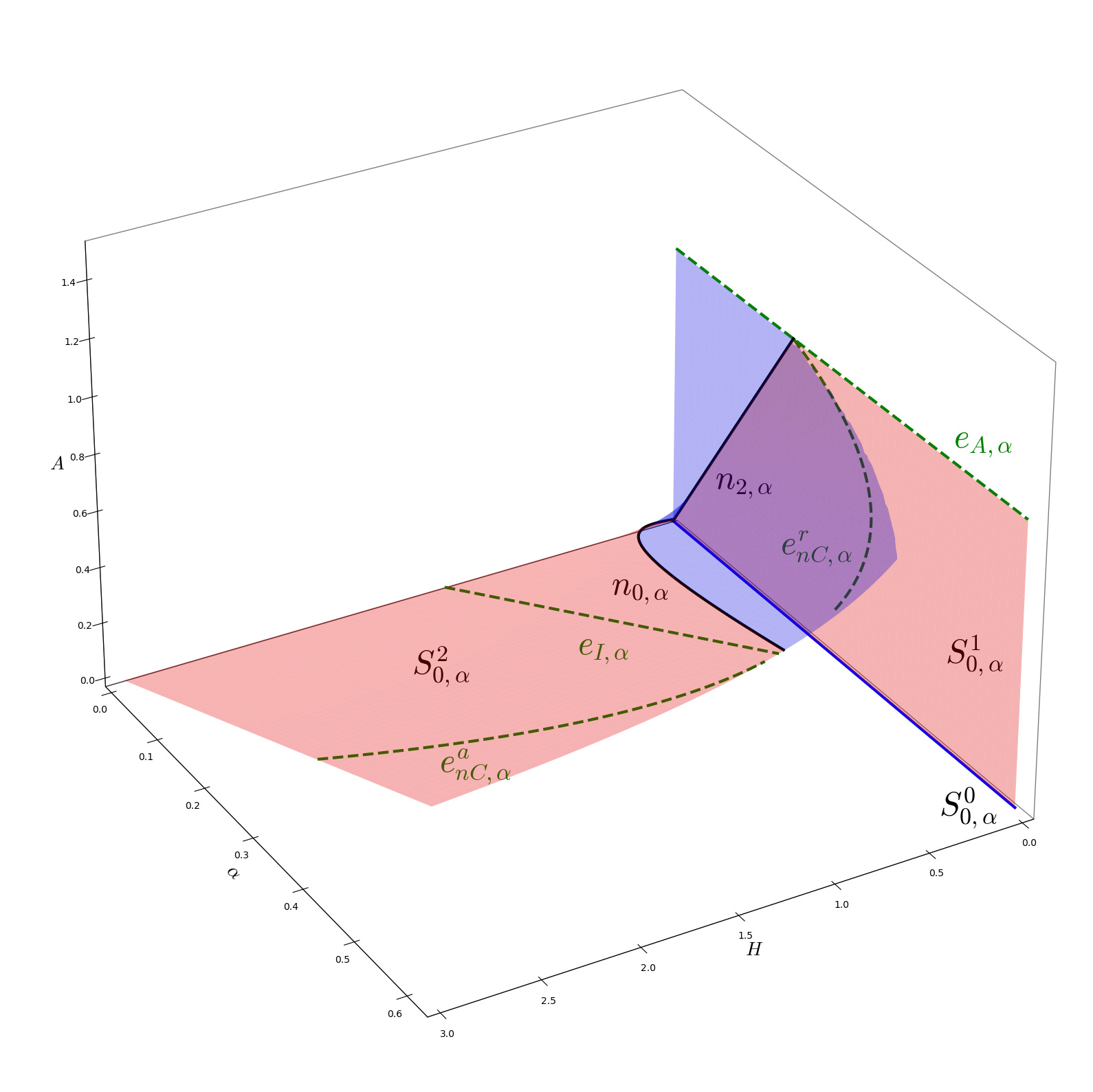}
    \caption{Geometry of the extended critical manifold $S_{0,\alpha}$, with attracting and repelling sections marked in red and blue, respectively, projected into $(H,\alpha,C)$-space (left panel) and $(H,\alpha,A)$-space (right panel). The corresponding parameter values are $\beta=0.2=\lambda$, with $d=0.22$ and $\alpha\in(0, \alpha_{\rm max})$.}
    \label{fig: full critical manifold}
\end{figure}

\begin{remark}
    We define ``tracking" on the basis of the definition given in \cite[Section $2$]{ashwin_2012_tipping}. We will always initiate the flow of \eqref{eq: full 4D system} at $(H, A, C, \alpha)(0)=(H_I, A_I, C_I, \alpha_{{\rm min},\delta})$; then, we say that a trajectory ``tracks" the equilibrium $e_{I, \alpha}=(H_I, A_I, C_I, \alpha)$, which is the extension of the interior equilibrium $e_I$ for \eqref{fast-system} for any fixed $\alpha$, if $(H, A, C, \alpha)(\tau)\to(H_I, A_I, C_I, \alpha_{{\rm max}, \delta})$ as $\tau\to \infty$. Otherwise, we say that ``tipping" has occurred. While that condition hence strengthens the one in \cite{ashwin_2012_tipping}, in that it requires asymptotic, rather than neutral, stability, it seems appropriate here, given that the flow will always tend to the attractor $e_{I,\alpha}$ once $\alpha=\alpha_{{\rm max}, \delta}$ if it has not tipped.
    \label{remark: tipping vs tracking}
\end{remark}

\subsection{Folded singularities}

We are primarily interested in the sub-manifold $S^2_{0,\alpha}$, as it contains the extended coexistence state $e_{I, \alpha}$. The reduced flow on $S_{0,\alpha}^2$ can be derived by differentiation of the algebraic constraints in \eqref{eq: crit manifold 1} and \eqref{eq: crit manifold 2} with respect to the slow time $\tau$, followed by substitution of the expressions for $A$ and $C$ from \eqref{eq: S^2_0 def}, which gives the two-dimensional reduced system
\begin{equation}
    \begin{aligned}
        \frac{\dd H}{\dd \tau} &= -\frac{(1+H)^2}{Q(H,\alpha)}\bigg[\bigg(1-\frac{\alpha}{s(H)}-\lambda H s(H)\bigg)(\lambda H s(H)-\beta)+\frac{r}{s(H)}\bigg]
    \\ 
    &\equiv \frac{\Lambda(H,\alpha)}{Q(H,\alpha)},\\
    \frac{\dd \alpha}{\dd \tau}&=r.
    \label{eq: Ramped reduced problem}
    \end{aligned}
\end{equation}
Note that we now write $Q(H,\alpha)$ instead of $Q(H;\alpha)$ to indicate that $\alpha$ is no longer a parameter, but a dynamic variable in \eqref{eq: full 4D system}. Note also that the right-hand side in the $H$-equation in \eqref{eq: Ramped reduced problem} becomes singular when $Q(H,\alpha)=0$. We will refer to the curve defined by $Q(H,\alpha)=0$ as the fold curve $n_{0,\alpha}$, as illustrated in Figure~\ref{fig: full critical manifold}; specifically, we define
\begin{equation}
    n_{0,\alpha} = \big\{(H,A,C,\alpha)\in S^2_{0,\alpha}\,\big|\, \alpha  = \lambda s(H)^2\big[H+s(H)(1+H)^2\big]\big\}.
    \label{eq: fold curve}
\end{equation}

The singularity at $n_{0, \alpha}$ can be removed by rescaling time $\tau$ as follows: $\frac{\dd}{\dd s}=Q(H, \alpha)\frac{\dd}{\dd\tau}$, which transforms \eqref{eq: Ramped reduced problem} into 
\begin{equation}
    \begin{aligned}
        \frac{\dd H}{\dd s}&=\Lambda(H,\alpha),\\
        \frac{\dd\alpha}{\dd s}&=rQ(H, \alpha).
        \label{eq: Desingularised system}
    \end{aligned}
\end{equation}
Equation~\eqref{eq: Desingularised system}
is referred to as the ``desingularised" system \cite{SZMOLYAN2001419, Canards}. Note that the equilibria of \eqref{eq: Desingularised system} lie on the fold curve $n_{0,\alpha}$. These equilibria are referred to as ``folded singularities" and are of interest, as they potentially allow for canard trajectories to pass through them, depending on the character of the corresponding linearisation of \eqref{eq: Desingularised system}. We summarise the properties of the folded singularities of \eqref{eq: Desingularised system} as follows.
\begin{lem}\label{lem: desingularised system} For $r>0$ sufficiently small, the following statements hold for the folded singularities of \eqref{eq: Desingularised system}.
\begin{enumerate}
    \item For $H>0$, there generically exist two equilibria of \eqref{eq: Desingularised system}, which we write as the folded singularities $p_1=(H_{FS1,r},\alpha_{FS1,r})$ and $p_2=(H_{FS2,r},\alpha_{FS2,r})$.
    \item There holds $\min\{H_{FS1,r},H_{FS2,r}\}\uparrow \min\{\hat{H},H_{I}\}$ for $r\to 0^+$.
    \item There holds $\max\{H_{FS1,r},H_{FS2,r}\}\downarrow \max\{\hat{H},H_{I}\}$ for $r\to 0^+$.
    \item The $H$-value $H_{FS1,r}$ ($H_{FS2,r}$) is a regular perturbation of $\hat{H}$ ($H_{I}$). 
\end{enumerate}
Here, $H_{I}$ is the $H$-coordinate of the coexistence state $e_I$, recall \eqref{eq: coexistence point}, while $\hat{H}$ is the value of $H$ such that $\hat{\alpha}=\Pi(\hat{H})$, as defined in Lemma~\ref{lem: S_0^2 intersections}.
\end{lem}
\begin{proof}
See Appendix~\ref{sec: desingularised system proof}.
\end{proof}

The following result relates the regions in Figure~\ref{fig: minimum alpha plot} to the ordering between $H_I$ and $\hat{H}$. 
\begin{prop} There holds $H_I>\hat{H}$ ($H_I<\hat{H}$) if and only if $\alpha^+<\hat{\alpha}<\alpha^{\ast}$ ($\alpha^{\ast}<\alpha^+<\hat{\alpha}$).
\label{prop: ordering}
\end{prop}
\begin{proof}
The fold curve $n_{0,\alpha}$, which is a one-dimensional curve in the $4$-dimensional $(H,A,C,\alpha)$-space, may be parametrised by the single variable $H$. By Equation~\eqref{eq: fold curve}, we may write
 \begin{equation}
 \begin{split}
     n_{0,\alpha}=&\biggl\{(H,A(H,\alpha(H)),C(H,\alpha(H)),\alpha(H))\in S_{0,\alpha}^{2}\,\bigg|\,  A(H,\alpha(H))=\frac{\alpha(H)}{s(H)}, \\ & C(H,\alpha(H))=1-\frac{\alpha(H)}{s(H)}-\lambda Hs(H),\ \alpha(H)=\lambda s(H)^2(H+s(H)(1+H)^2\biggr\},
     \label{eq: fold curve param}
\end{split}
 \end{equation}
with $\alpha(H)$ as given in Equation~\eqref{eq: fold curve}. In particular, since $\alpha'(H)>0$, the function $\alpha(H)$ is monotone. Also, recall that $H_I$ is independent of $\alpha$; see Equation~\eqref{eq: coexistence point}. Hence, the question of whether $H_I>\hat{H}$ or $H_I<\hat{H}$ is equivalent to considering whether the coexistence state $e_{I,\alpha}$ has crossed $n_{0,\alpha}$ before the saddle-node bifurcation occurs at $\alpha=\hat{\alpha}$. Specifically, $H_I>\hat{H}$ if and only if $e_{I,\alpha}$ has not crossed $n_{0,\alpha}$ when $\alpha=\hat{\alpha}$, which forces $\hat\alpha<\alpha^\ast$. Furthermore, by Corollary~\ref{cor: hat bound}, $\alpha^+<\hat\alpha$, and we may conclude that $H_I>\hat{H}$ if and only if $\alpha^+<\hat{\alpha}<\alpha^{\ast}$.

Similarly, if $H_I<\hat{H}$, then the coexistence state $e_{I,\alpha}$ must have crossed $n_{0,\alpha}$ before $\hat\alpha$, which is equivalent to $\alpha^{\ast}<\hat\alpha$. Furthermore, the case where $\alpha^+<\alpha^{\ast}<\hat\alpha$ is excluded, as it would then follow that $e_{I,\alpha}$ crosses $n_{0,\alpha}$ in a point with $C<0$; however, the fold curve $n_{0,\alpha}$ only lies in $\{C<0\}$ when $\alpha>\hat\alpha$. In particular, $H_I<\hat{H}$ if and only if $\alpha^{\ast}<\alpha^+<\hat\alpha$. 
\end{proof}
We note that the regions where $H_I<\hat{H}$ and $\hat{H}<H_I$, respectively, in the $(\beta,\lambda)$-plane are, by necessity, identical to those where $\alpha^+<\hat{\alpha}<\alpha^{\ast}$ (yellow) and $\alpha^{\ast}<\alpha^+<\hat{\alpha}$ (purple) in Figure~\ref{fig: minimum alpha plot}, Furthermore $H_I=\hat{H}$ along the same bifurcation curve $\mathcal{C}$ as is given in Equation~\eqref{eq: bifurcation curve}.

With the existence of these folded singularities guaranteed and their relationship to the relevant $\alpha$-regimes established, we proceed to study their stability, which determines the potential for canard dynamics to occur in \eqref{eq: full 4D system}.

\begin{thm} Let $r>0$ be sufficiently small, and suppose that $\alpha^{+}<\hat\alpha<\alpha^\ast$, respectively that $\alpha^\ast<\alpha^{+}<\hat\alpha$. Then, the following statements hold for the two equilibria of \eqref{eq: Desingularised system}, which correspond to the folded singularities $p_1=(H_{FS1,r},\alpha_{FS1,r})$ and $p_2=(H_{FS2,r},\alpha_{FS2,r})$.
\begin{enumerate}
    \item The folded singularity $p_2$, respectively $p_1$, lies outside the ecologically relevant regime.
    \item The folded singularity $p_1$, respectively $p_2$, is attracting, and either a folded node or a folded focus. 
\end{enumerate} 
\label{thm: folded singularities}
\end{thm}
\begin{remark}
We note that $p_i$ ($i=1,2$) can remain in the positive quadrant in the $(H,\alpha)$-plane and still be ecologically irrelevant if the corresponding $C$-coordinate is negative.
\end{remark}

\begin{proof}
 To begin, we assume that $\alpha^{+}<\hat\alpha<\alpha^\ast$, which in turn implies that $\hat{H}<H_I$, as the coexistence state $e_I$ has not crossed the fold curve $n_{0,\alpha}$. From Lemma~\ref{lem: desingularised system}, we may conclude that we have the ordering
 \begin{align*}
     H_{FS1,r}<\hat{H}<H_I<H_{FS2,r}
 \end{align*}
for $r>0$ sufficiently small.
Furthermore, we recall that the fold curve $n_{0,\alpha}$ may be parametrised by the single variable $H$, with a parametrisation of the form $(H,A(H,\alpha(H)),C(H,\alpha(H)),\alpha(H))$, as given in \eqref{eq: fold curve param}.
 
In particular, we note that $n_{0,\alpha}$ only lies in the positive orthant for the four-dimensional system, Equation~\eqref{eq: full 4D system}, provided $C>0$ which, by Lemma~\ref{lem: S_0^2 intersections}, requires $\alpha(H)<\hat{\alpha}$. Furthermore, the unique value of $H$ for which $C(H,\alpha(H))=0$ when $\alpha(H)=\hat\alpha$ is $\hat{H}$. Also, one may directly show that $\frac{\partial C}{\partial \alpha}<0$ and $\frac{\partial \alpha(H)}{\partial H}>0$. In particular, one may conclude that for $H>\hat{H}$, $C(H,\alpha(H))<0$; since $H_{FS2,r}>\hat{H}$ and $S_{\varepsilon,\alpha}^2$ are regular perturbations of $H_I$ and $S_{0,\alpha}^2$, respectively, for $r$ and $\varepsilon$ positive and sufficiently small, it follows that the corresponding point in $(H,A,C,\alpha)$-space lies outside the positive orthant. Hence, $p_2$ is not ecologically relevant in that case.

For stability, we evaluate the Jacobian $J$ of \eqref{eq: Desingularised system} in $(H,\alpha)$ at the relevant folded singularity $p_1$,
\begin{align*}
    J = \begin{pmatrix}
        \frac{r}{s(H)^2}-\lambda \big[H+s(H)(1+H)^2\big]\big[1-2\lambda Hs(H)-\lambda s(H)^2(1+H)^2\big] && \frac{(\lambda H s(H)-\beta)(1+H)^2}{s(H)} \\
        2\lambda r\Big[2+(1+H)s(H)+\frac{H}{s(H)(1+H)^2}\Big] && -\frac{r}{s(H)^2}
        \label{Jacobian desingularised}
    \end{pmatrix},
\end{align*}
where we have used the fact that $Q(H,\alpha)=0=\Lambda(H,\alpha)$ at $p_1$ to simplify. The eigenvalues of $J$ are given by
\begin{align}
    \lambda_{\pm} = \frac{\mathrm{tr}J\pm\sqrt{(\mathrm{tr}J)^2-4\det{J}}}{2},
\end{align}
with
\begin{align*}
    \mathrm{tr}J = -\lambda \big[H+s(H)(1+H)^2\big]\big[1-2\lambda Hs(H)-\lambda s(H)^2(1+H)^2\big].
\end{align*}
Recall from the proof of Lemma~\ref{lem: desingularised system} in Appendix~\ref{sec: desingularised system proof} that $u(H)=1-2\lambda H s(H)-\lambda s(H)^2(1+H)^2$ is monotonically decreasing in $H$, while $v(H)=\lambda H s(H)-\beta$ is monotonically increasing, with $u(\hat{H})=0=v(H_I)$, respectively. By the ordering of the roots, with $H_{FS1,r}<\hat{H}<H_I$, we have
\begin{align*}
1-2\lambda H_{FS1,r}s(H_{FS1,r})-\lambda s(H_{FS1,r})^2(1+H_{FS1,r})^2>0\quad\text{and}\quad
\lambda H_{FS1,r}s(H_{FS1,r})-\beta<0.
\end{align*}

Since $H+s(H)(1+H^2)>0$ regardless, it follows that $\mathrm{tr}J<0$ when evaluated at $H=H_{FS1,r}$. To show stability, we must rule out the case of $p_1$ being a saddle, which follows if we can show that $\det{J}>0$. We have
\begin{align*}
    \det{J} =& r\bigg\{-\frac{\mathrm{tr}J}{s(H)^2}-2\lambda\bigg[2+(1+H)s(H)+\frac{H}{s(H)(1+H)^2}\bigg]\frac{(\lambda H s(H)-\beta)(1+H)^2}{s(H)}\bigg\}-\frac{r^2}{s(H)^4} 
\end{align*}
In particular, due to $\mathrm{tr}J$ and $\lambda Hs(H)-\beta$ both being negative at $H=H_{FS1,r}$ and $s(H_{FS1,r})\in[d,d+1]$, $\det J>0$ for $r$ sufficiently small; hence, the folded singularity $p_1$ is attracting, and either a node or focus. 

The case where $\alpha^\ast<\alpha^{+}<\hat\alpha$ follows analogously, with the exception that the ordering of the roots now is 
\begin{align*}
     H_{FS2,r}<H_I<\hat{H}<H_{FS1,r}.
\end{align*}
\end{proof}

Theorem~\ref{thm: folded singularities} establishes that in either of the above two $\alpha$-regimes, there is only one relevant folded singularity, a folded node or a folded focus, which we will denote by $p_{fn}$ and $p_{ff}$, respectively. The following result refines the statement of Theorem~\ref{thm: folded singularities}, allowing us to distinguish between these two scenarios.
\begin{prop} Let $r>0$ be sufficiently small, and suppose that $\alpha^{+}<\hat\alpha<\alpha^\ast$, respectively that $\alpha^\ast<\alpha^+<\hat\alpha$. Then, the folded singularity $p_1$, respectively, $p_2$, is a folded focus, respectively a folded node.
\label{prop: folded node/focus}
\end{prop}
\begin{proof}
As shown in Theorem~\ref{thm: folded singularities}, the folded singularities $p_1$ and $p_2$, respectively, are attracting in either regime; hence, to determine whether they are foci or nodes, we consider the discriminant
\begin{equation} 
\begin{split}
    \Delta(r)\equiv(\mathrm{tr}J)^2-4\det J =& \bigg\{\frac{2r}{s(H)^2}-\lambda\big[H+s(H)(1+H)^2\big]\big[1-2\lambda H s(H)-\lambda s(H)^2(1+H)^2\big]\bigg\}^2 \\
    &+8\lambda r \frac{(\lambda Hs(H)-\beta)(1+H)^2}{s(H)}\bigg[2+(1+H)s(H)+\frac{H}{s(H)(1+H)^2}\bigg].
    \label{eq: Discriminant}
\end{split}
\end{equation}
If $\alpha^{+}<\hat\alpha<\alpha^\ast$, we note that $\mathrm{tr}J$ is zero when evaluated at $\hat{H}$ and that $p_1=(H_{FS1,r},\alpha_{FS1,r})$ is the relevant singularity, where $H_{FS1,r}<\hat{H}$, with $\frac{\partial H_{FS1,r}}{\partial r}<0$; see again the proof of Lemma~\ref{lem: desingularised system} in Appendix~\ref{sec: desingularised system proof}. In particular, as $\mathrm{tr}J=0$ when $r=0$, it also follows that $\Delta(0)=0$ and, hence, that the sign of the discriminant for sufficiently small $r$ is determined by the derivative. Calculating the derivative of $\Delta(r)$ with respect to $r$ and evaluating at $\hat{H}$ for $r=0$, we have
\begin{equation}
    \frac{\partial \Delta}{\partial r}\Big\rvert_{r=0}=8\lambda \frac{\big(\lambda \hat{H}s(\hat{H})-\beta\big)\big(1+\hat{H}\big)^2}{s(\hat{H})}\Bigg[2+\big(1+\hat{H}\big)s(\hat{H})+\frac{\hat{H}}{s(\hat{H})\big(1+\hat{H}\big)^2}\Bigg]<0
    \label{eq: deriv discriminant}
\end{equation}
due to $\lambda \hat{H}s(\hat{H})-\beta<0$.
Therefore, as $\Delta(0)=0$, we may conclude that $\Delta(r)<0$ for $r$ sufficiently small and, hence, that the folded singularity $p_1$ is a focus.

For $\alpha^\ast<\alpha^+<\hat\alpha$, $p_2=(H_{FS2,r},\alpha_{FS2,r})$ is the folded singularity that is biologically relevant, where $H_{FS2,r}$ is a perturbation of $H_{I}$, with $\frac{\partial H_{FS2,r}}{\partial r}<0$; see Appendix~\ref{sec: desingularised system proof}. In particular, since $\Delta(0)=(\mathrm{tr}J)^2>0$ at $H=H_{FS2,0}(=H_I)$, it follows that $\Delta(r)$ will be positive for sufficiently small $r$; therefore, $p_2$ will be a folded node.
\end{proof}

\begin{remark}\label{rem: bifurcation curve}
We note that the choice of $r$ in Proposition~\ref{prop: folded node/focus} is not uniform in $\beta$, $\lambda$, and $d$. In particular, near the bifurcation curve $\mathcal{C}$ defined in \eqref{eq: bifurcation curve}, $\lambda \hat{H}s(\hat{H})-\beta$ tends to zero, which implies that \eqref{eq: deriv discriminant}, for instance, will degenerate as we approach $\mathcal{C}$.
\end{remark}

\begin{remark}\label{degeneracy}
The proof of Proposition~\ref{prop: folded node/focus} is consistent with the observation that Equation~\eqref{eq: full 4D system} can naturally be interpreted as a three-scale system for $r$ sufficiently small, with $(H,A)$ fast, $C$ slow, and $\alpha$ ``super-slow". Correspondingly, for $\alpha^+<\hat\alpha<\alpha^\ast$, $p_1$ is fully non-hyperbolic for $r=0$, with zero being a double eigenvalue of the Jacobian $J$ due to ${\rm tr}J=0=\det J$ then; similarly, for $\alpha^\ast<\alpha^+<\hat\alpha$ and $r=0$, $p_2$ degenerates to a folded saddle-node (of type II) due to ${\rm tr}J<0$ and $\det J=0$ \cite{MR2136520}.
\end{remark}

In Figure \ref{fig: desingularised systems}, we illustrate the dynamics of the desingularised system, Equation~\eqref{eq: Desingularised system}, in both $\alpha$-regimes. In particular, we confirm that the relevant folded singularity is a folded node when $\alpha^\ast<\alpha^{+}<\hat\alpha$ and a folded focus for $\alpha^{+}<\hat\alpha<\alpha^\ast$. As is well known \cite[Theorem $4.2$]{Canards}, the unfolding of a folded node in singularly perturbed systems of fast-slow type gives rise to robust canard phenomena, which include subthreshold oscillation; however, a folded focus can admit no canard trajectories \cite[Corollary~3.1]{Canards}. That distinction will be significant to our discussion of rate-induced tipping in the following section.

\begin{figure}[h]
    \centering
    \begin{subfigure}[t]{0.49\textwidth}
        \centering
        \includegraphics[width=\textwidth]{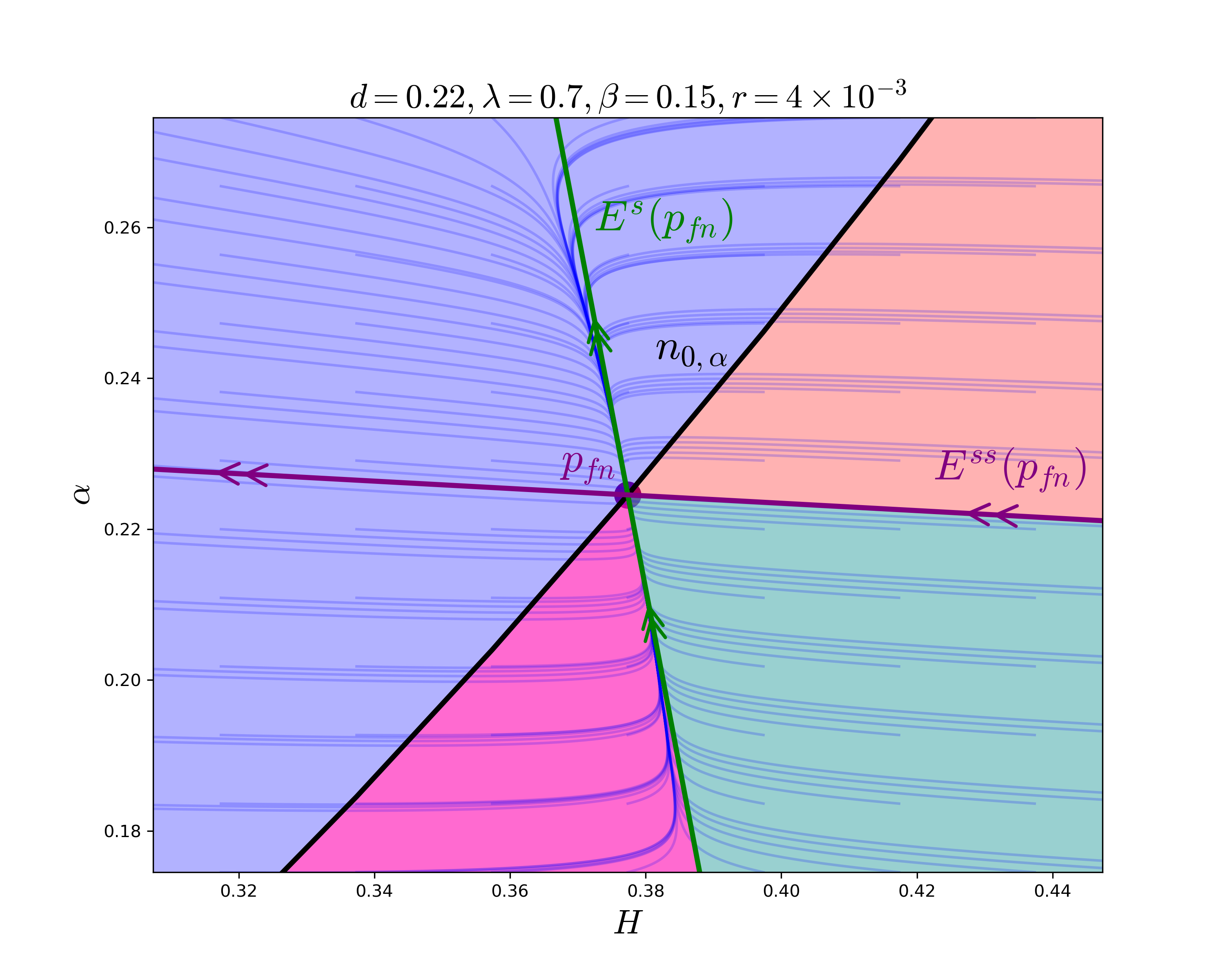}
        \caption{$\alpha^\ast<\alpha^{+}<\hat\alpha$ ($H_I<\hat{H}$)}
        \label{fig:first desing}
    \end{subfigure}
    \hfill
    \begin{subfigure}[t]{0.49\textwidth}
        \centering
        \includegraphics[width=\textwidth]{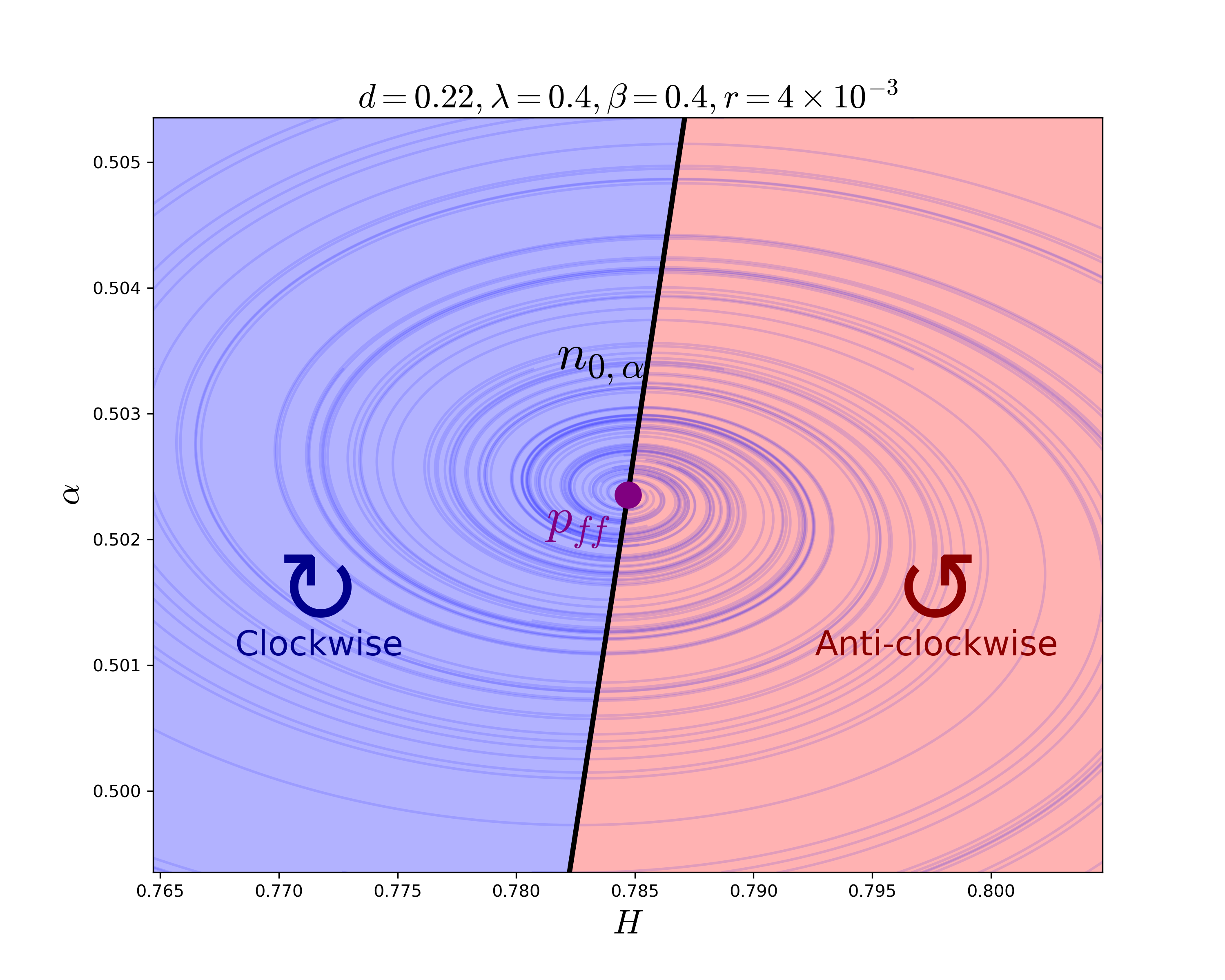}
        \caption{$\alpha^{+}<\hat\alpha<\alpha^\ast$ ($H_I>\hat{H}$)}
        \label{fig:second desing}
    \end{subfigure}
    \caption{Dynamics of the desingularised system, Equation~\eqref{eq: Desingularised system}, in a neighbourhood of the fold curve $n_{0,\alpha}$ (in black). Here, $d=0.22$ and $r=4\times 10^{-3}$, with (a) $\beta=0.15$ and $\lambda=0.7$ and (b) $\beta=0.4=\lambda$. Correspondingly, $\alpha^\ast<\alpha^{+}<\hat\alpha$ in panel (a), which implies the folded singularity $p_{fn}$ is a folded node. The corresponding strong (weak) eigendirection $E^{ss}(p_{fn})$ ($E^{s}(p_{fn})$) is indicated in purple (green); the region on $S^2_{0,\alpha}$ that lies between $n_{0,\alpha}$ and $E^{s}(p_{fn})$ is indicated in magenta, whereas the region that is bounded by the two eigendirections is marked in cyan. Arrows indicate the direction of the flow in time; note the reversal in direction between the attracting portion of $S_{0,\alpha}^2$ (red) and the repelling portion (blue) in the rescaled time $s$. In panel (b), $\alpha^{+}<\hat\alpha<\alpha^\ast$, which implies that the folded singularity $p_{ff}$ is a folded focus.}
    \label{fig: desingularised systems}
\end{figure}

\subsection{Dynamical classification}
From Proposition~\ref{prop: folded node/focus}, we may conclude that, for $r$ sufficiently small, the folded singularity is a folded node (folded focus) when $\alpha^\ast<\alpha^{+}<\hat\alpha$ ($\alpha^{+}<\hat\alpha<\alpha^\ast$). However, we cannot guarantee that this classification persists with increasing $r$. For illustration, we numerically classify the relevant folded singularity for two different values of $r$ in Figure~\ref{fig: bifurcation regions}. (We recall that we have excluded those parts of the $(\beta,\lambda)$-plane in Figure~\ref{fig: bifurcation regions} in which the coexistence state $e_{I,\alpha}$ is either repelling, or is not inside the positive orthant when $\alpha=d$, thus ensuring that it is both relevant and attracting initially.)

We define three different parameter regimes, as indicated in Figure~\ref{fig:set diff}: Region $\mathrm{I}$ (in yellow), wherein the relevant singularity remains a folded node for $r$ positive and sufficiently small; Region $\mathrm{II}$ (in green), wherein the singularity has changed from a folded focus to a folded node; and, finally, Region $\mathrm{III}$ (in purple), wherein the singularity remains a folded focus. 

The boundary between Regions $\mathrm{I}$ and $\mathrm{II}$ again agrees with the bifurcation curve $\mathcal{C}$ defined in \eqref{eq: bifurcation curve}, which follows from the fact that $\Delta(0)=(\mathrm{tr}J)^2$: for $\alpha^+<\hat{\alpha}<\alpha^{\ast}$ ($\alpha^{\ast}<\alpha^{+}<\hat{\alpha}$), the relevant folded singularity is a regular perturbation of $\hat{H}$ ($H_I$); hence, $\Delta(0)=0$ in the former regime, whereas $\Delta(0)>0$ in the latter if we assume that $H_I$ and $\hat{H}$ are distinct. Correspondingly, on the boundary, $\Delta(0)=0$ if and only if $H_I=\hat{H}$, which occurs precisely along the bifurcation curve $\mathcal{C}$; recall Proposition~\ref{prop: ordering}. Figure~\ref{fig:r small} indicates that the boundary between Regions $\mathrm{II}$ and $\mathrm{III}$ is still well approximated by $\mathcal{C}$ while $r$ is sufficiently small. However, the quality of that approximation deteriorates as $r$ increases; cf. Figure~\ref{fig:set diff} as well as the discussion in Remark~\ref{rem: bifurcation curve}.

\begin{figure}[h]
    \centering
    \begin{subfigure}[t]{0.48\textwidth}
        \centering
        \includegraphics[width=\textwidth]{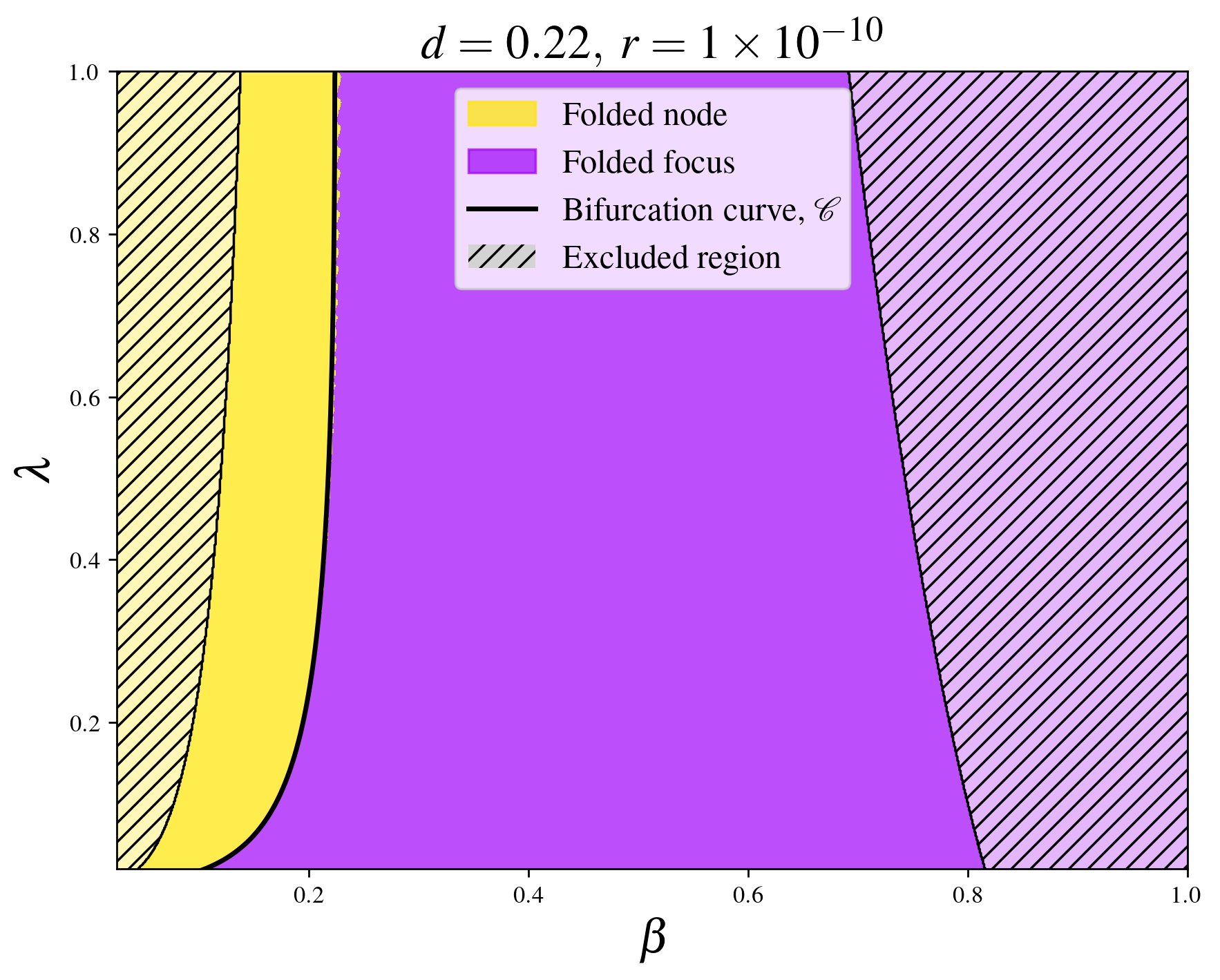}
        \caption{$r=1\times10^{-10}$}
        \label{fig:r small}
    \end{subfigure}
    \hfill
    \begin{subfigure}[t]{0.48\textwidth}
        \centering
        \includegraphics[width=\textwidth]{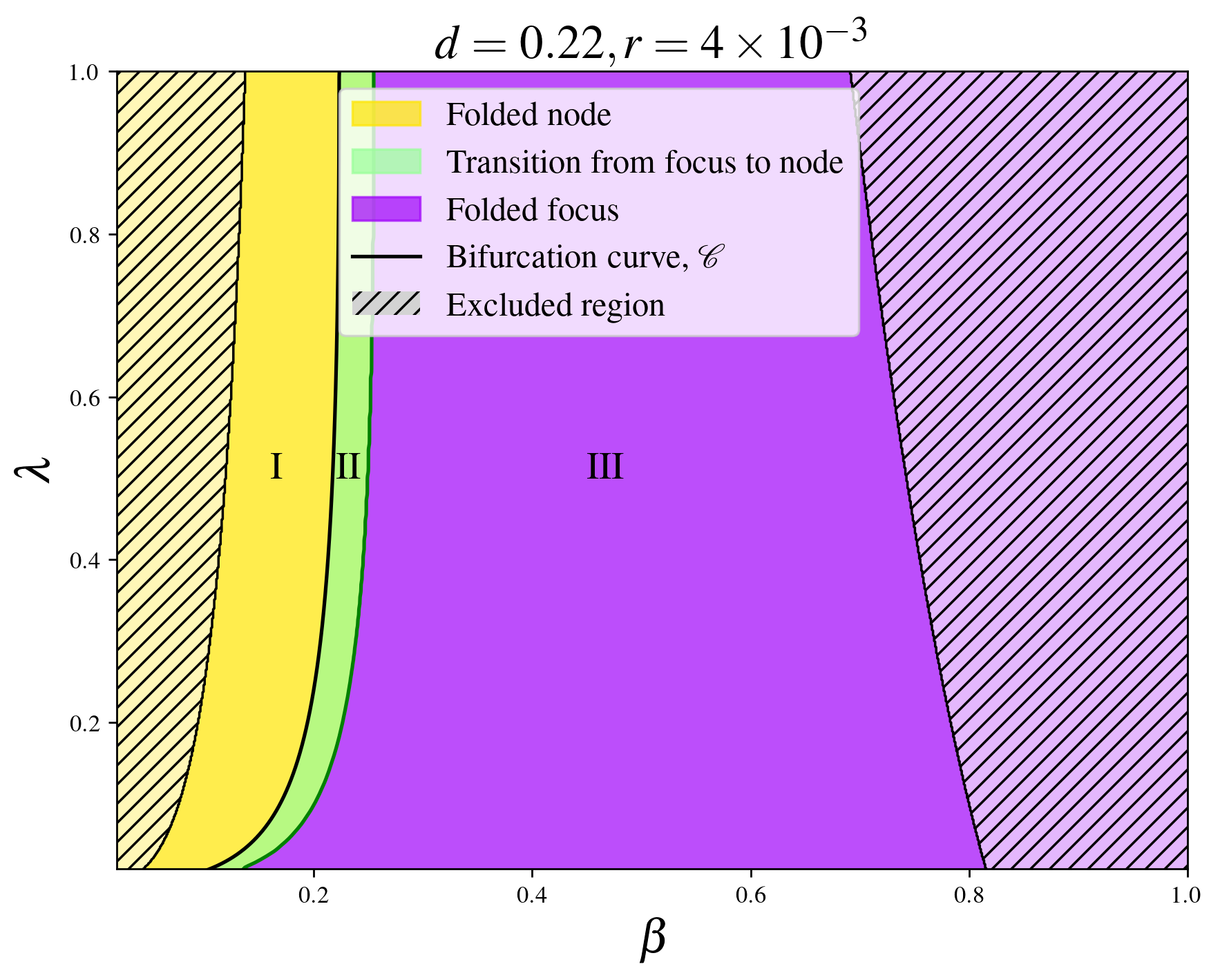}
        \caption{$r=4\times10^{-3}$}
        \label{fig:set diff}
    \end{subfigure}
    \caption{Illustration of the classification between folded nodes and folded foci for $d=0.22$ and (a) $r=1\times 10^{-10}$ and (b) $r=4\times 10^{-3}$. For both values of $r$, the region where the folded singularity is a node (focus) is highlighted in yellow (purple), with the bifurcation curve $\mathcal{C}$ given by Equation~\eqref{eq: bifurcation curve} marked in black. Panel (b) highlights the region where the classification of the folded singularity changes from a folded focus to a folded node (green).}
    \label{fig: bifurcation regions}
\end{figure}

To classify the dynamics of Equation~\eqref{eq: full 4D system} in these various parameter regimes, we now consider Regions $\mathrm{I}$, $\mathrm{II}$, and $\mathrm{III}$ in turn.

\begin{remark}
As \eqref{eq: full 4D system} cannot be solved analytically, we resort to numerical integration. Stability requirements for explicit time stepping methods for multi-scale systems of the type in Equation~\eqref{eq: full 4D system}, with $0<\varepsilon\ll 1$, require the time step to scale as $\mathcal{O}(\varepsilon)$, which may be restrictive if one wants to examine their long-term dynamics. In general, one should use an implicit or A-stable method to circumvent these restrictions for stiff problems \cite{MR2478556}. We shall make use of a variable-order ``Backward Differentiation Formula (BDF)" algorithm, as implemented in SciPy \cite{doi:10.1137/S1064827594276424,2020SciPy-NMeth}, to integrate \eqref{eq: full 4D system} numerically.
\end{remark}

\subsubsection{Region \texorpdfstring{$\mathrm{I}$}{I}}
\label{sec: canard 1}
We begin by illustrating rate-induced tipping in Region $\mathrm{I}$ of Figure~\ref{fig:set diff}, which corresponds to the regime where $\alpha^\ast<\alpha^{+}<\hat\alpha$ when $r=0$, in which case the relevant folded singularity $p_2(=p_{fn})$ in Equation~\eqref{eq: Desingularised system} is a folded node; we recall that, correspondingly, we increase $\alpha$ with rate $r$ in Region $\mathrm{I}$ until it reaches $\alpha_{\max,\delta}=\alpha^{\ast}-\delta$. Throughout, we fix $d=0.22$, $\varepsilon=0.01$, and $\delta=0.01$.

\paragraph{Canard-induced tipping}
We first choose $\beta=0.18$ and $\lambda=0.5$, as well as $r=4\times10^{-3}$. A representative trajectory is presented in Figure~\ref{fig:canards in other regime}, where the flow is initiated at the coexistence state $e_{I,\alpha}$ with $\alpha=d+\delta$. One would expect the flow to track $e_{I,\alpha}$ as the fishing effort $\alpha$ is increased. However, Figure~\ref{fig:canards in other regime} indicates ``tipping" from that state to the algae-only equilibrium $e_{A,\alpha}$. In the process, the flow transits from the attracting portion $S^{2,a}_{\varepsilon,\alpha}$ of the slow manifold $S^2_{\varepsilon,\alpha}$ onto the repelling portion $S^{2,r}_{\varepsilon,\alpha}$ via a folded-node induced canard trajectory in the vicinity of the folded singularity $p_{fn}$. It is then attracted to the slow manifold $S_{\varepsilon,\alpha}^1$ before settling into the equilibrium $e_{A,\alpha}$. Ecologically, we hence observe a catastrophic collapse in the populations of both herbivorous fish and coral, which indicates that rate-induced tipping triggers algal bloom in the underlying coral reef ecosystem. 

\begin{figure}
    \centering
     \includegraphics[scale = 0.2]{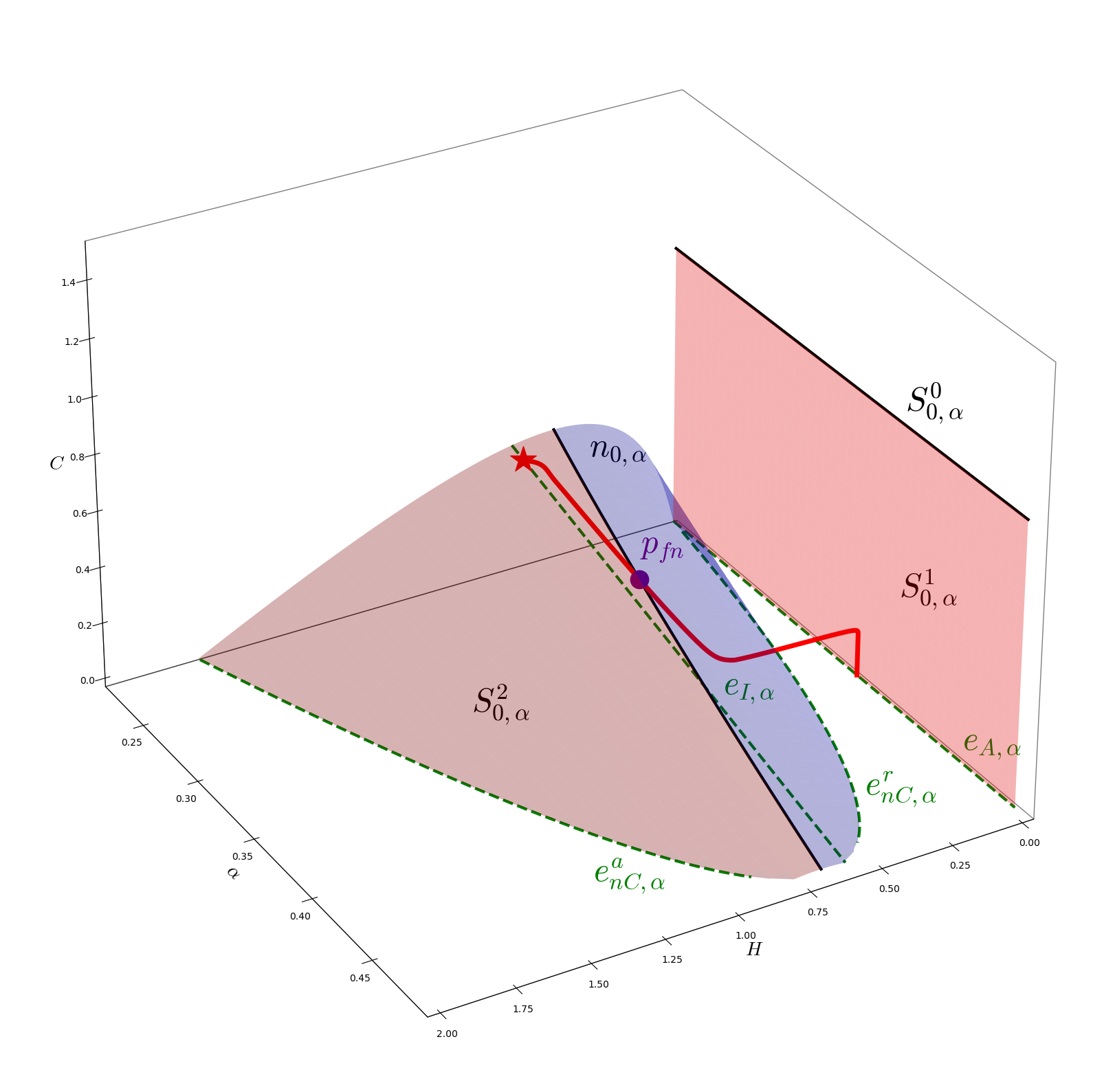}
    \includegraphics[scale = 0.2]{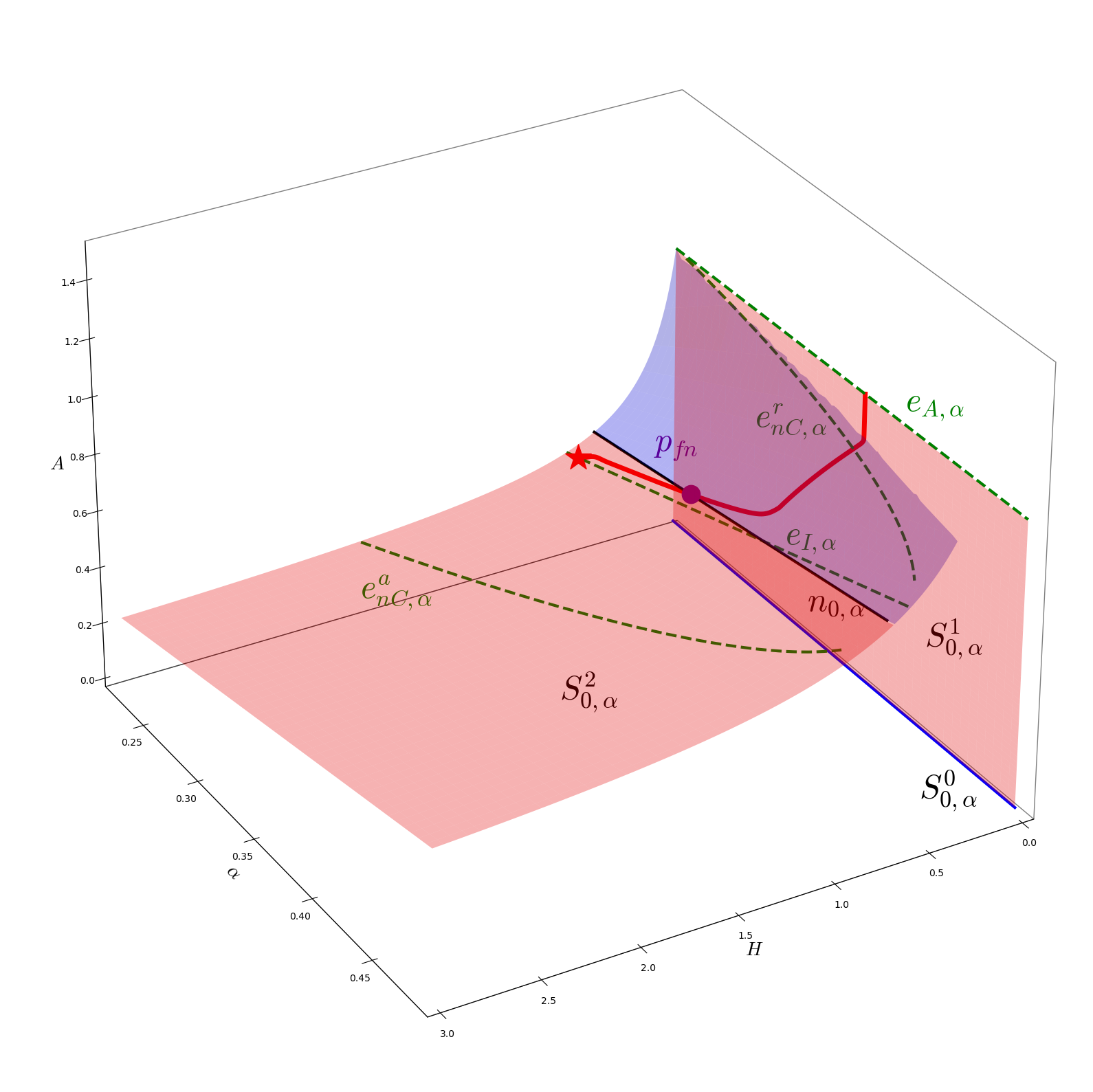}
    \caption{Canard-induced R-tipping in Region $\mathrm{I}$, with $\beta=0.18$, $\lambda = 0.5$, and $d=0.22$ in the ``ramped" Equation~\eqref{eq: full 4D system}, for $\alpha\in [d+\delta,\alpha^{\ast}-\delta]$. Here, we have taken $\varepsilon=0.01$ and $r = 4\times10^{-3}$. The red star represents the initial condition, chosen such that the trajectory originates at the stable coexistence state $e_{I,\alpha}$, with $\alpha = d+\delta$ initially.}
    \label{fig:canards in other regime}
\end{figure}

We emphasise that the trajectory illustrated in Figure~\ref{fig:canards in other regime} undergoes canard-induced R-tipping, as it is initiated in a regime where Equation~\eqref{eq: full 4D system} is bistable, with both $e_{I,\alpha}$ and $e_{A,\alpha}$ being attracting equilibria due to the restriction that $\alpha\leq\alpha_{\max,\delta}=\alpha^{\ast}-\delta$. (While Figure~\ref{fig:canards in other regime} does indicate that $e_{I,\alpha}$ loses stability soon after the tipping point, that classical bifurcation plays no role in the dynamics described above.) With reference to Remark~\ref{remark: tipping vs tracking}, Figure~\ref{fig:canards in other regime} clearly demonstrates tipping as opposed to tracking. 

\begin{remark}
We note that, as we only integrate in $\alpha$ up to $\alpha^{\ast}-\delta(=\alpha_{\max,\delta})$, for extremely small $r$ it is possible to not observe tipping numerically in Region $\mathrm{I}$, as the folded singularity $p_2$ would be such that $\alpha_{FS2,r}\in(\alpha^{\ast}-\delta,\alpha^{\ast})$. Hence, the ``ramped" Equation~\eqref{eq: full 4D system} would reduce to \eqref{fast-system} with $\alpha=\alpha_{\max,\delta}$ before tipping can occur, resulting in convergence to the coexistence state $e_{I,\alpha_{\max,\delta}}$; recall again Remark~\ref{remark: tipping vs tracking}.
\end{remark}

\paragraph{Subthreshold oscillation}

As is well known \cite{MR2916308,MR2136520}, folded nodes permit richer canard dynamics than is illustrated above; in particular, they admit subthreshold oscillation which may give rise to mixed-mode oscillatory patterns. The occurrence of subthreshold oscillation may be explained by the twisting of the slow manifolds  $S^{2,a}_{\varepsilon,\alpha}$ and $S^{2,r}_{\varepsilon,\alpha}$ about the ``weak canard" which perturbs off the weak eigendirection of the linearisation of \eqref{eq: Desingularised system} at $p_{fn}$ for $\varepsilon>0$; recall Figure~\ref{fig:first desing}. Intersections between these twisted manifolds generate $\lfloor \frac{\mu-1}{2}\rfloor$ ``secondary" canards, with $\mu>1$ denoting the ratio of the eigenvalues of the linearisation about $p_{fn}$. These secondary canards are observed in the form of subthreshold oscillations, the maximum number of which is equally given by $\lfloor \frac{\mu-1}{2}\rfloor$, and subdivide the ``funnel" region that is associated to the folded node singularity at $p_{fn}$: recalling again Figure~\ref{fig:first desing}, one would, to leading order, expect trajectories initiated in the magenta region between the fold curve $n_{0,\alpha}$ and the weak eigendirection at $p_{fn}$ to follow the weak canard, whereas trajectories initiated in the cyan region that is bounded by the weak and strong eigendirections undergo subthreshold oscillation about the weak canard; see again \cite{MR2916308,MR2136520} for details.

An illustration can be found in Figure~\ref{fig:R tipping oscillations}, where the parameters are chosen as $\beta=0.15$ and $\lambda = 0.5$, with $r=4\times 10^{-3}$, as above. The flow of \eqref{eq: full 4D system} is again initiated at the attracting coexistence state $e_{I,\alpha}$ for $\alpha=d+\delta$; furthermore, that initial condition is located within the cyan region ilustrated in Figure~\ref{fig:first desing}. With these choices, we obtain $\mu\approx 20.4255$ for the ratio of the eigenvalues of the linearisation at $p_{fn}$, corresponding to a maximum of $\lfloor \frac{20.4255-1}{2}\rfloor=9$ subthreshold oscillations, which is consistent with the trajectory shown in Figure~\ref{fig:R tipping oscillations}. We observe that canard-induced tipping still occurs as before, after the trajectory has undergone subthreshold oscillation; the long-term dynamics hence remains unchanged.

\begin{figure}[h]
    \centering
     \includegraphics[scale = 0.2]{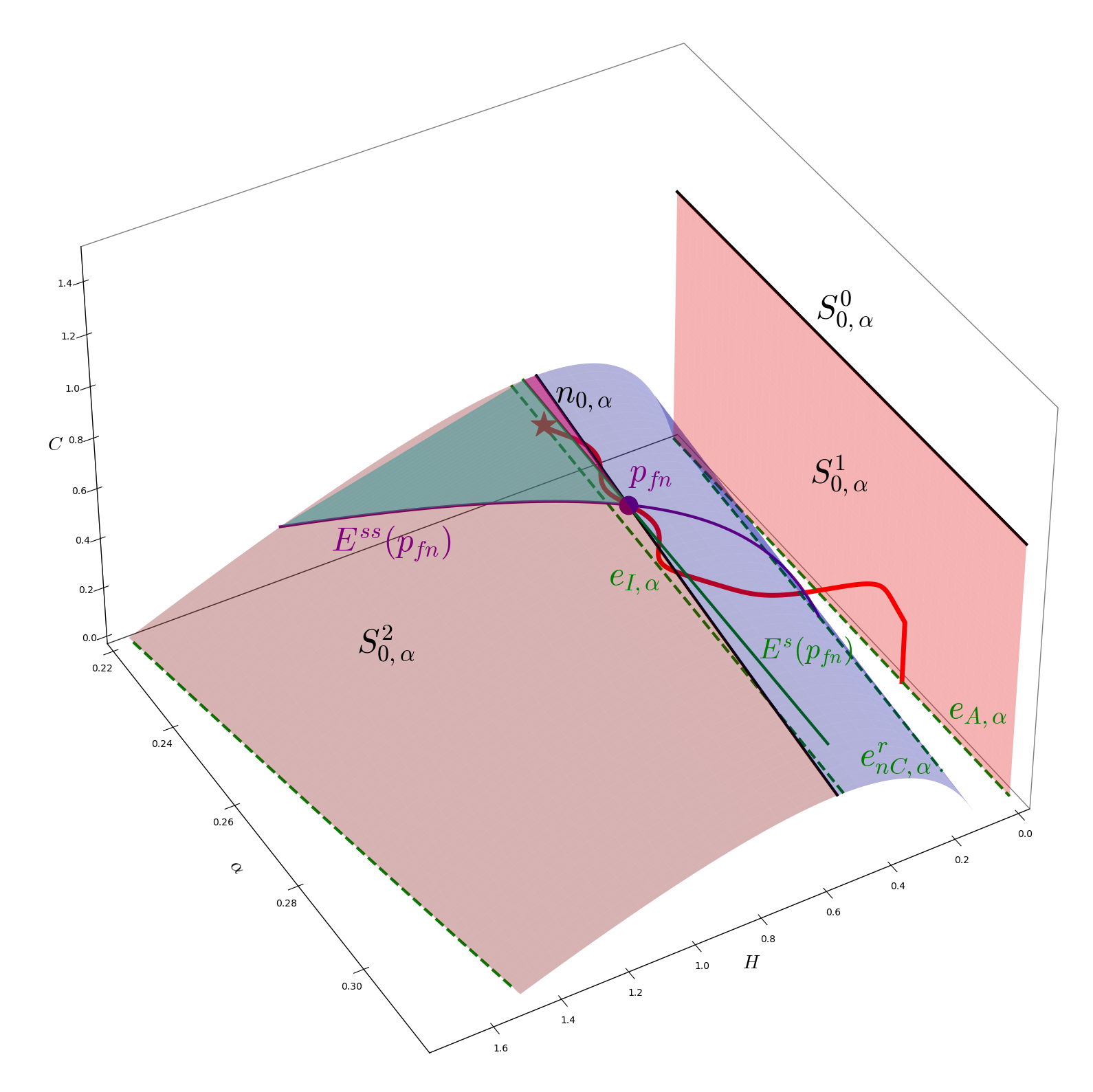}
    \includegraphics[scale = 0.3]{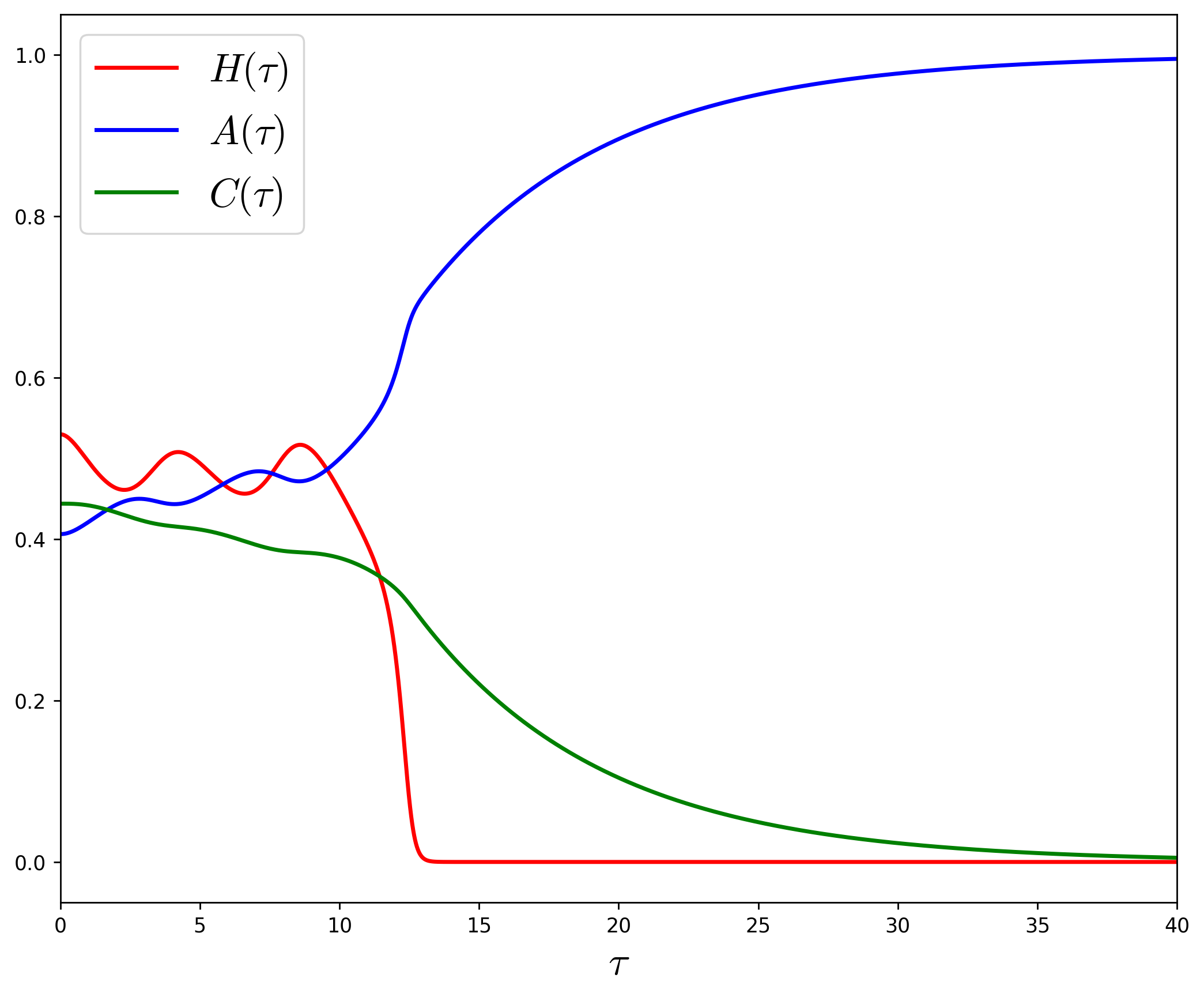}
    \caption{Canard-induced R-tipping in Region $\mathrm{I}$, with the trajectory undergoing subthreshold oscillation before tipping. The red star again represents the initial condition, chosen such that the trajectory originates at the stable coexistence state $e_{I,\alpha}$, with $\alpha = d+\delta$ initially. The strong (weak) eigendirection $E^{ss}(p_{fn})$ ($E^{s}(p_{fn})$) is indicated in solid purple (green), as before. As in Figure~\ref{fig:first desing}, the region on the attracting portion $S_{0,\alpha}^{2,a}$ that lies between the two eigendirections is indicated in cyan, whereas the region between the fold curve $n_{0,\alpha}$ and $E^{s}(p_{fn})$ is marked in magenta. The parameter values are given by $\beta=0.15$, $\lambda=0.5$, $d=0.22$, and $r=4\times10^{-3}$. The time series of the canard trajectory is also illustrated (right panel).}
    \label{fig:R tipping oscillations}
\end{figure}

From an ecological perspective, the occurrence of subthreshold oscillation in tipping trajectories indicates that a critical collapse of fish and coral populations may not be monotone. In fact, the non-monotone collapse seen in Figure~\ref{fig:R tipping oscillations} can potentially be viewed as an early warning signal for rate-induced tipping: by default, trajectories are expected to track the coexistence state $e_{I, \alpha}$, the entries $(H_{I}, A_{I}, C_{I})$ of which, as defined in \eqref{eq: coexistence point}, are all monotone in $\alpha$ and, therefore, monotone in time. 

\paragraph{Delayed Hopf bifurcation}

For $r$ sufficiently small, the dynamics of Equation~\eqref{eq: full 4D system} evolves on three distinct time scales, with $(H,A)$ the fast variables, $C$ the slow one, and $\alpha$ the ``super-slow" one; correspondingly, both $\varepsilon$ and $r$ are interpreted as singular perturbation parameters then. In the singular limit of $\varepsilon=0$, \eqref{eq: full 4D system} reduces to the ``partially perturbed" planar system
\begin{align}\label{part pert}
\begin{split}
\frac{\dd C}{\dd \tau} &= C(1-\beta-A-C),\\
    \frac{\dd \alpha}{\dd \tau} &=\begin{cases}
        r & \quad\text{for } \alpha_{\min,\delta}< \alpha < \alpha_{\max,\delta},\\
        0 & \quad \mathrm{otherwise},
        \end{cases}
\end{split}
\end{align}
which is a $(1,1)$ fast-slow system (in $r$) that is defined on the critical manifold $S_{0,\alpha}$. For $r=0$, \eqref{part pert} admits a so-called ``2-critical" manifold which corresponds precisely to the (extended) coexistence state $e_{I,\alpha}$. Figure~\ref{fig:R tipping oscillations DH} indicates that, at least in the parameter regime considered there, $e_{I,\alpha}$ is focally attracting and that the ``super-slow" flow thereon undergoes a Hopf bifurcation at which it loses stability; for $r=0$, the bifurcation point coalesces with $p_{fn}$ in a ``canard delayed Hopf singularity" in the double singular limit of $\varepsilon=0=r$. Details and further references can be found in \cite{6c8c3bfcb06e4fb8b0cd94ae49d0c15b}.  As is evident from Figure~\ref{fig:R tipping oscillations DH}, the resulting time series exhibit typical ``delayed Hopf"-type behaviour, in that dense subthreshold oscillation occurs with initially decreasing and then increasing amplitude. While that dynamics is interesting from a mathematical perspective, it appears less relevant ecologically, since trajectories that are initiated near the coexistence state $e_{I,\alpha}$, i.e., in the magenta region on the attracting portion $S_{0,\alpha}^{2,a}$ of $S_{0,\alpha}^2$ in Figure~\ref{fig:R tipping oscillations DH}, will typically tip without undergoing subthreshold oscillation.

\begin{figure}[h]
    \centering
     \includegraphics[scale = 0.2]{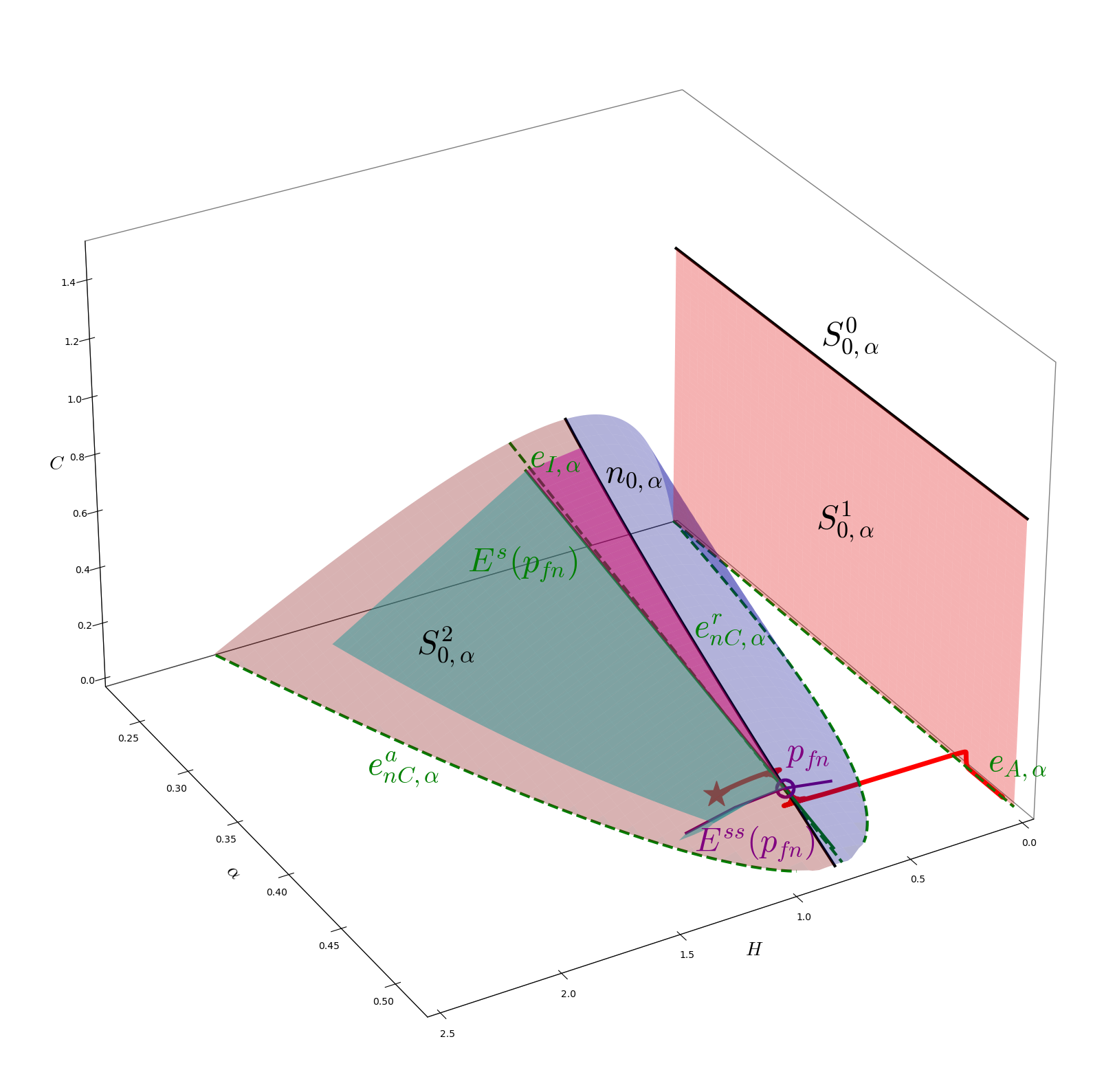}
    \includegraphics[scale = 0.3]{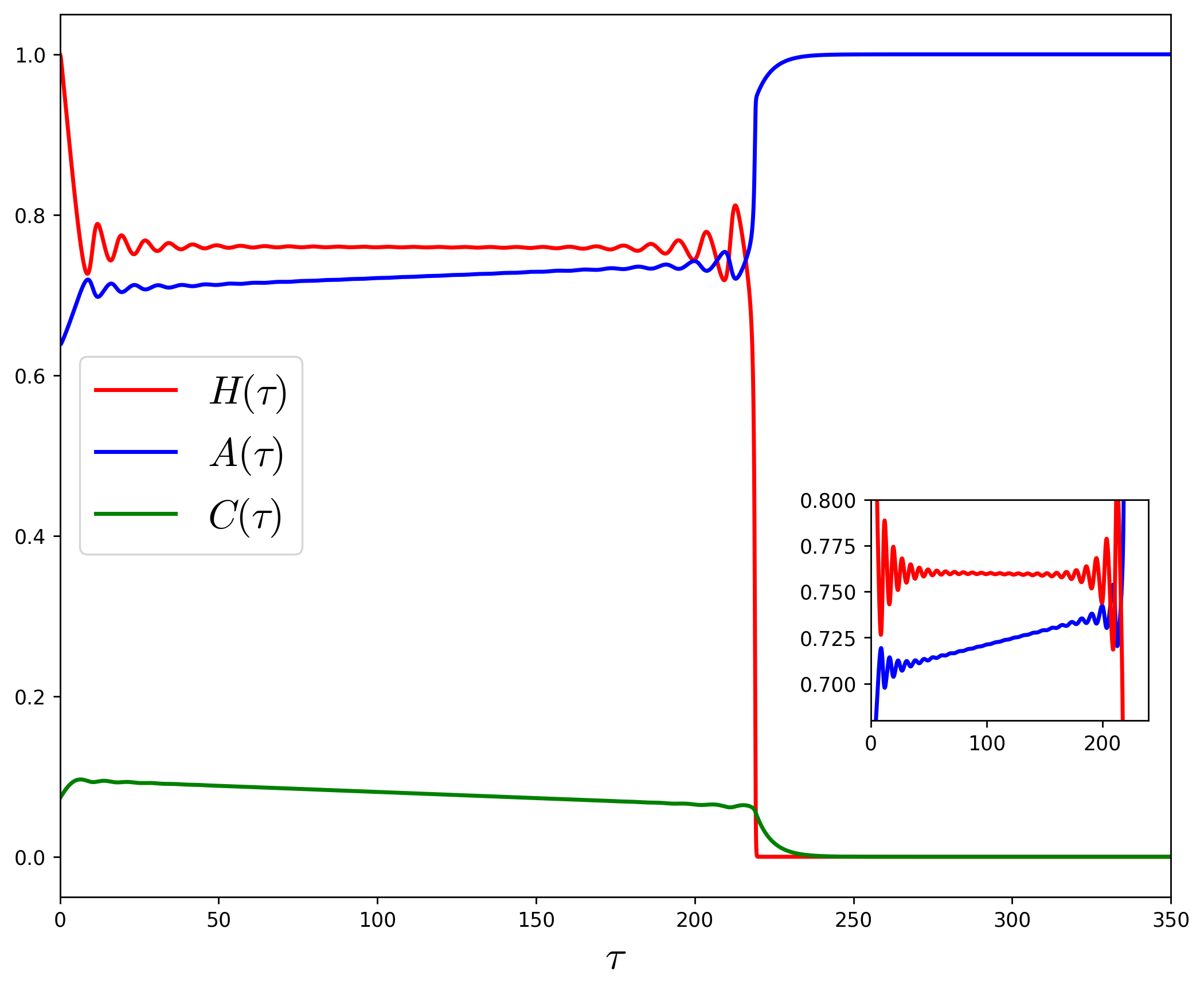}
    \caption{Canard-induced R-tipping in Region $\mathrm{I}$, with the trajectory exhibiting delayed Hopf-type behaviour before tipping. The red star again represents the initial condition, chosen such that the trajectory originates at the stable coexistence state $e_{I,\alpha}$, with $\alpha = d+\delta$ initially. The strong (weak) eigendirection $E^{ss}(p_{fn})$ ($E^{s}(p_{fn})$) is indicated in solid purple (green), as before; the region on the attracting portioin $S_{0,\alpha}^{2,a}$ that lies between the two eigendirections is indicated in cyan, whereas the region between the fold curve $n_{0,\alpha}$ and $E^{s}(p_{fn})$ is marked in magenta. The parameter values are given by $\beta=0.2$, $\lambda=0.4$, $d=0.22$, and $r=1\times10^{-4}$. The time series of the canard trajectory is also illustrated (right panel).}
    \label{fig:R tipping oscillations DH}
\end{figure}

\subsubsection{Region \texorpdfstring{$\mathrm{II}$}{II}}

Region $\mathrm{II}$ is defined as the extension of Region $\mathrm{I}$ for $r>0$ sufficiently small: Proposition~\ref{prop: folded node/focus} implies that, if $\alpha^\ast<\alpha^+<\hat\alpha$ for $r=0$, i.e., if $p_2$ is a folded node, it remains so under small perturbations in $r$. Within Region $\mathrm{I}$, we have $\mathrm{tr}J<0$ as well as $\frac{\partial\mathrm{tr}J}{\partial r}<0$, evaluated at the folded singularity $p_2$; hence, one should not expect the classification to change for $r$ sufficiently small. The regime where $\alpha^+<\hat{\alpha}<\alpha^\ast$ is more subtle, as $\mathrm{tr}J=0$ for $r=0$ in that case, which implies that the classification will depend on the higher-order asymptotics of $\Delta(r)$ in $r$. (In particular, while the coefficient $\lambda\hat{H}s(\hat{H})-\beta$ in \eqref{eq: deriv discriminant} is negative, it can be made arbitrarily small as $H_I$ approaches $\hat{H}$, i.e., near the bifurcation curve \eqref{eq: bifurcation curve}, which suggests that one should consider $\frac{\partial^2}{\partial r^2}\Delta(r)$ for a correct classification.) 

Validating numerically, we find that Region $\mathrm{II}$ has non-zero measure for $r$ positive and sufficiently small, with the classification changing from a focus to a node near the bifurcation curve $\mathcal{C}$ as $r$ increases; see Figure~\ref{fig:set diff}. Correspondingly, it vanishes as $r\to 0^+$, as is also evident from a comparison of panels (a) and (b) in Figure~\ref{fig: bifurcation regions}.

Ecologically speaking, it is reasonable to assume that the parameters $\beta$, $\lambda$, and $d$ are fixed in a given reef ecosystem, whereas $r$ can be varied. Correspondingly, one is interested in the critical rate $r_{\rm crit}$ which separates the scenario where canard-induced tipping is possible, with $p_1$ a folded node for $r$ sufficiently large, from the one where $p_1$ is a folded focus for $r$ sufficiently small, precluding the presence of canard trajectories; in other words, $r_{\rm crit}$ defines the boundary between Regions $\mathrm{II}$ and $\mathrm{III}$ for that choice of parameters. For illustration, we will choose $\beta=0.2=\lambda$ here, which lies in Region $\mathrm{II}$, with $d=0.22$, as before. 

\begin{lem}
Let $\beta=0.2=\lambda$ and $d=0.22$. Then, there exists a critical rate $r_{\rm crit}\approx 4.6602\times 10^{-6}$ such that the relevant folded singularity in \eqref{eq: Desingularised system} is a folded focus for $0<r<r_{\rm crit}$ and a folded node for $r>r_{\rm crit}$ sufficiently small. 
\end{lem}
\begin{proof}
The rate $r_{\rm crit}$ is calculated numerically by solving $\Delta(r)=0$ in Equation~\eqref{eq: Discriminant} for $r$, with $(\beta,\lambda,d)$ as given. 
\end{proof}
To confirm the criticality of $r_{\rm crit}$ numerically, we first consider $\beta=0.2=\lambda$, $d=0.22$, and $r_{\rm crit}<r=4\times 10^{-3}$; see Figure \ref{fig:R-tipping}. The flow is again initiated at the coexistence equilibrium $e_{I,\alpha}$, with $\alpha=d+\delta$. As the relevant folded singularity is a node, the parameter triple $(\beta,\lambda,d)$ must be located in Region $\mathrm{II}$; correspondingly, canard-induced tipping will be observed. The dynamics is qualitiatively similar to that seen in Section~\ref{sec: canard 1}, with the important distinction that the coexistence state $e_{I,\alpha}$ does not also lose stability soon after the trajectory has tipped.

\begin{figure}[h]
    \centering
     \includegraphics[scale = 0.2]{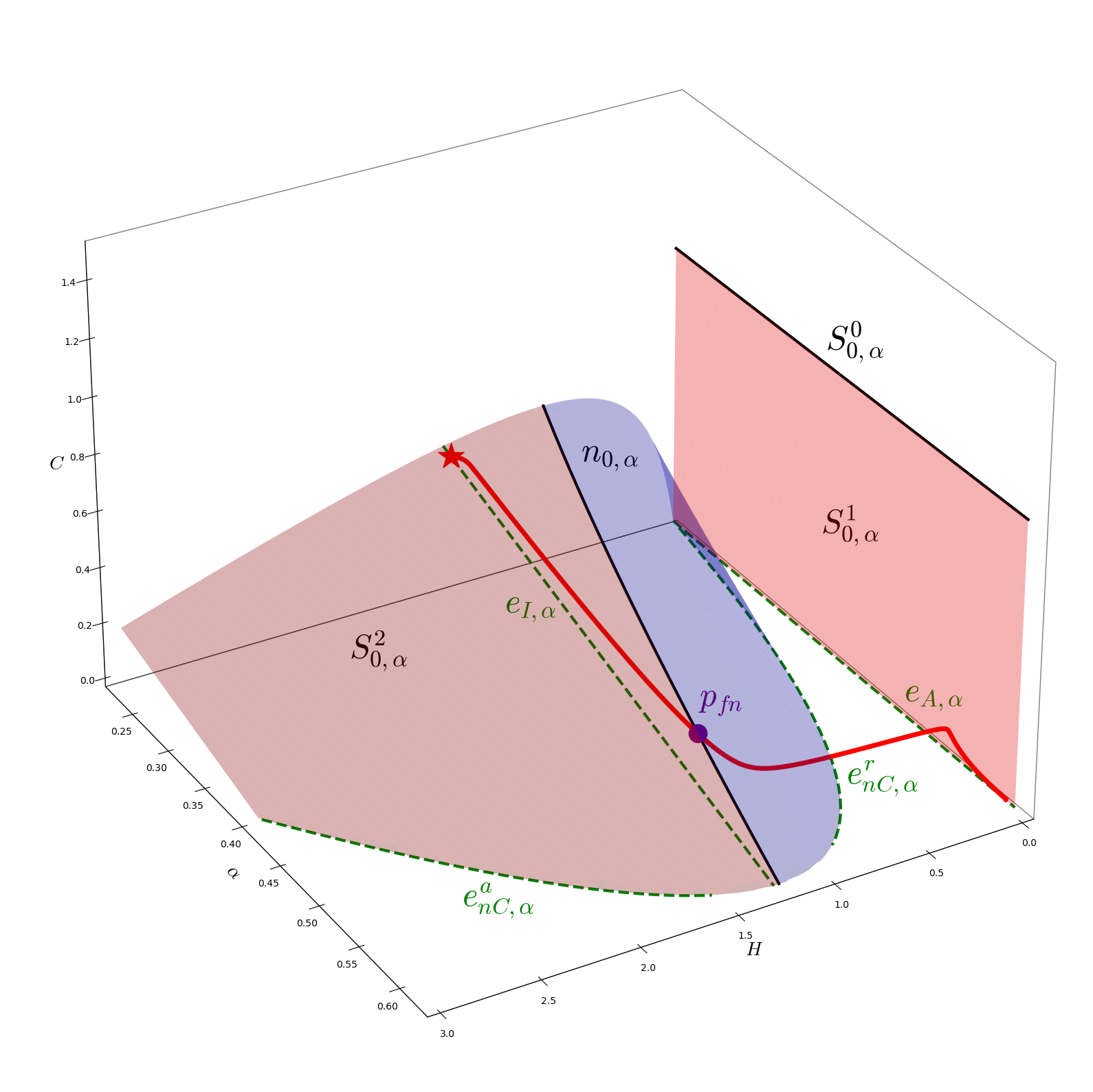}
    \includegraphics[scale = 0.2]{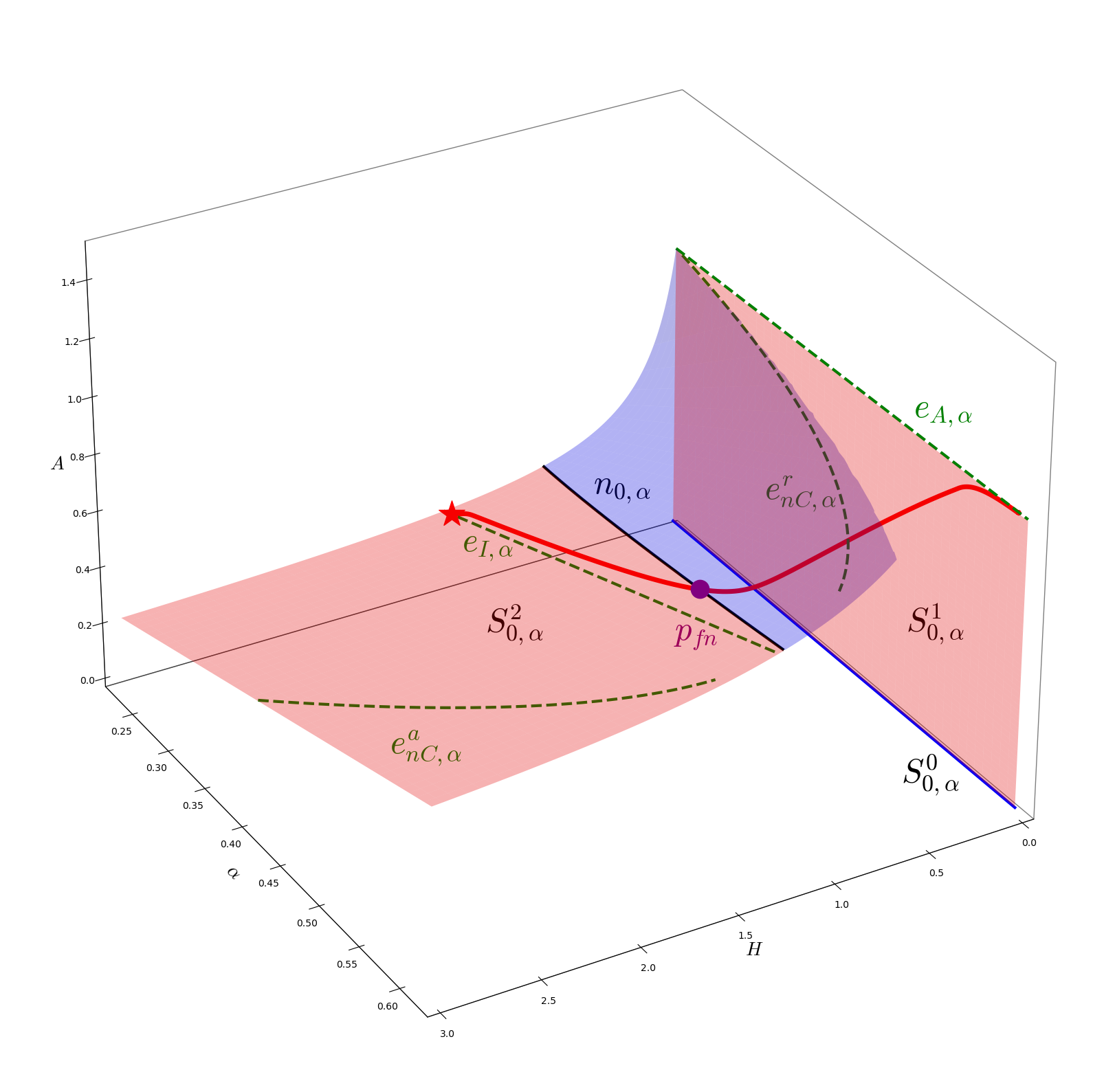}
    \caption{Canard-induced R-tipping in Region $\mathrm{II}$, with $\beta = 0.2 = \lambda$ and $d=0.22$, for $\alpha\in [d+\delta,\alpha^+-\delta]$. Here, we have taken $\varepsilon=0.01$ and $r = 4\times10^{-3}$. The red star represents the initial condition, chosen such that the trajectory originates at the stable coexistence state $e_{I,\alpha}$, with $\alpha = d+\delta$ initially.}
    \label{fig:R-tipping}
\end{figure}

In Figure~\ref{fig:2D canard trajectory}, we compare the time series of the canard trajectory illustrated in Figure~\ref{fig:R-tipping} to that of the tracked equilibrium $e_{I, \alpha}$. Standard linear theory predicts that any trajectory which is initiated sufficiently close to the coexistence state $e_{I, \alpha}$ will be attracted towards it for any fixed $\alpha\in(\alpha_{\rm min,\delta}, \alpha_{\rm max,\delta})$. However, the underlying canard phenomenon allows for such trajectories to ``tip" towards $e_{A, \alpha}$ instead. As the value of $H_{I}$ is independent of $\alpha$, by \eqref{eq: coexistence point}, an increase in $\alpha$ in \eqref{eq: full 4D system} would indicate an increase in the algal population $A$, which would necessitate a decrease in the competing population of coral ($C$), as evidenced in Figure~\ref{fig:2D canard trajectory}. While an increase in the fishing effort $\alpha$ hence leads to algae dominating over herbivorous fish, the three-dimensional model, Equation~\eqref{fast-system}, predicts that the fish population should remain unchanged and that there is still coexistence between algae and coral, as is again reflected in Figure~\ref{fig:2D canard trajectory}. However, the canard phenomenon which occurs in the ``ramped", four-dimensional model, Equation~\eqref{eq: full 4D system} results in the population of herbivorous fish collapsing catastrophically, which leads to a collapse in the coral population, with the flow settling into the algae-only state.

\begin{figure}
    \centering
    \includegraphics[width=0.5\linewidth]{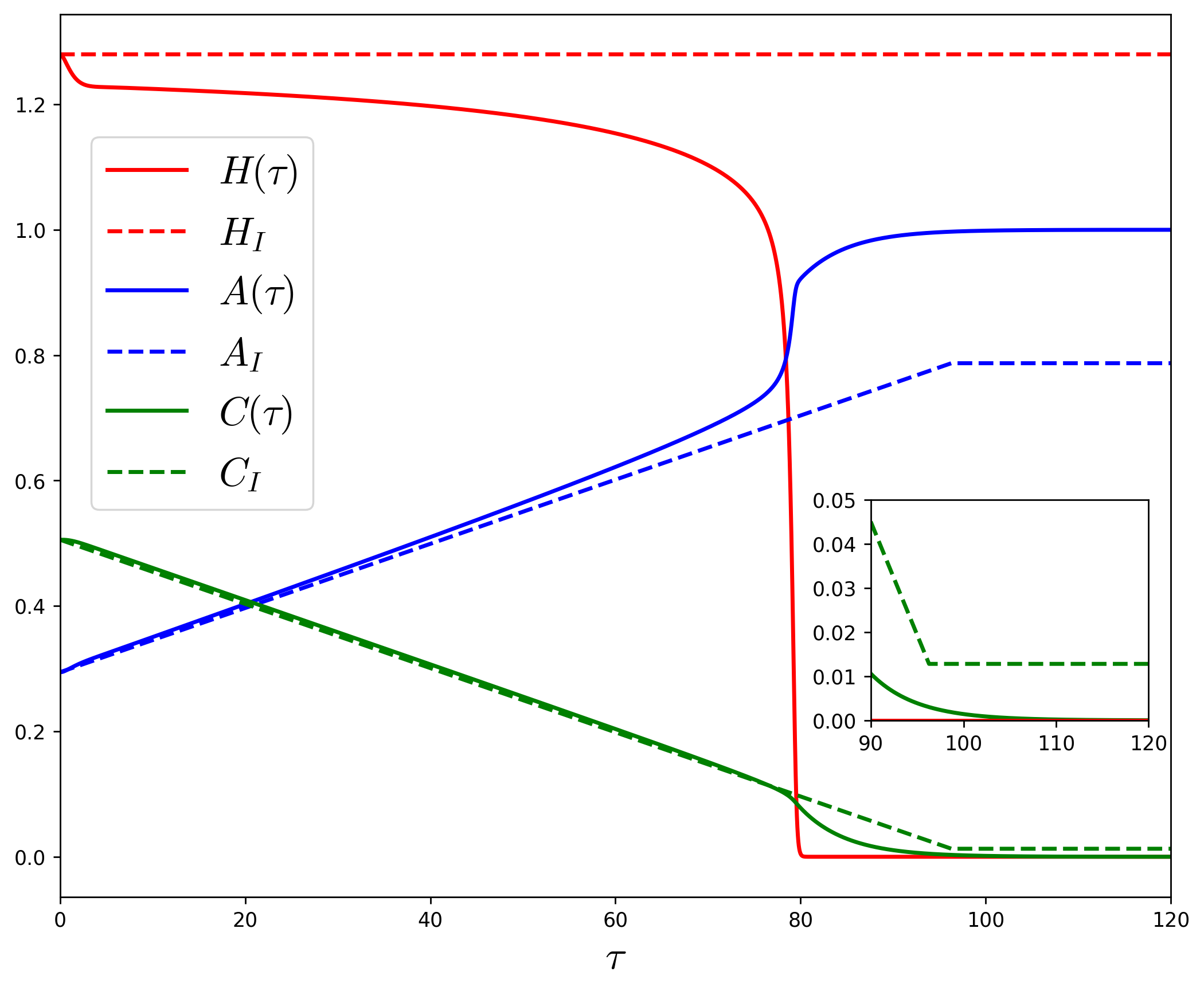}
    \caption{Time series of the canard trajectory illustrated in Figure~\ref{fig:R-tipping} in the slow time $\tau$. Here, herbivorous fish, algae, and coral are represented in red, blue, and green, respectively. A comparison against the evolution of the stable coexistence state $e_{I, \alpha}$ is also given (dashed).}
    \label{fig:2D canard trajectory}
\end{figure}

For the same parameter triple $(\beta,\lambda,d)$ and the same initial condition, but $r=3\times10^{-6}<r_{\rm crit}$, the dynamics are illustrated in Figure~\ref{fig:No R-tipping}. As predicted, the folded singularity is a focus, in which no R-tipping occurs; rather, the trajectory simply tracks $e_{I,\alpha}$, which implies that the chosen parameter regime is located within Region $\mathrm{III}$ now.

\begin{figure}[h]
    \centering
     \includegraphics[scale = 0.2]{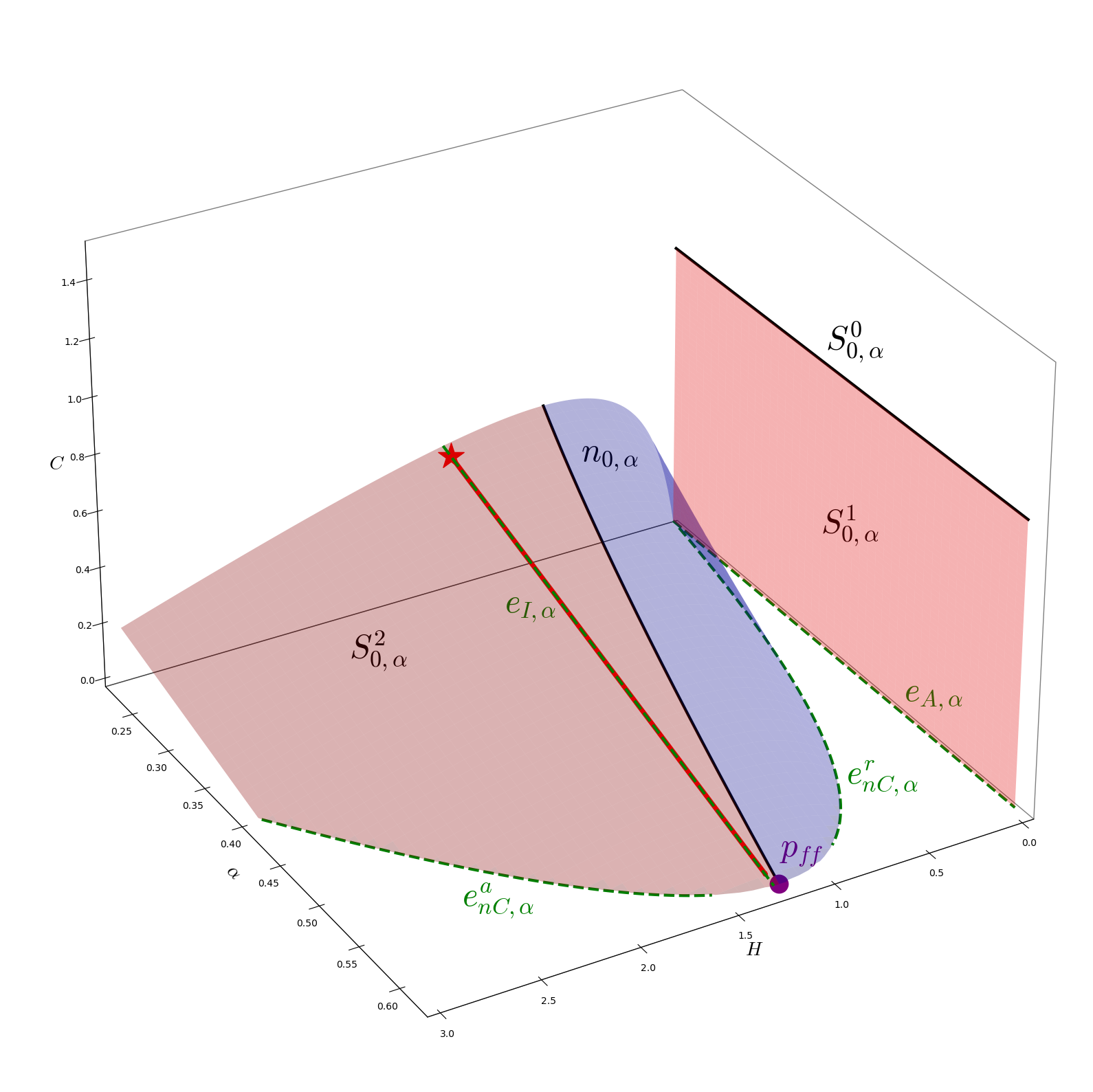}
    \includegraphics[scale = 0.2]{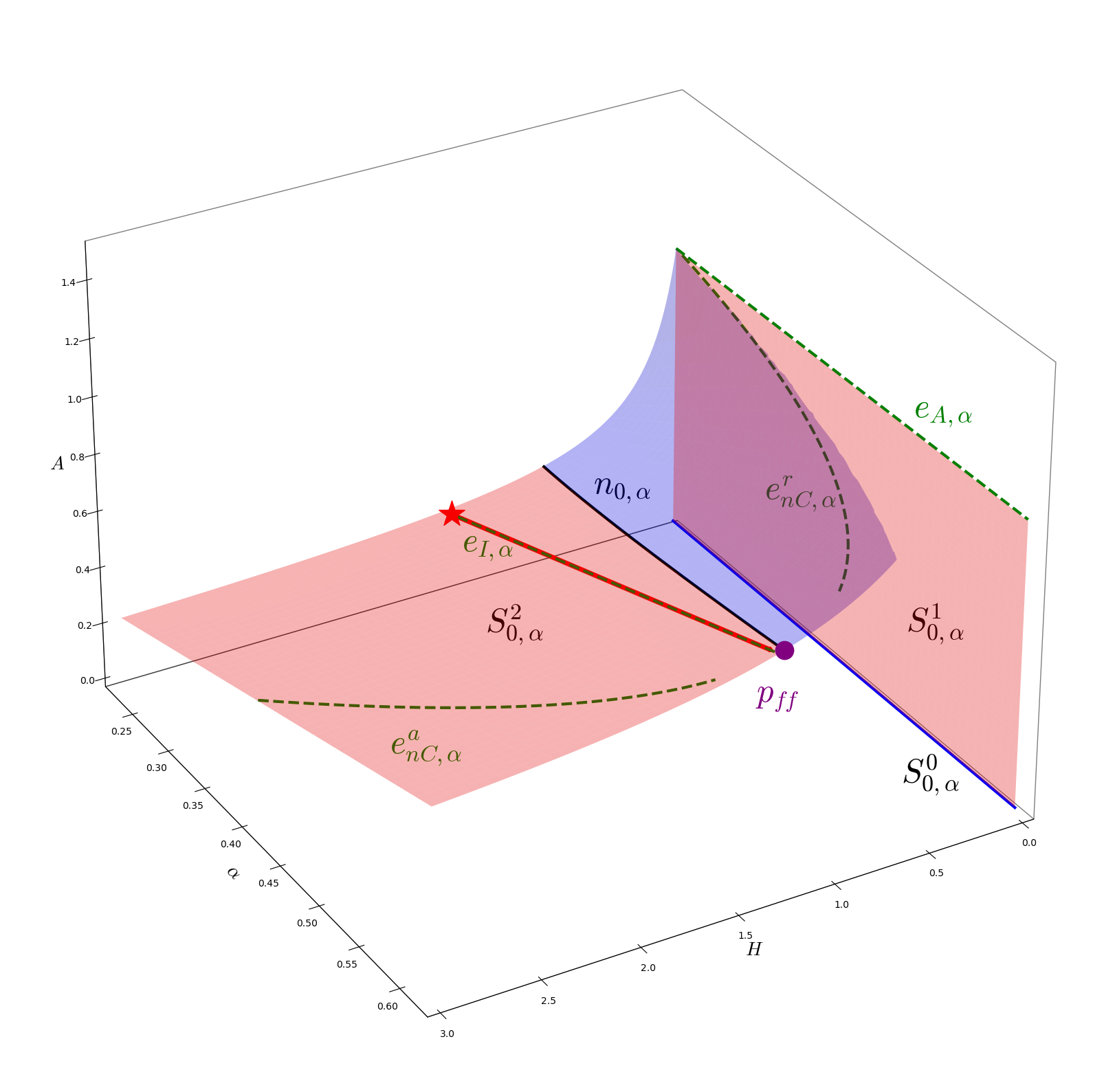}
    \caption{Tracking in Region $\mathrm{II}$, with $\beta = 0.2 = \lambda$ and $d=0.22$, for $\alpha\in [d+\delta,\alpha^+-\delta]$. Here, we have taken $\varepsilon=0.01$ and $r = 3\times10^{-6}$. The red star represents the initial condition, chosen such that the trajectory originates at the stable coexistence state $e_{I,\alpha}$, with $\alpha = d+\delta$ initially.    
    }
    \label{fig:No R-tipping}
\end{figure}

\subsubsection{Region \texorpdfstring{$\mathrm{III}$}{III}}
In Region $\mathrm{III}$, the relevant folded singularity of \eqref{eq: Desingularised system} is a focus; as folded foci do not admit canard trajectories \cite{SZMOLYAN2001419}, any rate-induced tipping that is observed cannot be of canard type; conversely, we do not expect to see tipping throughout Region $\mathrm{III}$. An example trajectory with $\beta=0.3$, $\lambda = 0.4$, and $r = 1\times10^{-5}$ is illustrated in Figure~\ref{fig:alpha_hat trajectory}, where we note that $\alpha_{\rm max} = \alpha^+$; we observe that the trajectory simply tracks the coexistence state throughout until integration is halted at $\alpha_{\max,\delta}=\alpha^{+}-\delta$ so that the transcritical bifurcation at $\alpha=\alpha^{+}$ is avoided.

\begin{figure}[h]
    \centering
     \includegraphics[scale = 0.2]{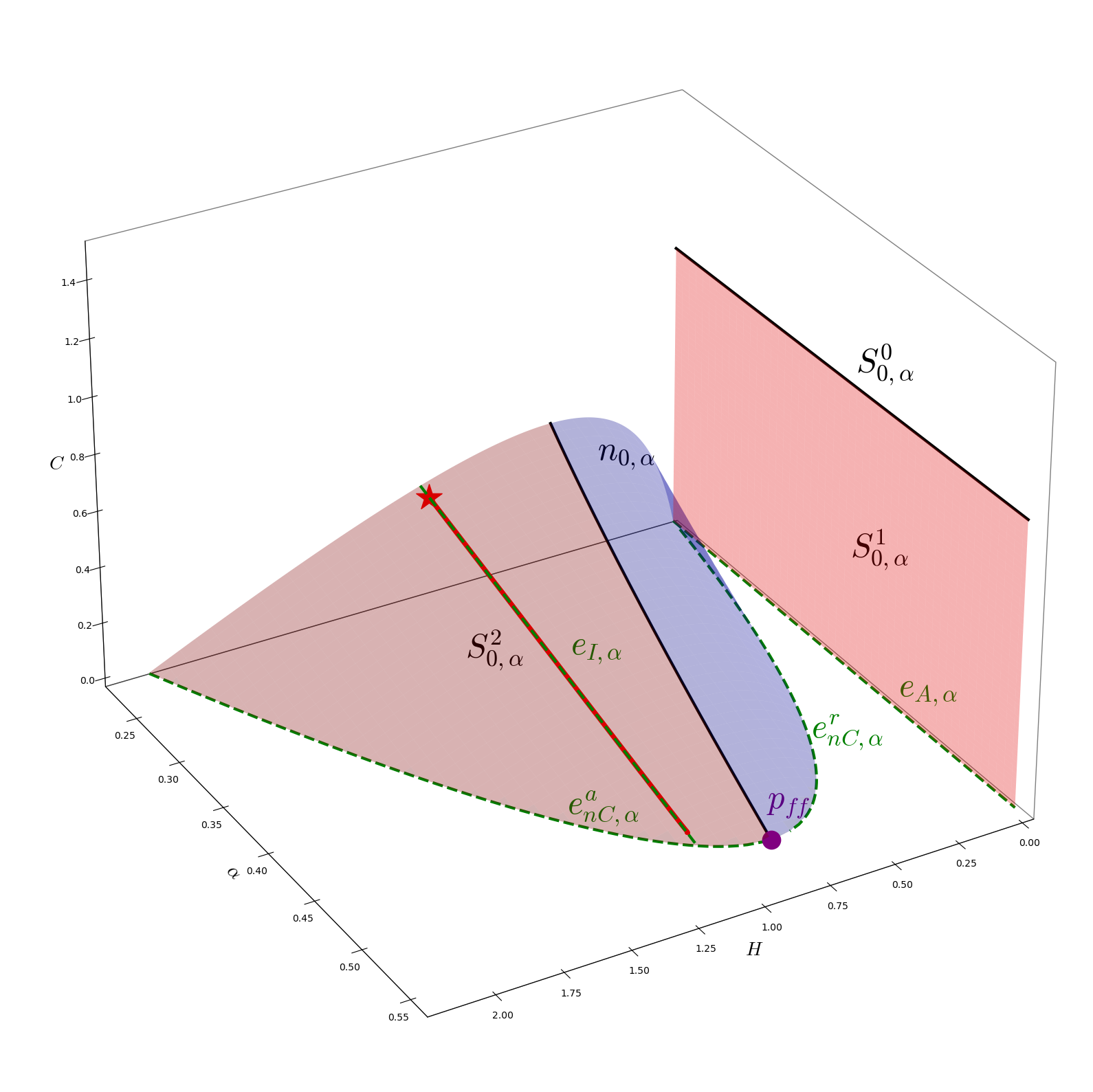}
    \includegraphics[scale = 0.2]{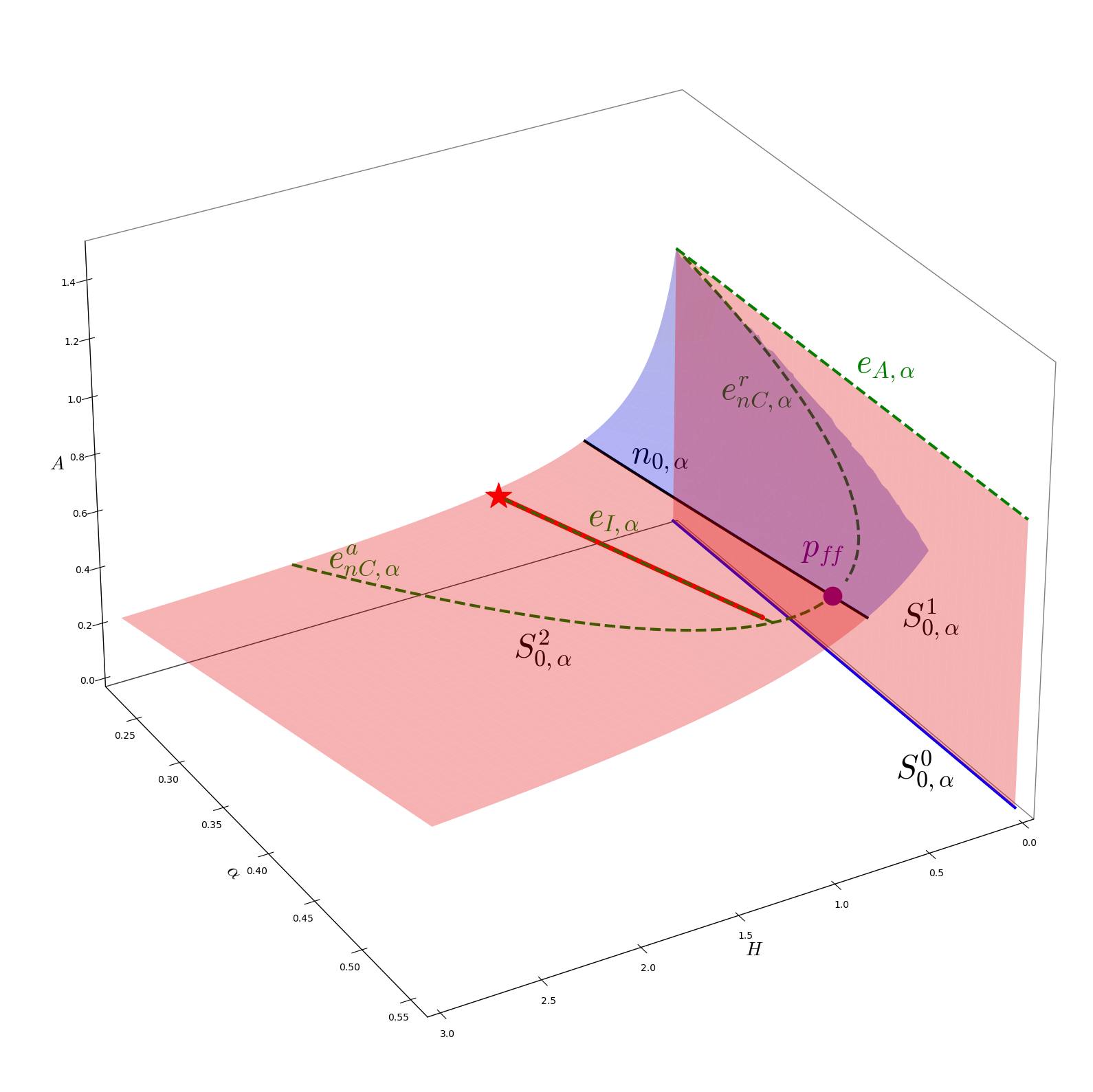}
    \caption{Tracking in Region $\mathrm{III}$, with $\beta=0.3$, $\lambda = 0.4$, and $d=0.22$, for $\alpha\in [d+\delta,\alpha^+-\delta]$. Here, we have taken $\varepsilon=0.01$ and $r = 1\times10^{-5}$. The red star represents the initial condition, chosen such that the trajectory originates at the stable coexistence state $e_{I,\alpha}$, with $\alpha = d+\delta$ initially.}
    \label{fig:alpha_hat trajectory}
\end{figure}

For $r$ sufficiently large, it is still possible to see rate-induced tipping in Region $\mathrm{III}$; however, it will be of ``jump type", rather than canard-induced, in contrast to Regions $\mathrm{I}$ and $\mathrm{II}$. A representative example is illustrated in Figure~\ref{fig:focus jump trajectory}, where we have chosen the same parameter set and initial condition as in Figure~\ref{fig:alpha_hat trajectory}, with the exception of $r=4\times10^{-3}$. One observes, in particular, that the trajectory now leaves the attracting portion $S_{\varepsilon,\alpha}^{2,a}$ of the manifold $S_{\varepsilon,\alpha}^{2}$ via a jump point on the fold curve $n_{0,\alpha}$. The corresponding $\alpha$-value exceeds the value $\alpha_{FS1,r}$ at the folded singularity $p_{1}$, which is consistent with the geometry of the reduced flow in a neighbourhood of the folded focus $p_{ff}$ on the critical manifold $S_{0,\alpha}^{2,a}$; recall Figure~\ref{fig:second desing}. After jumping, the trajectory is immediately attracted to the slow manifold $S_{\varepsilon,\alpha}^{1}$, eventually settling into the algae-only state $e_{A,\alpha}$. Ecologically speaking, we hence again observe a catastrophic decline in the populations of herbivorous fish $H$ and coral $C$.

\begin{figure}[h]
    \centering
     \includegraphics[scale = 0.2]{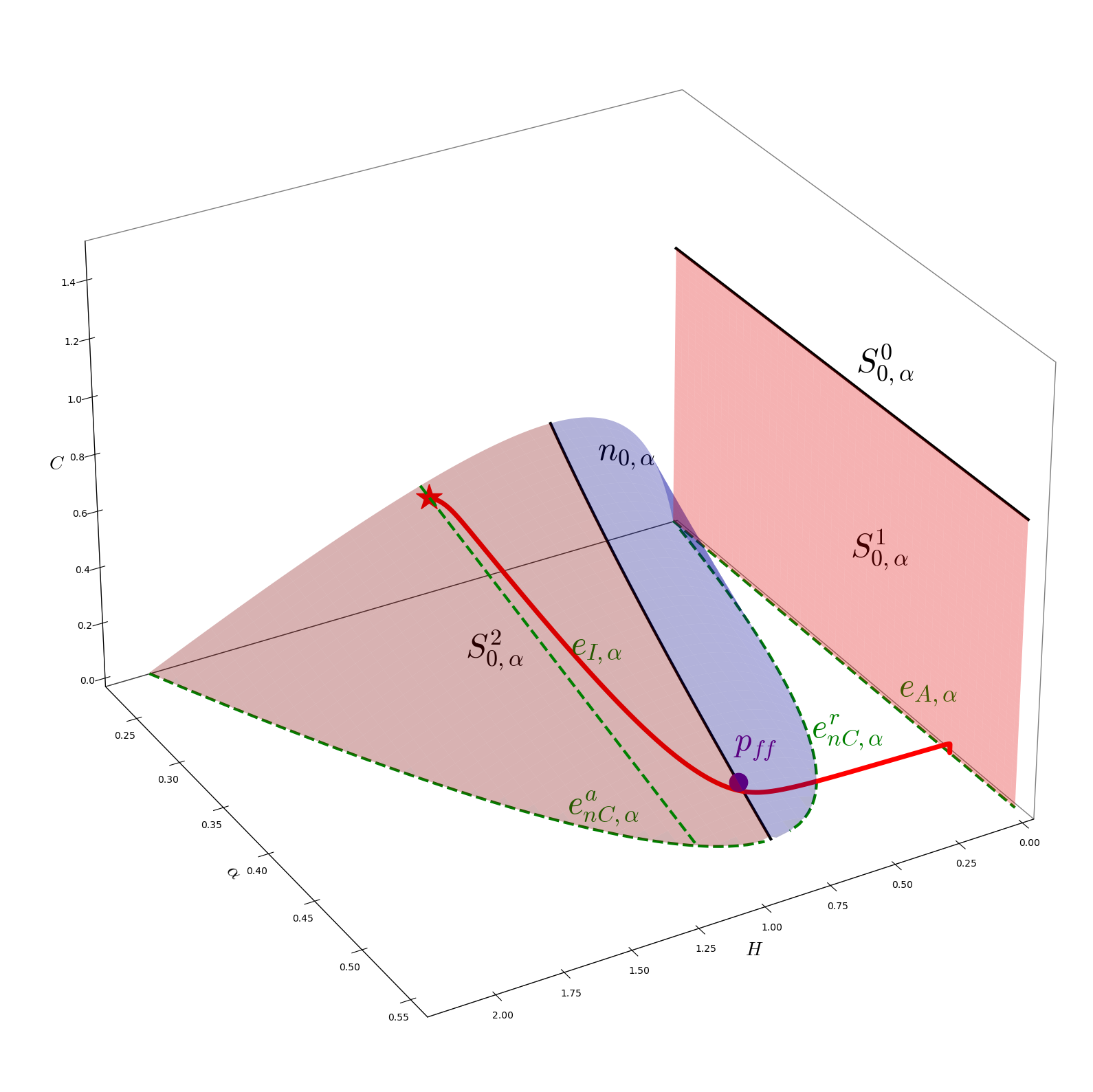}
    \includegraphics[scale = 0.2]{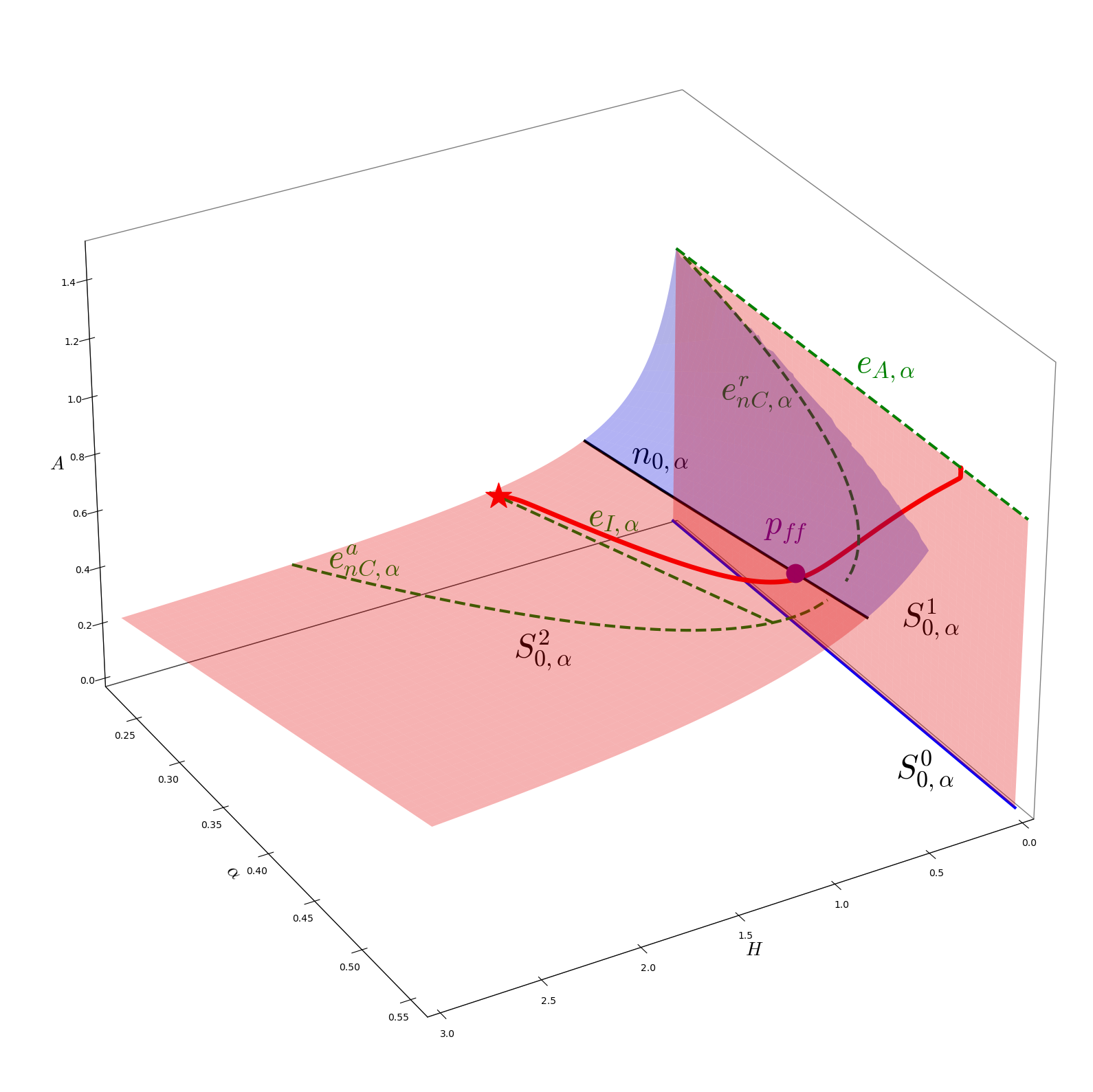}
    \caption{Jump-induced R-tipping in Region $\mathrm{III}$, with $\beta=0.3$, $\lambda = 0.4$, and $d=0.22$, for $\alpha\in [d+\delta,\alpha^+-\delta]$. Here, we have taken $\varepsilon=0.01$ and $r = 4\times10^{-3}$. The red star represents the initial condition, chosen such that the trajectory originates at the stable coexistence state $e_{I,\alpha}$, with $\alpha = d+\delta$ initially.}
    \label{fig:focus jump trajectory}
\end{figure}

In summary, Region $\mathrm{III}$ is hence naturally divided into two subregions, which we label Region $\mathrm{III}(a)$ and Region $\mathrm{III}(b)$; these are defined by $\alpha_{FS1,r}<\alpha^+$ and $\alpha_{FS1,r}>\alpha^+$, respectively, and are illustrated in Figure~\ref{fig:All regions} for $r=4\times 10^{-3}$. By the above, it is possible to observe R-tipping of jump type in Region $\mathrm{III}(a)$, whereas in Region $\mathrm{III}(b)$, the flow will not tip, but will remain on $S^{2,a}_{\varepsilon,\alpha}$, and will track the coexistence state $e_{I,\alpha}$; recall Figures~\ref{fig:focus jump trajectory} and \ref{fig:alpha_hat trajectory}, respectively.

\begin{figure}[h]
    \centering
     \includegraphics[scale = 0.7]{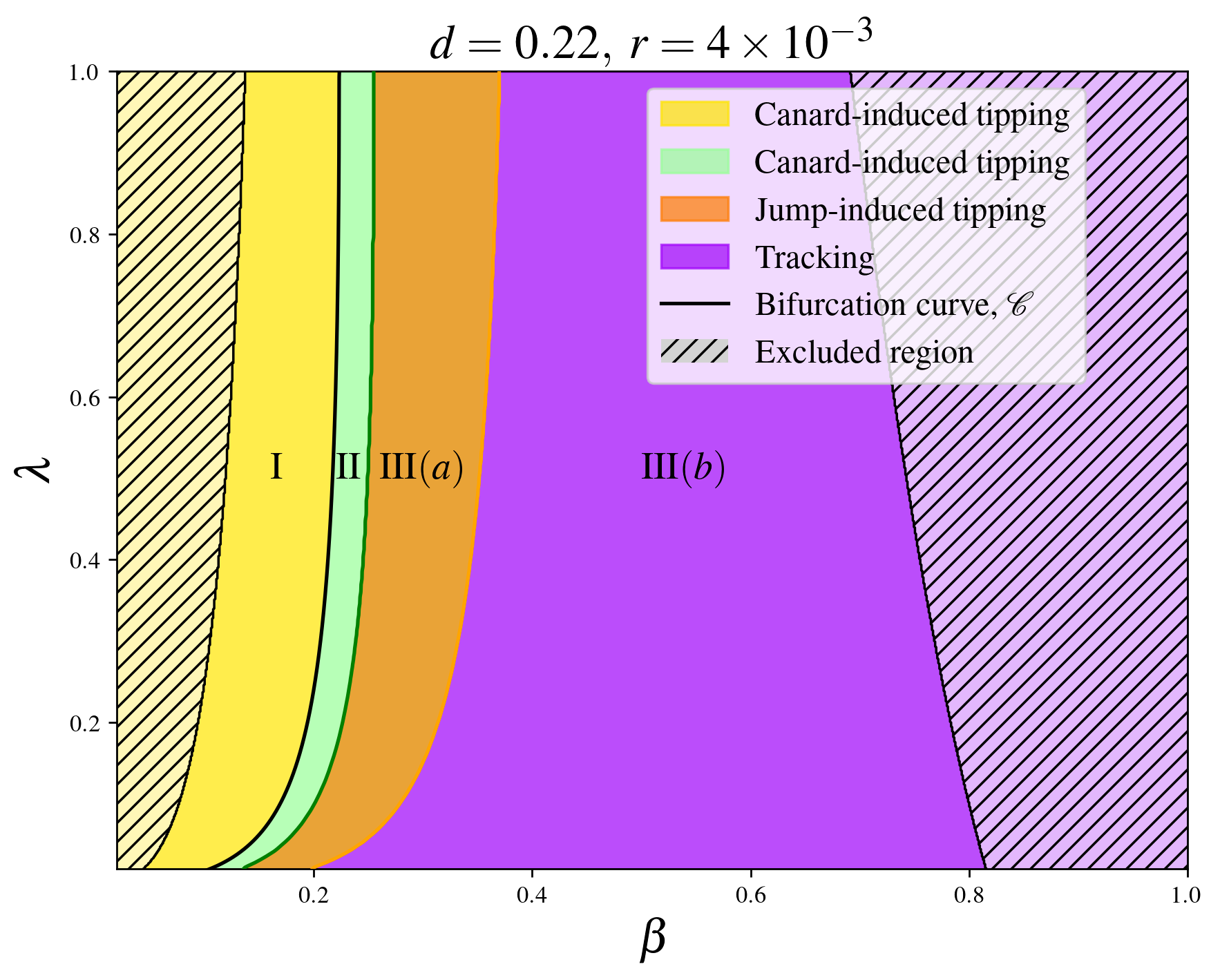}
    \caption{Summary classification of the dynamics of Equation~\eqref{eq: full 4D system} in dependence of $(\beta,\lambda)\in(0,1)^2$ in the absence of bifurcation, with $d=0.22$ and $r=4\times 10^{-3}$ fixed. Canard-induced R-tipping is observed in Region $\mathrm{I}$ (yellow) and Region $\mathrm{II}$ (green). Region $\mathrm{III}$ is subdivided into Region $\mathrm{III}(a)$ (orange) and Region $\mathrm{III}(b)$ (purple) in which one observes jump-induced R-tipping and tracking, respectively.}
    \label{fig:All regions}
\end{figure}

Finally, we illustrate that the dynamics of Equation~\eqref{eq: full 4D system} is somewhat robust to the choice of $\alpha_{\max,\delta}$. For illustration, we choose $\alpha_{\max,\delta}>\hat\alpha(>\alpha^+)$ and $r=1\times10^{-5}$ small enough such that R-tipping is not observed. Figure~\ref{fig:focus hat trajectory} implies that the flow simply tracks the coexistence state $e_{I,\alpha}$ until the transcritical bifurcation at $\alpha^{+}$, where $e_{I,\alpha}$ exchanges stability with $e_{nC,\alpha}^a$. The trajectory then follows $e_{nC,\alpha}^a$ until the saddle-node bifurcation at $\hat \alpha$, where the coral-free equilibria $e_{nC,\alpha}^a$ and $e_{nC,\alpha}^r$ coalesce. For $\alpha>\hat\alpha$, given the invariance of $\{C=0\}$, the trajectory is then attracted to the only remaining attractor, the algae-only state $e_{A,\alpha}$. A rigorous proof would require us to desingularise the dynamics in the neighbourhood of the fully degenerate state $\hat{e}_\alpha=\Big(\hat{H}, \hat{\alpha}, 0, \frac{\hat{\alpha}}{s(\hat{H})}\Big)$, where $\hat{\alpha}=\Pi(\hat{H})$ is as defined in Lemma~\ref{lem: S_0^2 intersections}; recall Remark~\ref{degeneracy}. However, we conjecture that desingularisation would reveal a standard fold-like jump point \cite{fold} in the invariant hyperplane $\{C=0\}$, as indicated in Figure~\ref{fig:focus hat trajectory}.

\begin{figure}[h]
    \centering
     \includegraphics[scale = 0.2]{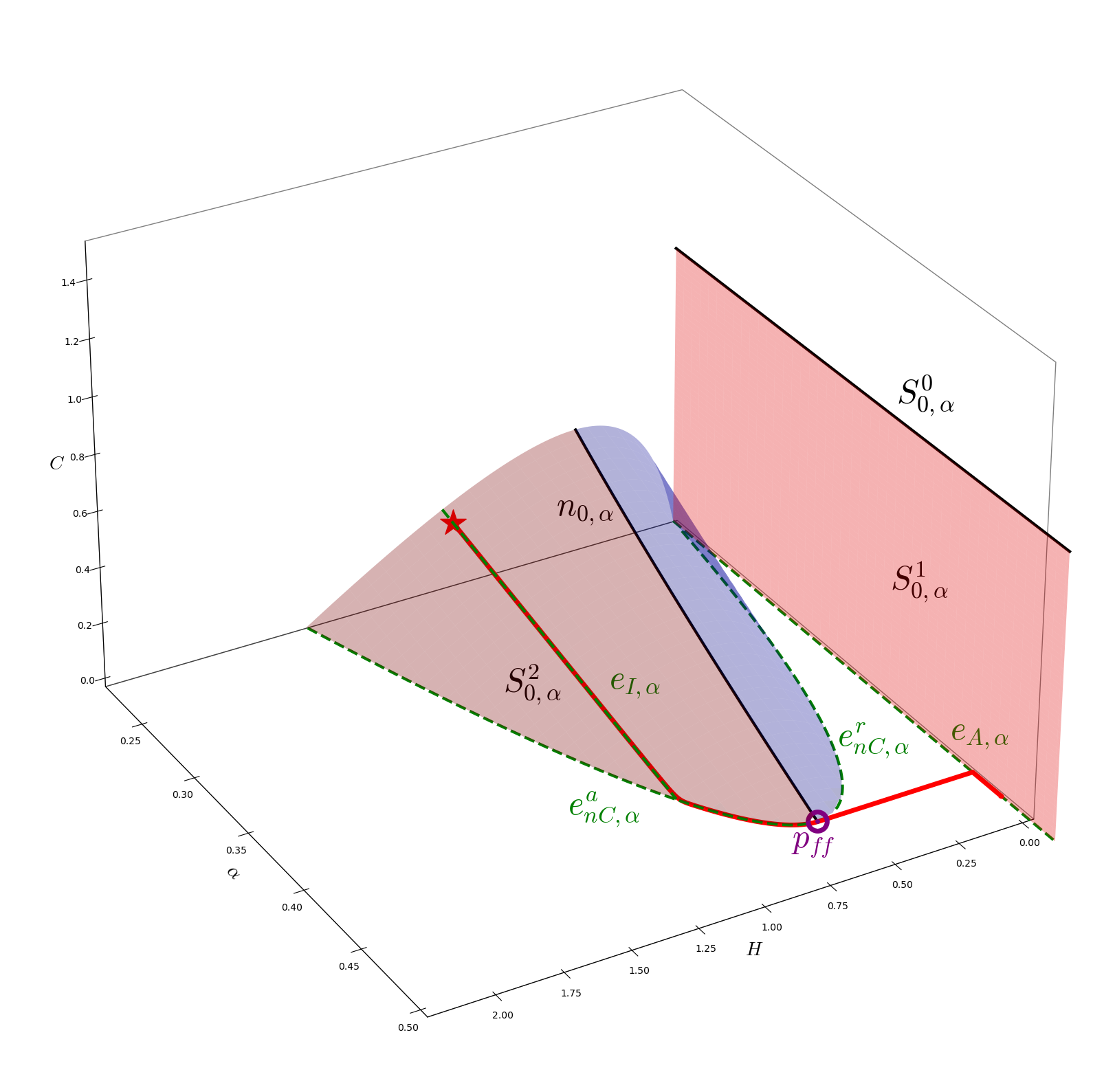}
    \includegraphics[scale = 0.2]{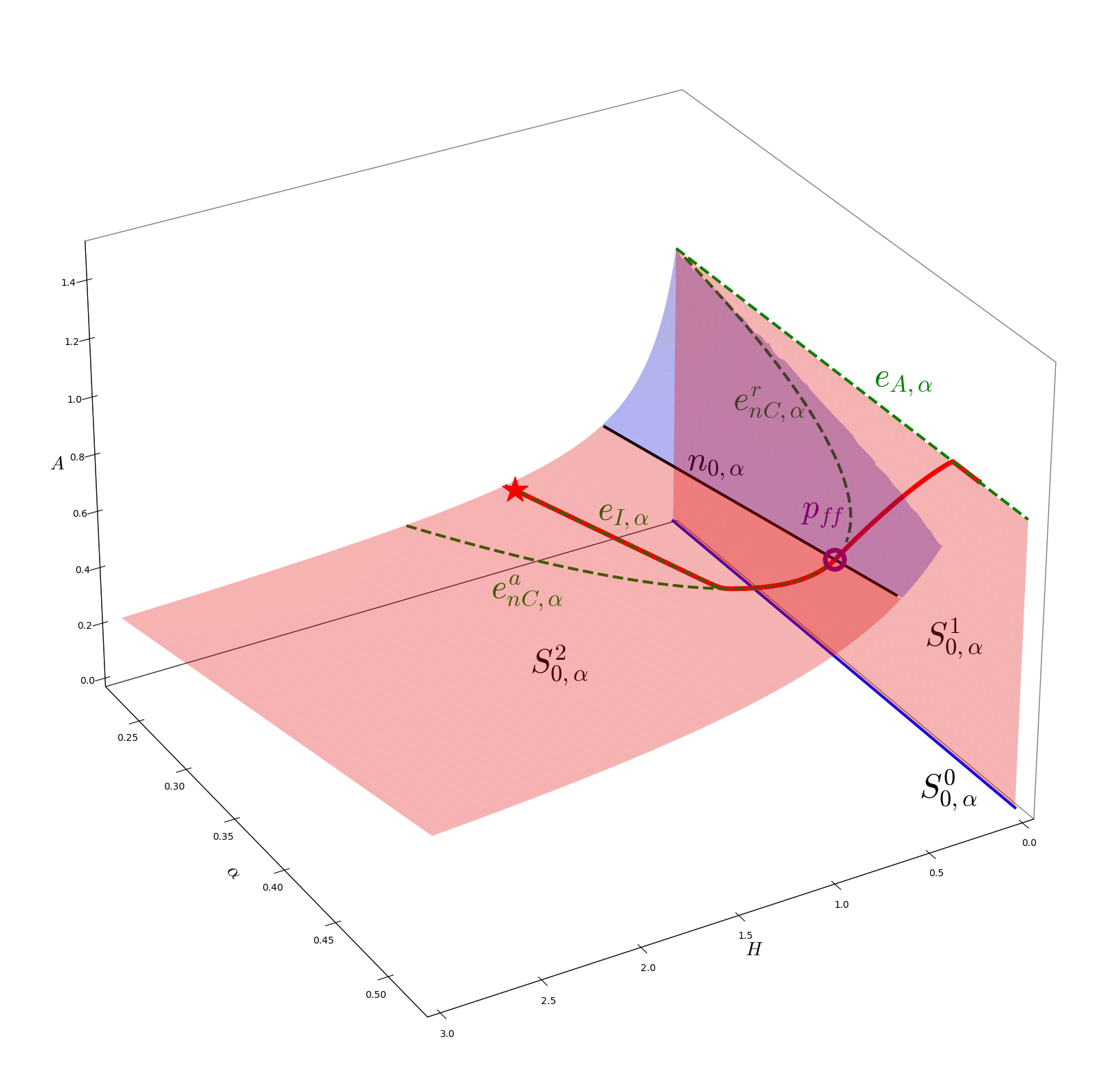}
    \caption{Integration beyond $\hat\alpha>\alpha^+$ in Equation~\eqref{eq: full 4D system}, with $\beta=0.4$, $\lambda=0.6$, and $d=0.22$, for $\alpha\in [d+\delta,0.49]$ such that $\alpha$ is allowed to exceed $\hat{\alpha}\approx0.4676$. Here, we have taken $\varepsilon=0.01$ and $r=1\times10^{-5}$. The trajectory tracks $e_{I,\alpha}$ and $e_{nC,\alpha}^a$ until their respective bifurcations at $\alpha^+$ and $\hat\alpha$, respectively, before tipping towards $e_{A,\alpha}$. The red star represents the initial condition, chosen such that the trajectory originates at the stable coexistence state $e_{I,\alpha}$, with $\alpha = d+\delta$ initially.}
    \label{fig:focus hat trajectory}
\end{figure}

\section{Discussion}
In this article, we have applied geometric singular perturbation theory (GSPT) to analyse a demographic model for a reef ecosystem consisting of fish, algae, and coral, Equation~\eqref{PNAS Model}, that was originally introduced in \cite{gil2020}. 

In Section~\ref{section GPST}, we have identified a separation of time scales in \eqref{PNAS Model}, and we have described the singular geometry of the resulting $(2,1)$ fast-slow system, Equation~\eqref{fast-system}. We have then given a complete classification of the bifurcation structure of that system in terms of the (non-dimensionalised) fishing effort $\alpha$. For $\alpha>d$ in \eqref{fast-system}, we observe three $\alpha$-regimes, each of which characterises a different ecological scenario. In particular, we have concluded that, in the ecologically relevant regime where $d<\alpha<\min\{\alpha^+,\alpha^{\ast}\}$, \eqref{fast-system} exhibits bistability, with both the coexistence state $e_I$ and the algae-only state $e_A$ being local attractors. 

In Section~\ref{section R-tipping}, we have formulated an extended, ``ramped" version of \eqref{fast-system} in which the fishing effort $\alpha$ is an additional state variable, by appending the equation $\frac{\dd \alpha}{\dd t}=r$; recall Equation~\eqref{eq: full 4D system}. We have shown that, for $d<\alpha(\tau)<\alpha_{\max}\equiv \min \{\alpha^+, \alpha^{\ast}\}$, the rate of change $r$ of the fishing effort $\alpha$ plays a critical role for the long-term sustainability of the ecosystem. Specifically, we have demonstrated that a slow increase in the fishing effort $\alpha$ can lead to rate-induced tipping (``R-tipping"), whereby the populations of fish and coral collapse while the population of algae blooms. We emphasise that this regime shift, from stable coexistence to an algal bloom, is not due to a classical bifurcation (``B-tipping"), as the fishing effort itself, rather than its rate of change, would need to exceed a certain threshold in that scenario. Correspondingly, \eqref{eq: full 4D system} remains bistable throughout, with both the coexistence state and the algae-only state being locally attracting for any $\alpha\in(d,\alpha_{\rm max})$. 

Modifying the methodology introduced in \cite{Plankton-tipping,vanselow_2019_when}, where  planar fast-slow systems were shown to undergo rate-induced tipping due to a slowly evolving parameter, we have given a detailed description of the folded geometry of the extended critical manifold for \eqref{eq: full 4D system}. In particular, we have concluded that tipping will typically occur due to the presence of a folded node-type canard (``Region $\mathrm{I}$" and ``Region $\mathrm{II}$"); less typically, trajectories may tip via a fast jump past a folded focus singularity (``Region $\mathrm{III}$").

While our analysis is hence based on previous studies of canard-induced R-tipping in singularly perturbed systems of fast-slow type \cite{Plankton-tipping,vanselow_2019_when}, our findings differ from theirs in two significant respects.

First, in both \cite{Plankton-tipping,vanselow_2019_when}, a folded saddle-type canard was shown to give rise to rate-induced tipping; as is well known, folded saddles generate unique (strong) canards which then act as ``thresholds", with trajectories to either side of them either tracking or tipping. Furthermore, trajectories that tip due to a folded saddle canard do so via a jump point past the threshold. By contrast, in our model, canard-induced tipping occurs due to a folded node-type canard, resulting in qualitatively very different dynamics: a folded node imposes a ``funnel"-type structure on the corresponding critical manifold; furthermore, in addition to the two primary canards corresponding to the eigendirections of the linearsation at the folded node, one observes bifurcating (``secondary") canards which subdivide (part of) the funnel into so-called ``sectors of rotation" \cite{MR2136520}. Dynamically, it hence follows that, in our model, trajectories which are initiated close to the coexistence state will enter the funnel, and will hence tip by following the ``weak" canard, possibly undergoing subthreshold oscillation by passing through sectors of rotation. In sum, any trajectory that tips in the presence of a folded node will be of canard type, rather than of jump type.

Bifurcations of folded node-type canards in systems where more than two distinct scales occur have received substantial interest in the recent literature \cite{6c8c3bfcb06e4fb8b0cd94ae49d0c15b}; in such systems, folded nodes typically degenerate to folded saddle-nodes which can give rise to complex oscillatory dynamics, such as mixed-mode oscillation \cite{MR2916308}. Related phenomena have been reported that arise due to delayed Hopf bifurcation in models of fast-slow-``super-slow" type in which the dynamics evolves on three distinct scales \cite{de2016sector}. As Equation~\eqref{eq: full 4D system} is of that type for $r$ sufficiently small, it is unsurprising that we observe delayed-Hopf-type dynamics then. Future work could consider the three-scale dynamics of \eqref{eq: full 4D system} more extensively and rigourously.

To the best of our knowledge, our study represents the first instance in the literature where rate-induced tipping occurs due to a folded node-type canard. (On a more minor note, the rate-induced tipping of jump type in our model that is due to the presence of a folded focus also appears to be novel.) While R-tipping in the demographic model studied here has been identified previously by Gil et al. \cite{gil2020}, the underlying geometric mechanisms were not elucidated systematically, nor was a separation of scales considered explicitly.

A second significant difference to previous work is that our basic model, with $\alpha$ a fixed parameter, is a three-dimensional bistable system; by contrast, in \cite{Plankton-tipping,vanselow_2019_when}, the underlying models are both planar and monostable. The existence theory for canards was originally developed in three dimensions and, specifically, for $(1,2)$ fast-slow systems; see \cite[Theorem $4.1$]{SZMOLYAN2001419}.
As \eqref{eq: full 4D system} is four-dimensional, with a $(2,2)$ fast-slow structure, the existence of canards follows from more recently developed results \cite[Theorem $4.2$]{Canards}. Correspondingly, our study evidences clear differences in rate-induced tipping between monostable and bistable systems. In \cite{vanselow_2019_when}, trajectories initialised near the coexistence equilibrium of their Equation~\eqref{Vanselow} tip in the form of a canard trajectory which initially appears to have driven both populations to extinction. However, as the interior equilibrium is the unique attractor there, ``resurgence" occurs even for extremely small prey densities -- of the order $10^{-12}$. Hence, the populations of both prey and predator recover and again reach stable coexistence. 
In our case, however, the bistability of \eqref{eq: full 4D system} implies that tipping trajectories will not experience resurgence but, rather, that they will converge to the alternative local attractor. That lack of resurgence is ecologically realistic, as one would expect populations to go extinct at low densities due to stochastic effects.

On a related note, one may ask whether a population collapse due to rate-induced tipping is reversible. Specifically, can the flow of \eqref{eq: full 4D system} ever return to stable coexistence following an algal bloom? In practice, fishing can be halted near-instantaneously in an attempt to allow the populations of fish and coral to recover; mathematically, the dimensionless fishing effort $\alpha=\frac{\mu+f}{r_A}$ can be reduced to its value pre-fishing by setting the mortality rate of fish due to fishing to zero ($f=0$). (Note that we are now considering \eqref{fast-system} instead of the ``ramped" system, Equation~\eqref{eq: full 4D system}, as we are taking $\alpha(\tau)\equiv\frac{\mu}{r_A}$ constant.)
By Table~\ref{param_table} $\alpha=\frac{\mu}{r_A}$ is maximised at $0.02<d=0.22$. It follows from Lemma~\ref{lemma: Stability of Critical Manifold} that the algae-only state then lies on a repelling portion of $S_0^1$, which implies that the only remaining attractor for \eqref{fast-system} is the coexistence state. Figure~\ref{fig:resurgence} visualises this scenario, where population collapse occurs due to rate-induced tipping; however, an instantaneous reduction in $\alpha$ results in the resurgence of fish and coral populations.

\section{Acknowledgements}
BH and ZS were supported by the EPSRC Centre for Doctoral Training in Mathematical Modelling, Analysis
and Computation (MAC-MIGS), funded by the UK Engineering and Physical Sciences Research
Council (grant EP/S023291/1, Heriot-Watt University and the University of Edinburgh). This study is based on a dissertation that had been submitted
by IA for the degree of an MSc in Computational Applied Mathematics at the University of Edinburgh in 2025.

\begin{figure}
    \centering
    \includegraphics[width=0.5\linewidth]{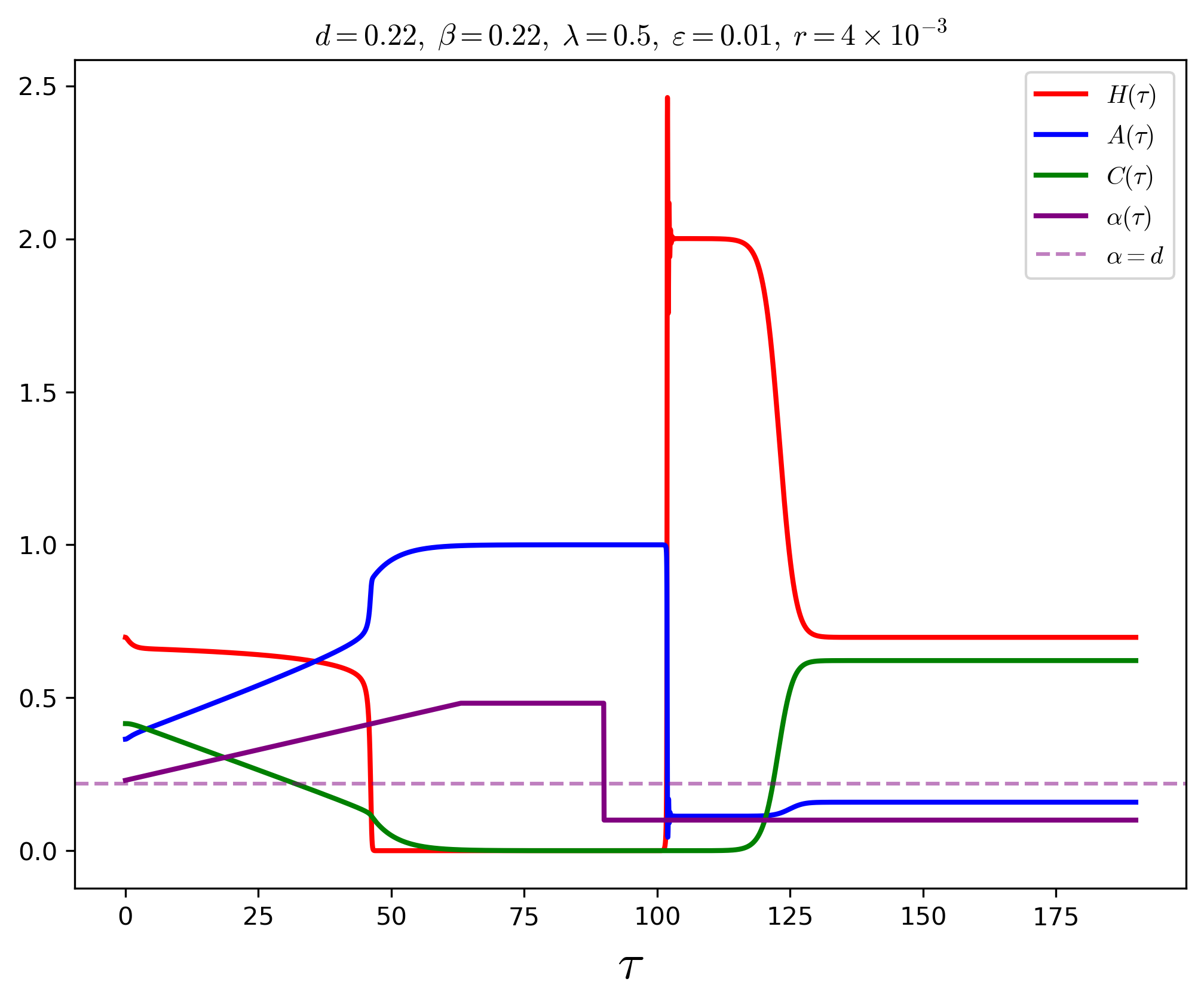}
    \caption{Time series plot of ``resurgence" in Equation~\eqref{eq: full 4D system}, with $\alpha(\tau)$ reset to $\alpha<d$ once population collapse has occurred.}
    \label{fig:resurgence}
\end{figure}

\appendix
\section{Proof of Lemma~\ref{lemma: Stability of Critical Manifold}}
\label{sec: stability of S_0}
In this section, we give the proof of Lemma~\ref{lemma: Stability of Critical Manifold}.

To prove that the function $Q(H;\alpha) = \lambda \big[H+s(H)(1+H)^2\big]-\frac{\alpha}{s(H)^2}$ admits a unique root $H_0$ under the assumption that $\alpha> \lambda d^3$, we first show that $Q$ is a monotonically increasing function in $H$: since $s'(H)=(1+H)^{-2}$ implies
\begin{equation}
    \frac{d Q(H;\alpha)}{dH} = 2\lambda (d+1)(1+H)+\frac{2\alpha}{(1+H)^2s(H)^3}>0,
\end{equation}
$Q(H;\alpha)$ is increasing for all $H \geq 0$. We have $Q(0;\alpha) = \lambda d-\frac{\alpha}{d^2}<0$, by our assumption; furthermore, as $s(H)\in[d,d+1]$ is bounded, $Q(H;\alpha)\to \infty$ as $H{\to \infty}$. Therefore, $Q(H;\alpha)$ has a unique root, denoted by $H_0$, which implies item 1.   

To investigate the stability of the three sub-manifolds $S_0^0$, $S_0^1$, and $S_0^2$ of $S_0$, we calculate the Jacobian $\mathcal{J}$ of the linearisation of the layer problem, Equation~\eqref{layer-problem}, with respect to the fast variables $H$ and $A$: 
\begin{align*}
    \mathcal{J} = \begin{pmatrix}
        \lambda\Big\{\Big[s(H)+\frac{H}{(1+H)^2}\Big]A-\alpha\Big\} &\lambda H s(H) \\
        -\lambda A\Big[s(H)+\frac{H}{(1+H)^2}\Big] & 1-2A-C-\lambda Hs(H)
    \end{pmatrix}.
\end{align*}
First, evaluating $\mathcal{J}$ on $S_0^0$, we have
\begin{align*}
    \mathcal{J}_{S_0^0} = \begin{pmatrix}
        -\lambda\alpha & 0 \\
        0 & 1-C
    \end{pmatrix},
\end{align*}
which has one non-negative eigenvalue for $0\leq C\leq 1$. Hence, the sub-manifold $S_0^0$ is normally repelling when $0\leq C<1$, losing normal hyperbolicity through the transcritical singularity $n_1$ at $C=1$, which proves item 2 in Lemma~\ref{lemma: Stability of Critical Manifold}.
Similarly, we evaluate $\mathcal{J}$ on $S_0^1$ to find
\begin{align*}
    \mathcal{J}_{S_0^1} = \begin{pmatrix}
        \lambda(dA-\alpha) &0 \\
        -\lambda A d & -A
    \end{pmatrix},
\end{align*}
which is a lower triangular matrix whose eigenvalues are simply the entries on its diagonal. As the eigenvalue $-A(=-1+C)$ is always negative for $A>0$ or, equivalently for $C<1$, the stability of $S_0^1$ is determined by the sign of $dA-\alpha$. Hence, it follows that $S_0^1$ is normally attracting for $A<\frac{\alpha}{d}$ and normally repelling for $A>\frac{\alpha}{d}$, losing normal hyperbolicity in the transcritical singularity $n_2$, which shows item 3.  

Finally, the Jacobian $\mathcal{J}$ evaluates to
\begin{align*}
    \mathcal{J}_{S_0^2} = \begin{pmatrix}
        \frac{\lambda\alpha H}{s(H)(1+H)^2} & \lambda H s(H) \\
        -\lambda\alpha\Big[1+\frac{H}{s(H)(1+H)^2}\Big] & -\frac{\alpha}{s(H)}
    \end{pmatrix}
\end{align*}
on $S_0^2$. Normal hyperbolicity is then most easily determined from the trace and determinant of that matrix: we have
\begin{align*}
    \mathrm{tr}\big(\mathcal{J}_{S_0^2}\big) = \frac{\alpha}{s(H)}\bigg[\frac{\lambda H}{(1+H)^2}-1\bigg],
\end{align*}
which is negative for $\lambda <4$,  and
\begin{align*}
    \det\big(\mathcal{J}_{S_0^2}\big) &=-\frac{\alpha^2\lambda H}{s(H)^2(1+H)^2}+\lambda^2\alpha H\bigg[s(H)+\frac{H}{(1+H)^2}\bigg] \\
    &\equiv \frac{\lambda \alpha H}{(1+H)^2}Q(H;\alpha),
\end{align*}
where $Q(H;\alpha) = \lambda \big[H+s(H)(1+H)^2\big]-\frac{\alpha}{s(H)^2}$, as before.
For $\lambda < 4$, the stability of $S_0^2$ is therefore decided by the sign of $Q(H;\alpha)$. Since we assume that $\alpha>\lambda d^3$, it again follows that $Q(H;\alpha)$ is negative for $0 <H<H_0$ and that the manifold $S_0^2$ is normally repelling (of saddle type) in that case. For $H>H_0$, $Q(H;\alpha)$ is positive; hence, $S_0^2$ is normally (focally or nodally) attracting then, with hyperbolicity being lost at $H=H_0$, as stated in item 4. 

To show that $H=H_0$ maximises the curve $C=1-\frac{\alpha}{s(H)}-\lambda H s(H)$, we note that $\frac{\dd C}{\dd H}=-\frac{1}{(1+H)^2}Q(H;\alpha)$, which implies that $\frac{\dd C}{\dd H}=0$ is equivalent to $Q(H;\alpha)=0$ and, hence, that $C(H)$ has a unique critical point at $H=H_0$ for $H\in[0, \infty)$. Furthermore, 
$\frac{d^2 C}{dH^2}=-2\Big\{\frac{\lambda}{(1+H)^3}+\frac{\alpha(1+d)}{[d+(1+d)H]^3} \Big\}<0$, which completes the proof of Lemma~\ref{lemma: Stability of Critical Manifold}.

\section{Proof of Proposition~\ref{prop: equilibrium stability}}
\label{sec: Stability of internal equilibrium}
Recall that the reduced flow on $S^2_0$ is given by 
\begin{align*}
    \frac{d H}{d \tau} = -\frac{(1+H)^2}{Q(H;\alpha)}\bigg(1-\frac{\alpha}{s(H)}-\lambda Hs(H)\bigg)(\lambda Hs(H)-\beta),
\end{align*}
see Equation~\eqref{reduced flow}, where $\lambda Hs(H)-\beta =0$ for $H=H_I$ and $1-\frac{\alpha}{s(H)}-\lambda Hs(H)=0$ for $H\in\big\{H_{nC}^a,H_{nC}^r\big\}$. Moreover, it follows from Lemma~\ref{lemma: Stability of Critical Manifold} and Lemma~\ref{lem: S_0^2 intersections} that $H_{nC}^r<H_0<H_{nC}^a$ for $1-\lambda d^2>0$.
We first consider the relative positions of the points $e_{I}$, $e_{nC}^a$, and $e_{nC}^r$ in each $\alpha$-regime.
For $\alpha<\min\{\alpha^+, \alpha^\ast\}$, we have to show $H_{nC}^r<H_0<H_I<H_{nC}^a$. To infer that $H_0<H_I$ when $\alpha<\alpha^\ast$, we first note that $\alpha^\ast$ is defined so that $Q(H_I;\alpha^\ast)=0$. Then, $\frac{\partial}{\partial \alpha}Q(H_I; \alpha)=-\frac{1}{s(H_I)^2}<0$ implies that $\alpha^{\ast}$ is unique.
Similarly, by Lemma~\ref{lemma: Stability of Critical Manifold}, the definition of $H_0$ implies $Q(H_0;\alpha)=0$. From the definition of $Q(H;\alpha)$, we have 
\begin{align*}
    \lambda s(H_0)^2 \big[s(H_0)(1+H_0)^2+H_0\big]=\alpha<\alpha^\ast=\lambda s(H_I)^2\big[s(H_I)(1+H_I)^2+H_I\big].
\end{align*}
Due to the strict monotonicity of $\lambda s(H)^2 \big[s(H)(1+H)^2+H\big]$ in $H$, we conclude that $H_0<H_I$ when $\alpha<\alpha^\ast$. Note that, correspondingly, we have $H_I<H_0$ ($H_I=H_0$) for $\alpha^\ast<\alpha$ ($\alpha=\alpha^\ast$).

Given that $H_0<H_I$ when $\alpha<\alpha^\ast$, we will additionally show that $H_I<H_{nC}^a$ for $\alpha<\alpha^+$. To begin, note that $C_I=0$ when $\alpha=\alpha^+$, which implies that $H_I=H_{nC}^a$ or $H_I=H_{nC}^r$. However, for $\alpha<\alpha^\ast$, $H_{nC}^r<H_0<H_I$ and, therefore, $H_I=H_{nC}^a$ when $\alpha=\alpha^+$. It follows that $H_I<H_{nC}^a$ $(H_{nC}^a<H_I)$ for $\alpha<\alpha^+$ $(\alpha^+<\alpha)$.

To summarise, we have the following cases:
\begin{enumerate}
    \item For $\alpha<\text{min}\{\alpha^+, \alpha^\ast\}$, $H_{nC}^r<H_0<H_I<H_{nC}^a$.
    \item For $\alpha^+<\alpha<\alpha^\ast$, $H_{nC}^r<H_0<H_{nC}^a<H_{I}$.
    \item For $\alpha^\ast<\alpha<\alpha^+$, $H_{nC}^r<H_I<H_0<H_{nC}^a$.
\end{enumerate}
Next, we determine the stability of $e_I$ by calculating
\begin{align*}
\frac{\partial}{\partial H}\Big(\frac{\dd H}{\dd\tau}\Big)\bigg\lvert_{H=H_I} =-\frac{(1+H_I)^2}{Q(H_I;\alpha)}\bigg(1-\frac{\alpha}{s(H_I)}-\lambda H_Is(H_I)\bigg)\lambda\bigg[s(H_I)+\frac{H_I}{(1+H_I)^2}\bigg],
\end{align*}
where we have again used $\lambda H_Is(H_I)-\beta=0$.
The first bracketed factor in the numerator equals $C_I$, by definition, which is positive for $\alpha<\alpha^+$; similarly, the second bracketed factor is always positive. Hence, the stability of $e_I$ is determined by the sign of $Q(H_I;\alpha)$. Specifically, $e_I$ is attracting provided $Q(H_I;\alpha)>0$, i.e., for
\begin{align*}
    \lambda \big[s(H_I)(1+H_I)^2+H_I\big]-\frac{\alpha}{s(H_I)^2}>0, 
\end{align*}
which is equivalent to $\alpha < \alpha^\ast=\lambda s(H_I)^2\big[s(H_I)(1+H_I)^2+H_I\big]$, showing that $e_I$ is an attracting equilibrium for the reduced flow on $S_0^2$ when $\alpha<\text{min}\{\alpha^+, \alpha^\ast\}$. It follows that $e_{I}$ is repelling for $\alpha^\ast<\alpha$. Finally, for $\alpha=\alpha^\ast$, $H_I=H_0$ and $1-\frac{\alpha^\ast}{s(H_0)}-\lambda H_0 s(H_0)=0$. It therefore follows from L'H\^opital's Rule that $\frac{\partial}{\partial H}\big(\frac{\dd H}{\dd\tau}\big)\big\lvert_{H=H_0}=0$, implying a loss of hyperbolicity.

We now consider the stability of the remaining equilibria $H^\ast$ on $S_0^2$, where $H^\ast\in\big\{H_{nC}^a, H_{nC}^r\big\}$: we find
\begin{align*}
\frac{\partial}{\partial H}\Big(\frac{\dd H}{\dd\tau}\Big)\bigg\lvert_{H=H^\ast} &=-(1+H^\ast)^2\frac{-\frac{Q(H^\ast;\alpha)}{(1+H^\ast)^2}(\lambda H^\ast s(H^\ast)-\beta)}{Q(H^\ast;\alpha)} \\
&=\lambda H^\ast s(H^\ast)-\beta,
\end{align*}
as $1-\frac{\alpha}{s(H^\ast)}-\lambda H^\ast s(H^\ast)=0$. With $\lambda H s(H)-\beta =0$ for $H=H_I$, it follows by monotonicity that $\lambda H_{nC}^a s(H_{nC}^a)-\beta>0$ when $H_{nC}^a>H_I$, which is equivalent to $\alpha>\alpha^\ast$. Similarly, $\lambda H_{nC}^a s(H_{nC}^a)-\beta<0$ when $H_{nC}^a<H_I$, or $\alpha<\alpha^\ast$. For any $\alpha$, $H_{nC}^r< H_I$ and, therefore, $\lambda H_{nC}^r s(H_{nC}^r)-\beta< 0$, which completes the proof. 
\section{Proof of Lemma \ref{lem: desingularised system}}
\label{sec: desingularised system proof}
As $n_{0,\alpha}$ is given by a curve $\alpha(H)$, we can solve for the equilibria of \eqref{eq: Desingularised system}, which we will denote by $p_{FS}=(H_{FS},\alpha_{FS})$; these equilibria satisfy $Q(H_{FS},\alpha_{FS})=0$ and, hence,
\begin{align}
    \alpha_{FS} = \lambda s(H_{FS})^2\big[H_{FS}+s(H_{FS})(1+H_{FS})^2\big].
    \label{eq: alpha FS equation}
\end{align}
Substituting into the definition of $\Lambda$, one finds that $H_{FS}$ solves the following relation,
\begin{equation}
    F(H_{FS},r;\beta,\lambda,d):=\big[1-2\lambda H_{FS}s(H_{FS})-\lambda s(H_{FS})^2(1+H_{FS})^2\big](\lambda H_{FS}s(H_{FS})-\beta)+\frac{r}{s(H_{FS})} = 0.
    \label{eq: Folded sinularity steady state}
\end{equation}
We aim to apply the Implicit Function Theorem in a neighbourhood of the roots of \eqref{eq: Folded sinularity steady state} for $r=0$. For brevity, we write $F(H,0;\beta,\lambda,d) = u(H)v(H)$, with 
\begin{align*}
    u(H) &= 1-2\lambda Hs(H)-\lambda s(H)^2(1+H)^2\quad\text{and} \\
    v(H) &= \lambda H s(H)-\beta.
\end{align*}
Hence, $F(H,0;\beta,\lambda,d)$ is only zero when either $u=0$ or $v=0$. We see that for $H>0$, $u$ is monotonically decreasing in $H$, while $v$ is monotonically increasing, as
\begin{align*}
    u'(H) &= \frac{-2\lambda H}{(1+H)^2}-4\lambda s(H)-2\lambda s(H)^2(1+H)<0\quad\text{and} \\
    v'(H) &= \lambda\bigg[s(H)+\frac{H}{(1+H)^2}\bigg].
\end{align*}
Then, as $u(0)=1-\lambda d^2>0$ and $v(0)=-\beta<0$, and as $u(H)\to -\infty$ and $v(H)\to\infty$ for $H\to\infty$, by monotonicity there exist at most two roots to the expression $F(H,0;\beta,\lambda,d)=0$. We will denote these roots by $H_{FS1,0}$ and $H_{FS2,0}$, with $u(H_{FS1,0})=0$ and $v(H_{FS2,0})=0$, respectively. The two roots will generically be distinct, as only one of them is dependent on $\beta$. (Note that they are equal along the bifurcation curve $\mathcal{C}$; correspondingly, the Implicit Function Theorem is not applicable then.)

Next, we note that $\Pi(H)$ as given in Lemma~\ref{lem: S_0^2 intersections} is maximised at $\hat{H}$, which satisfies $\Pi'(\hat{H})=0$. A direct calculation shows that
\begin{align*}
    \Pi'(H) = \frac{1}{(1+H)^2}\big[1-2\lambda H s(H)-\lambda s(H)^2(1+H)^2\big] = \frac{1}{(1+H)^2}u(H).
\end{align*}

Then, $H_{FS1,0}$ corresponds to the value of $H$ where the saddle-node bifurcation occurs between $e^{a}_{nC}$ and $e^{r}_{nC}$ when $\alpha=\hat\alpha$; recall Lemma~\ref{lem: S_0^2 intersections}. In particular, $H_{FS1,0}= \hat{H}$, as defined there. Finally, $H_{FS2,0}=H_I$, by the definition of $v(H)$. The corresponding $\alpha_{FS}$-values may be obtained from $\eqref{eq: alpha FS equation}$. 

Now, calculating the derivative of $F$ with respect to $H$, we find 
\begin{align}
    \frac{\partial F}{\partial H} = u (H)v'(H) + u'(H)v(H) -\frac{r}{(1+H)^2}.
    \label{eq: F deriv x equation}
\end{align}
In particular, as $u'(H),v'(H)\neq 0$ for $H>0$, it follows that $\frac{\partial F}{\partial H}\rvert_{(H_{FSi},0)}\neq 0$ for $i=1,2$. Hence, applying the Implicit Function Theorem, we can conclude that there exist functions $H_i(r)$ ($i=1,2$) such that $F(H_i(r),r;\beta,\lambda,d)=0$ for $r$ sufficiently small. We denote these two roots (in $H$) by $H_{FS1,r}$ and $H_{FS2,r}$, which concludes the proof of item 1.

We now determine how the values of these roots change as $r$ is varied. To begin, we take the total derivative of $F(H_i(r),r)=0$ with respect to $r$:

\begin{align*}
    \frac{\partial F}{\partial H_i}(H_i(r),r)\frac{\dd H_i}{\dd r}+\frac{\partial F}{\partial r} (H_i(r),r)=0.
\end{align*}
(Here, we have suppressed the dependence of $F$ on $(\beta,\lambda,d)$ for simplicity of notation.)
Solving and evaluating at $r=0$, we obtain
\begin{align*}
    \frac{\dd H_i}{\dd r}\Big\rvert_{r=0} = -\frac{\frac{\partial F}{\partial r}(H_{FSi,0},0)}{\frac{\partial F}{\partial H_i} (H_{FSi,0},0)}
\end{align*}
Since $\frac{\partial F}{\partial r}(H_i(r),r) = \frac{1}{s(H_i(r))}>0$, the sign of $\frac{\dd H_i}{\dd r}\big\rvert_{r=0}$ will be determined by that of $\frac{\partial F}{\partial H_i} (H_{FSi,0},0)$. Evaluating \eqref{eq: F deriv x equation}, we have
\begin{align*}
    \frac{\partial F}{\partial H_i}(H_{FSi,0},0) = u(H_{FSi,0})v'(H_{FSi,0})+u'(H_{FSi,0})v(H_{FSi,0}).
\end{align*}
Without loss of generality, we first assume that $H_{FS1,0}<H_{FS2,0}$; then, evaluating at $i=1$, we find
\begin{align*}
    \frac{\dd H_1}{\dd r}\Big\rvert_{r=0} = -\frac{1}{s(H_{FS1,0})}\frac{1}{u'(H_{FS1,0})v(H_{FS1,0})},
\end{align*}
as $u(H_{FS1,0})=0$.
Then, $\frac{\dd H_1}{\dd r}\big\rvert_{r=0} <0$, as $u'(H)<0$ and $v(H_{FS1,0})<0$ if we assume that $H_{FS1,0}<H_{FS2,0}$. Similarly, $\frac{\dd H_2}{\dd r}\big\rvert_{r=0}>0$. (If the ordering of the roots is reversed, the result still holds, as $v(H_{FS1,0})>0$ in that case.)  

In sum, we have established that
\begin{align*}
    \min\{H_{FS1,r},H_{FS2,r}\} \downarrow\min\{H_{FS1,0},H_{FS2,0}\}
\end{align*}
as $r\to 0^+$, which shows item 2. Item 3 can be shown in a similar fashion.

Finally, for item 4, we note that $H_{FS1,0}=\hat{H}$ and $H_{FS2,0}=H_I$, as already shown, which concludes the proof.
\bibliography{bibliography}

\providecommand{\bysame}{\leavevmode\hbox to3em{\hrulefill}\thinspace}
\providecommand{\MR}{\relax\ifhmode\unskip\space\fi MR }
\providecommand{\MRhref}[2]{%
  \href{http://www.ams.org/mathscinet-getitem?mr=#1}{#2}
}
\providecommand{\href}[2]{#2}
\begin{thebibliography}{10}

\bibitem{ashwin_2012_tipping}
Peter Ashwin, Sebastian Wieczorek, Renato Vitolo, and Peter~Timothy Cox, \emph{Tipping points in open systems: bifurcation, noise-induced and rate-dependent examples in the climate system}, Philosophical Transactions of the Royal Society A \textbf{370} (2012), 1166--1184.

\bibitem{blackwood_2010_the}
Julie~C. Blackwood, Alan Hastings, and Peter~J. Mumby, \emph{The effect of fishing on hysteresis in caribbean coral reefs}, Theoretical Ecology \textbf{5} (2010), 105--114.

\bibitem{briggs_2018_macroalgae}
Cheryl~J. Briggs, Thomas~C. Adam, Sally~J. Holbrook, and Russell~J. Schmitt, \emph{Macroalgae size refuge from herbivory promotes alternative stable states on coral reefs}, PLOS ONE \textbf{13} (2018), e0202273.

\bibitem{de2016sector}
Peter De~Maesschalck, Ekaterina Kutafina, and Nikola Popovi{\'c}, \emph{Sector-delayed-hopf-type mixed-mode oscillations in a prototypical three-time-scale model}, Applied Mathematics and Computation \textbf{273} (2016), 337--352.

\bibitem{MR2916308}
Mathieu Desroches, John Guckenheimer, Bernd Krauskopf, Christian Kuehn, Hinke~M. Osinga, and Martin Wechselberger, \emph{Mixed-mode oscillations with multiple time scales}, SIAM Rev. \textbf{54} (2012), no.~2, 211--288. \MR{2916308}

\bibitem{fung_2011_alternative}
Tak Fung, Robert~M. Seymour, and Craig~R. Johnson, \emph{Alternative stable states and phase shifts in coral reefs under anthropogenic stress}, Ecology \textbf{92} (2011), 967--982.

\bibitem{gil2020}
Michael~A. Gil, Marissa~L. Baskett, Stephan~B. Munch, and Andrew~M. Hein, \emph{Fast behavioral feedbacks make ecosystems sensitive to pace and not just magnitude of anthropogenic environmental change}, Proceedings of the National Academy of Sciences \textbf{117} (2020), no.~41, 25580--25589.

\bibitem{gil_2017_social}
Michael~A. Gil and Andrew~M. Hein, \emph{Social interactions among grazing reef fish drive material flux in a coral reef ecosystem}, Proceedings of the National Academy of Sciences \textbf{114} (2017), no.~18, 4703--4708.

\bibitem{hughes_1994_catastrophes}
Terence~P. Hughes, \emph{Catastrophes, phase shifts, and large-scale degradation of a caribbean coral reef}, Science \textbf{265} (1994), 1547--1551.

\bibitem{MR2478556}
Arieh Iserles, \emph{A first course in the numerical analysis of differential equations}, second ed., Cambridge Texts in Applied Mathematics, Cambridge University Press, Cambridge, 2009. \MR{2478556}

\bibitem{6c8c3bfcb06e4fb8b0cd94ae49d0c15b}
Panagiotis Kaklamanos, Nikola Popovi{\'c}, and Kristian~Uldall Kristiansen, \emph{Bifurcations of mixed-mode oscillations in three-timescale systems: An extended prototypical example}, Chaos \textbf{32} (2022) (English).

\bibitem{fold}
Martin Krupa and Peter Szmolyan, \emph{Extending geometric singular perturbation theory to nonhyperbolic points---fold and canard points in two dimensions}, SIAM Journal on Mathematical Analysis \textbf{33} (2001), no.~2, 286--314.

\bibitem{MKrupa_2001}
\bysame, \emph{Extending slow manifolds near transcritical and pitchfork singularities}, Nonlinearity \textbf{14} (2001), no.~6, 1473.

\bibitem{MTSD}
Christian Kuehn, \emph{Multiple time scale dynamics}, Applied Mathematical Sciences, vol. 191, Springer, Cham, 2015.

\bibitem{mumby_2007_thresholds}
Peter~J. Mumby, Alan Hastings, and Helen~J. Edwards, \emph{Thresholds and the resilience of caribbean coral reefs}, Nature \textbf{450} (2007), 98--101.

\bibitem{doi:10.1137/S1064827594276424}
Lawrence~F. Shampine and Mark~W. Reichelt, \emph{The matlab ode suite}, SIAM Journal on Scientific Computing \textbf{18} (1997), no.~1, 1--22.

\bibitem{suding_2009_threshold}
Katharine~N. Suding and Richard~J. Hobbs, \emph{Threshold models in restoration and conservation: a developing framework}, Trends in Ecology \& Evolution \textbf{24} (2009), 271--279.

\bibitem{SZMOLYAN2001419}
Peter Szmolyan and Martin Wechselberger, \emph{Canards in $\mathbb{R}^3$}, Journal of Differential Equations \textbf{177} (2001), no.~2, 419--453.

\bibitem{TRUSCOTT1994981}
James~E. Truscott and John Brindley, \emph{Ocean plankton populations as excitable media}, Bulletin of Mathematical Biology \textbf{56} (1994), no.~5, 981--998.

\bibitem{Plankton-tipping}
Anna Vanselow, Lukas Halekotte, Pinaki Pal, Sebastian Wieczorek, and Ulrike Feudel, \emph{Rate-induced tipping can trigger plankton blooms}, Theoretical Ecology \textbf{17} (2024), 89--105.

\bibitem{vanselow_2019_when}
Anna Vanselow, Sebastian Wieczorek, and Ulrike Feudel, \emph{When very slow is too fast - collapse of a predator-prey system}, Journal of Theoretical Biology \textbf{479} (2019), 64--72.

\bibitem{2020SciPy-NMeth}
Pauli Virtanen, Ralf Gommers, Travis~E. Oliphant, Matt Haberland, Tyler Reddy, David Cournapeau, Evgeni Burovski, Pearu Peterson, Warren Weckesser, Jonathan Bright, St{\'e}fan~J. {van der Walt}, Matthew Brett, Joshua Wilson, K.~Jarrod Millman, Nikolay Mayorov, Andrew R.~J. Nelson, Eric Jones, Robert Kern, Eric Larson, C~J Carey, {\.I}lhan Polat, Yu~Feng, Eric~W. Moore, Jake {VanderPlas}, Denis Laxalde, Josef Perktold, Robert Cimrman, Ian Henriksen, E.~A. Quintero, Charles~R. Harris, Anne~M. Archibald, Ant{\^o}nio~H. Ribeiro, Fabian Pedregosa, Paul {van Mulbregt}, and {SciPy 1.0 Contributors}, \emph{{{SciPy} 1.0: Fundamental Algorithms for Scientific Computing in Python}}, Nature Methods \textbf{17} (2020), 261--272.

\bibitem{MR2136520}
Martin Wechselberger, \emph{Existence and bifurcation of canards in {$\mathbb{R}^3$} in the case of a folded node}, SIAM J. Appl. Dyn. Syst. \textbf{4} (2005), no.~1, 101--139. \MR{2136520}

\bibitem{Canards}
\bysame, \emph{{\`A{}} propos de canards ({A}propos canards)}, Trans. Amer. Math. Soc. \textbf{364} (2012), no.~6, 3289--3309. \MR{2888246}

\bibitem{wieczorek_2010_excitability}
Sebastian Wieczorek, Peter Ashwin, Catherine~M. Luke, and Peter~M. Cox, \emph{Excitability in ramped systems: the compost-bomb instability}, Proceedings of the Royal Society A: Mathematical, Physical and Engineering Sciences \textbf{467} (2011), 1243--1269.

\end{thebibliography}

\end{document}